\documentclass[12pt,reqno]{amsart}  
\usepackage{hyperref}

\usepackage{ifthen}
\newboolean{plots_included}
\setboolean{plots_included}{true}

\newboolean{alefonts}
\setboolean{alefonts}{true}

\ifthenelse{\boolean{alefonts}}{

\usepackage[full]{textcomp} 
\usepackage{newtxtext} 
\usepackage{cabin} 
\usepackage{zlmtt}
\usepackage[bigdelims,vvarbb]{newtxmath} 
\usepackage[cal=boondoxo]{mathalfa} 

\renewcommand{\Euler}{{\pmb \gamma}}
}{\newcommand{\Euler}{{\pmb \gamma}}}

\usepackage{float}
\usepackage{url}
\usepackage{epsfig, subfigure, afterpage}
\usepackage[dvipsnames, x11names]{xcolor}

\usepackage{enumitem}

\usepackage{mathtools} 
\usepackage{microtype} 
\usepackage[vmargin=2.2cm,hmargin=2.6cm]{geometry} 

\newtheoremstyle{sltheorems}
{10pt}
{6pt}
{\slshape}
{}
{\bfseries}
{.}
{.5em}
{\thmname{#1}\thmnumber{~#2}\thmnote{~(#3)}}

\theoremstyle{sltheorems} 
\newtheorem{Thm}{Theorem}
\newtheorem{CThm}{Classical Theorem}
\newtheorem*{Thm*}{Theorem}
\newtheorem{conj}{Conjecture}
\newtheorem{Defi}{Definition}
\newtheorem{prop}{Proposition}
\newtheorem{lem}{Lemma}
\newtheorem{cor}{Corollary}
\newtheorem{Rem}{Remark}

\newtheorem{Speculation}{Speculation}

\newcommand{\p}{\partial}
\newcommand{\C}{\mathbb{C}}

\newcommand{\cal}{\mathcal}

\makeatletter
\let\@@pmod\pmod
\DeclareRobustCommand{\pmod}{\@ifstar\@pmods\@@pmod}
\def\@pmods#1{\mkern4mu({\operator@font mod}\mkern 6mu#1)}
\makeatother

\newcommand{\be}{\begin{equation}}
\newcommand{\ee}{\end{equation}}

\newcommand{\bal}{\begin{align*}}
\newcommand{\ea}{\end{align*}}
\newcommand{\eal}{\ensuremath{\end{align*}}}
\newcommand{\bea}{\begin{eqnarray}}
\newcommand{\eea}{\end{eqnarray}}

\newcommand{\eps}{\ensuremath{\varepsilon}}

\renewcommand{\AA}{\mathcal{A}}
\renewcommand{\le}{\leqslant}
\renewcommand{\ge}{\geqslant}
\renewcommand{\leq}{\leqslant}
\renewcommand{\geq}{\geqslant}

\DeclareMathAlphabet{\curly}{U}{rsfs}{m}{n}

\allowdisplaybreaks
\numberwithin{Thm}{section}

\numberwithin{conj}{section}
\numberwithin{Defi}{section}
\numberwithin{prop}{section}
\numberwithin{prop}{section}
\numberwithin{lem}{section}
\numberwithin{cor}{section}
\numberwithin{Rem}{section}
\numberwithin{Quest}{section}
\numberwithin{Speculation}{section}

\usepackage{mathtools}

\usepackage{microtype} 
\usepackage{geometry}
\usepackage{caption} 
\makeatletter
\patchcmd{\env@cases}{1.2}{1}{}{}
\makeatother

\begin{document}
\title[Euler--Kronecker constants of maximal real cyclotomic fields]{Relative class numbers and Euler--Kronecker constants of
maximal real cyclotomic subfields}
\author[N. Kandhil, A. Languasco, P. Moree, S. Saad Eddin and A. Sedunova]{Neelam Kandhil, Alessandro Languasco, Pieter Moree, \\ Sumaia Saad Eddin 
and Alisa Sedunova}


\date{}

\begin{abstract}
\noindent
The Euler--Kronecker constant of a number field $K$ is the ratio of the constant and the residue of
the Laurent series of the Dedekind zeta function $\zeta_K(s)$ at $s=1$.
We study the distribution of the Euler--Kronecker constant $\gamma_q^+$ of the maximal real subfield of  
$\mathbb Q(\zeta_q)$ as $q$ ranges over the primes. Further, we 
consider the distribution of $\gamma_q^+-\gamma_q$, with $\gamma_q$ the Euler--Kronecker 
constant of  $\mathbb Q(\zeta_q)$ 
and show how it is connected with Kummer's conjecture, 
which predicts the asymptotic growth of the 
relative class number of $\mathbb Q(\zeta_q)$. We improve, for example, the known results on the bounds on average
for the Kummer ratio and we prove analogous sharp bounds
for $\gamma_q^+-\gamma_q$. The methods employed are partly inspired by those used by 
Granville (1990) and Croot and Granville (2002) to investigate Kummer's conjecture. 
We supplement our theoretical findings with numerical illustrations to reinforce our conclusions.
\end{abstract}

\maketitle

\setcounter{tocdepth}{1}
\ifthenelse{\boolean{plots_included}}{}{\mbox{}\vspace{-1cm}\tableofcontents}

\section{Introduction}

Let $K$ be a number field, ${\cal O}$ its ring of integers and $s$ a complex variable.
Then for $\Re(s)>1$ the 
\emph{Dedekind zeta function} is defined as
\begin{equation*}
\zeta_K(s)=\sum_{\mathfrak{a}} \frac{1}{N{\mathfrak{a}}^{s}}
=\prod_{\mathfrak{p}}\frac{1}{1-N{\mathfrak{p}}^{-s}},
\end{equation*}
where $\mathfrak{a}$ ranges over 
the non-zero ideals in ${\cal O}$, 
$\mathfrak{p}$ ranges over the prime ideals in ${\cal O}$, and $N{\mathfrak{a}}$ denotes the 
\emph{absolute norm} of $\mathfrak{a}$,
that is the index of $\mathfrak{a}$ in $\cal O$.
It is known that $\zeta_K(s)$ can be analytically continued to $\C \setminus \{1\}$,
and that it has a simple pole at $s=1$. 
Observe that $\zeta_{\mathbb{Q}}(s)$ equals $\zeta(s)$, the Riemann zeta function.
Further, around $s=1$ we have the Laurent expansion
\begin{equation*}
\zeta_K(s)=\frac{\alpha_K}{s-1}+c_K+c_1(K)(s-1)+c_2(K)(s-1)^2+\ldots,
\end{equation*}
where $\alpha_K$ is a positive real number carrying a lot of arithmetic information
on $K$.
The constants $c_{\mathbb Q}$ and $c_j(\mathbb Q)$  are 
called \emph{Stieltjes constants}. 
In particular, 
we have $c_{\mathbb Q}=\Euler =0.57721566\dotsc$, 
the \emph{Euler--Mascheroni constant}, see, e.g., Lagarias \cite{Lagarias} for a  survey of 
related material. 
The ratio $\gamma_K=c_K/\alpha_K$ is called the \emph{Euler--Kronecker constant} (recall that the 
Kronecker
limit formula expresses $\gamma_K$ in terms of special values of the Dedekind 
$\eta$-function in case $K$ is imaginary quadratic). In 1903, Landau \cite{Landaugamma} showed that
$$
\gamma_K = \lim_ {x \to \infty} \Bigl( \frac{1}{\alpha_K} \sum_{N \mathfrak{a} \leq x}  \frac{1}{N \mathfrak{a}} - \log x \Bigr).
$$
Furthermore, we have, see, e.g., Hashimoto et al.\,\cite{HIKW},
\begin{equation}
\label{hickup}
\gamma_K=\lim_{x\to  \infty}\Bigl(\log x - \sum_{N\mathfrak{p}\le x}\frac{\log N\mathfrak{p}}{N\mathfrak{p}-1}\Bigr),
\end{equation}
making clear that the existence of (many) 
prime ideals in ${\cal O}_K$ of small norm has a decreasing effect on $\gamma_K$.
Yet another expression for $\gamma_K$ is 
\begin{equation*}
\gamma_K=\lim_{s \to 1^+}\,\Bigl(\frac{\zeta'_K(s)}{\zeta_K(s)}+\frac{1}{s-1}\Bigr),
\end{equation*}
which shows that $\gamma_K$ is the constant part in the Laurent series of the logarithmic derivative
of $\zeta_K(s)$. 
By defining $\Lambda_K$ via $-\zeta_K'(s)/\zeta_K(s) = \sum_{n=1}^\infty \Lambda_K(n)/n^s$, Dixit and Murty \cite{DixitM2023} proved that 
\[
1-\gamma_K = \int_{1}^\infty \Bigl( \sum_{n \le x}\Lambda_K(n)-x \Bigr) \frac{dx}{x^2},
\]
establishing  a connection between 
$\gamma_K$ and the \emph{prime ideal theorem} in the form $$\sum_{n \le x}\Lambda_K(n)\sim x.$$ 
\par Ihara \cite[Theorem~1 and Proposition~3]{I} proved that under
the Generalized Riemann Hypothesis (GRH) for Dedekind zeta functions
there are absolute constants
$c_1,c_2>0$ such that
\begin{equation}
\label{iha}
-c_1\log d_K \le \gamma_K \le c_2\log \log d_K,
\end{equation}
where 
$d_K$ is the absolute value of the discriminant $K/\mathbb Q$.

\par In case $q$ is an odd prime and $K=\mathbb Q(\zeta_q)$ 
a \emph{prime cyclotomic field}, there is an extensive literature on the distribution 
of $\gamma_q=\gamma_{\mathbb Q(\zeta_q)}$ and $\gamma_q^+=\gamma_{\mathbb Q^{+}(\zeta_q)}$, with
${\mathbb Q}^+(\zeta_q)={\mathbb Q}(\zeta_q+\zeta_q^{-1})$ being the \emph{maximal 
real cyclotomic field} (from now on the notation $q$ is exclusively used for odd primes).
We provide a brief survey in the next subsection. 

\subsection{The prime cyclotomic field and its maximal real subfield}

\noindent In the case of a prime cyclotomic field $\mathbb Q(\zeta_q)$ we have
\begin{equation}
\label{oud}
\zeta_{\mathbb Q(\zeta_q)}(s)=\zeta(s)\prod_{\chi\ne \chi_0}L(s,\chi),
\end{equation}
where $\chi$ runs over the non-principal characters modulo $q$,
leading by logarithmic differentiation to
\begin{equation}
\label{gammaq}
\gamma_q=\Euler+\sum_{\chi\ne \chi_0}\frac{L'(1,\chi)}{L(1,\chi)}.
\end{equation}
Thus the behaviour of $\gamma_q$ is related to that of $L(s,\chi)$ and $L'(s,\chi)$
at $s=1$. (Here and in the rest of the paper we often use the fundamental
fact that $L(1,\chi)\ne 0$.)
\par From (\ref{iha}) and
the fact that the discriminant of 
$\mathbb Q(\zeta_q)$ is $(-1)^{(q-1)/2}q^{q-2}$, we infer that, under GRH, $-c_1q\log q\le \gamma_q\le c_2\log (q\log q)$. 
Assuming GRH
for Dedekind zeta functions of cyclotomic fields, Ihara, Murty and Shimura \cite{IMS} proved that $\gamma_q \ll (\log q)^2$.
In 2010, Badzyan \cite{ABD} improved this result and proved that, under GRH, $\gamma_q \ll (\log q)(\log \log q).$
In 2011, Kumar Murty \cite{KM} proved that 
$\gamma_q\ll \log q$ 
is unconditionally true on average. We will now formulate this result more precisely, using the notation 
$\pi^*(Q)$
for the difference $\pi(Q)-\pi(Q/2)$, where
$\pi(x)=\sum_{p\le x}1$ is the \emph{prime counting function} and here and in the sequel the notation $p$ is exclusively used to indicate primes.

\begin{Thm*}{\rm (Kumar Murty \cite{KM})}
We have
$$\sum_{Q/2 < q \le Q}|\gamma_q| \ll  \pi^{*}(Q) \log{Q},$$
where the sum is over prime numbers $q$ in the interval $(Q/2, Q]$.
\end{Thm*}
\par For the maximal real subfield of the cyclotomic field one has analogously
to \eqref{oud}
the factorization (see \cite[ch.~4]{Wbook})
\begin{equation}
\label{maxreal}
\zeta_{\mathbb Q(\zeta_q)^+}(s)=\zeta(s)\prod_{\substack{\chi \neq \chi_0 \\ \chi(-1)=1}} L(s,\chi),
\end{equation}
and hence
\begin{equation}
\label{gammaplus}
\gamma_q^+=\Euler+\sum_{\substack{\chi \neq \chi_0 \\ \chi(-1)=1}}\frac{L'(1,\chi)}{L(1,\chi)}.
\end{equation}
\par {}From (\ref{iha}) and
the fact that the discriminant of 
$\mathbb Q^+(\zeta_q)$ is $q^{(q-3)/2}$, we infer that, under GRH, $-c_1q\log q\le \gamma_q^+\le c_2\log (q\log q)$. 
This was sharpened by Ihara et al.~\cite{IMS}, who showed that under GRH, we have $\gamma_q^+\ll (\log q)^2.$
Unconditionally, they proved that $\gamma_q^+\ll q^{\varepsilon}$ for any $\varepsilon >0$. 
Ihara \cite[Theorem~1]{I} showed that under GRH we have $\gamma_q^+\le (2+o(1))\log q$.
Unconditionally, we show that on average also $\gamma_q^+\ll \log q$. Our proof follows the approach of Kumar Murty.

\begin{Thm}
\label{VK1}
We have 
\[
\sum_{Q/2 < q \le Q}|\gamma_q^+|\ll \pi^*(Q)\log Q.
\]
\end{Thm}

{F}rom this result and that of Kumar Murty we immediately obtain that
 \[
\sum_{Q/2 < q \le Q}|\gamma_q-\gamma_q^+|\ll \pi^*(Q)\log Q.
\]  
We prove in Theorem \ref{kappaqzeroEH2} that,
assuming a strong form of Elliott--Halberstam conjecture \ref{EH2}, this upper 
bound can be improved to $o(\pi^*(Q)\log Q)$, as $Q$ tends to infinity.

\subsection{The difference $\gamma_q^+-\gamma_q$ and Siegel zeros}
Let $K \ne \mathbb{Q}$ be an algebraic number field having $d_K$ as its absolute discriminant over the rational numbers. Then, see
Stark \cite[Lemma~3]{stark}, $\zeta_K(s)$ has at most one zero in the region in the complex plane determined by
$$
 \Re(s) 
\geq
1 - \frac{1}{4 \log d_K},
\quad \quad
|\Im(s)| \leq \frac{1}{4 \log d_K}.
$$
If such a zero exists, it is
real, simple and often called 
\emph{Siegel zero}. 
Dixit and Murty  \cite{DixitM2023}\footnote{There is an oversight in \cite[Theorem 1.2]{DixitM2023}. 
The Siegel zero contribution is without the 1/2 factor, see 
\cite{Dixit-errata}.}
proved that if $\zeta_K(s)$ has a Siegel zero $\beta_0$,  
then
\begin{equation}
\label{Dixiland}
\gamma_K = \frac{1}{\beta_0(1-\beta_0)}+O(\log d_K),
\end{equation}
and 
$\gamma_K=O(\log d_K)$ otherwise.
Note that \eqref{Dixiland} can be alternatively rewritten as
$$
\gamma_K=\frac{1}{1-\beta_0}+O(\log d_K).
$$
It follows from this
estimate that
$$\gamma^+_q - \gamma_q =\frac{
\delta_{q}}{\beta_0-1}+O(q\log q),$$ 
where $\delta_{q}$ is $1$ if one of the odd Dirichlet $L$-functions of conductor $q$ 
has a Siegel zero $\beta_0$, and is zero otherwise.

The following theorem greatly improves on this result.
We will state it using a normalized version of $\gamma_q^+-\gamma_q$,
since this quantity will play an important role into the subsequent part 
of our work.
{F}rom now on we will denote
\begin{equation}
\label{gq-def}
\kappa(q)=\frac{\gamma_q^+-\gamma_q}{\log q}
\end{equation}
and use the notation $\log_k x$ to indicate
the $k$ times iterated logarithm: 
$\log_k x= \underbracket{\log\dotsm \log}_{k\, \textrm{times}}(x).$

\begin{Thm}
\label{gq-direct}
Given any function $\ell(q)$ tending  monotonically 
to infinity with $q$, there is an
effectively computable prime
$q_0$ (possibly depending on $\ell$)  such that 
\[
\Big|\kappa(q) + \frac{\delta_q}{\beta_0(1-\beta_0) \log q}\Big| 
<  
\log_2 q 
+ ( \delta_{q}+1) \ell(q) + \log \ell(q)
-0.59,
\]
where $\delta_q = 1$ if there exists a Siegel zero $\beta_0$ among the family 
of Dirichlet $L$-series $L(s, \chi)$, with $\chi$ any odd character modulo $q$,
and is zero otherwise.
In addition, there is an effectively computable prime $q_1$ such that if we assume
that the Riemann Hypothesis holds  for every Dirichlet $L$-series $L(s,\chi)$ with $\chi$ an 
odd character modulo $q$ for  $q\ge q_1$,  then
$$
|\kappa(q) | <  \log_2 q + 1.41.
$$
\end{Thm}

\begin{Rem}
The values $-0.59$ and $1.41$ above can be sharpened to $-0.61$, respectively  $1.39$,
by arguing as in Remark \ref{sharp-constant-2} below. Speculation \ref{spec2} 
gives what are believed to be the best possible asymptotical upper and lower bounds
for $\kappa(q)$.
\end{Rem}
The reader might be puzzled by the  appearance of $\beta_0(1-\beta_0)$, rather than
$1-\beta_0$ in Theorem \ref{gq-direct}, but should keep in mind that the contribution of $1/\beta_0$ is negligible.
However, it helps to simplify the analysis of the upper bound.

\subsection{Kummer's conjecture} 
Let $h_1(q)$ be the ratio of the class number $h(q)$ of 
$\mathbb Q(\zeta_q)$ and the class number 
of its maximal real subfield $\mathbb Q(\zeta_q+\zeta_q^{-1})$. 
Kummer proved that this is an integer. 
It is now called the \emph{relative class number}
(or sometimes \emph{first factor} of the class number $h(q)$). This quantity played an important role in 
Kummer's research on Fermat's Last Theorem. 
\begin{Defi}
Let $q$ be a prime number,
\begin{equation*}
G(q)=2q\Bigl(\frac{q}{4\pi^2}\Bigr)^\frac{q-1}{4}, \quad
R(q)=
\frac{h_1(q)}{G(q)} \quad and \quad r(q)=\log R(q).
\end{equation*}
\end{Defi}
The ratio $R(q)$ is  called the 
 \emph{Kummer ratio}. If no confusion can arise we
 will use the same terminology for $r(q)$.
In 1851, Kummer \cite{Kummer1851} 
made the following conjecture.
\begin{conj}
As $q$ tends to infinity,
$R(q)$ tends to $1$.
\end{conj}

\par The relative class number $h_1(q)$
is related to special values of Dirichlet $L$-series, namely, 
Hasse \cite{Hasse} showed that
\begin{equation}
\label{hasse}
R(q)=\prod_{\chi(-1)=-1}L(1,\chi),
\end{equation}
where the product is over all the odd characters modulo $q$.
It follows from this, \eqref{oud} and \eqref{maxreal} that
\[
R(q)=\lim_{s \to 1^+}\frac{\zeta_{\mathbb Q(\zeta_q)}(s)}
{\zeta_{\mathbb Q(\zeta_q)^+}(s)}.
\]
Indeed, it is not difficult to see
that
\begin{equation*}
\frac{\zeta_{\mathbb Q(\zeta_q)}(s)}
{\zeta_{\mathbb Q(\zeta_q)^+}(s)}=
R(q) \bigl( 1+(\gamma_q-\gamma_q^+)(s-1)+O_q((s-1)^2) \bigr).
\end{equation*}
For more details on Kummer's conjecture, see \cite{paper1-rq}. 
The following result can be regarded as an analogue of Theorem \ref{VK1}.
\begin{Thm}
\label{alerq}
We have 
\[
\sum_{Q/2 < q \le Q} \vert r(q) \vert \ll \pi^*(Q).
\]
\end{Thm}

This result says that $r(q)$ is absolutely bounded on average. 
For explicit estimates of the product $\prod_{\chi \ne \chi_0} L(1,\chi)$ and the Kummer ratio,
we refer the interested reader to  \cite{Full-K-ratio, paper1-rq}. 
Here, $\chi$ varies over \emph{all} the non-principal characters modulo $q$.

\subsection{Kummer's conjecture and the distribution of $\kappa(q)$} 
Kummer's conjecture was investigated by various authors \cite{CG,Gr,MP}. 
Some of the formulae appearing  in their work have a striking similarity with 
ones that occur in the study of Euler--Kronecker  constants. 
Exploiting this, Ford et al.\,\cite{FLM} could  translate 
the work of Granville \cite{Gr} to obtain distributional results on $\gamma_q/\log q$.
In 2018, Moree \cite{Msurvey} suggested that $\kappa(q)$,
defined in \eqref{gq-def},
should be even more analogous to the Kummer ratio. We will show that this is indeed the 
case by applying the results 
in \cite{CG, Gr} to $\kappa(q)$. 

In particular, one could wonder about the truth of the following conjecture.
\begin{conj}[Euler--Kronecker constant analogue of Kummer's conjecture]
\label{equivcon}
We have $\kappa(q)=o(1)$ as $q\to  \infty$.
\end{conj}

We recall that Granville \cite[Theorem~1]{Gr} showed that Kummer's conjecture is false under
the assumption of both the
Elliott--Halberstam conjecture and the Hardy--Littlewood conjecture for
Sophie Germain primes (see 
\S\ref{sec:prelim}). In Theorem 
\ref{gqrange} we will show that the same happens for Conjecture \ref{equivcon}; 
for example, we will see that $\kappa(q)$ approaches the value $1/4$ when
$q$ runs over the sequence of Sophie Germain primes (assuming, clearly, a suitable
conjecture about their ubiquity within the set of prime numbers).
In Theorem \ref{noEHproof}, we will show 
that Conjecture \ref{equivcon} is false 
assuming a stronger form of the Hardy--Littlewood conjecture \ref{HLconjecture}
(without having to rely on the Elliott--Halberstam conjecture), 
see Section \ref{sec:CG-analogues} for more details.

In fact, we will prove a series of results about $\kappa(q)$ 
showing that it has a striking similar behavior with $r(q)$. 
One of them is the following (for the proof
see Section \ref{sec:CG-analogues}).

\begin{Thm}\label{noEHproof}
Suppose there exists a set $M$ of $k$ distinct positive integers with 
$ \Sigma_{m \in M} 1/m  > 6$, such that there are 
$\gg x/(\log x)^{k+1}$ primes $q\leq x$ for which 
$mq \pm 1$ is also prime for every $m \in M$ (see Theorem \ref{unbounded}
and Conjecture \ref{HLconjecture} 
for the existence of such sets). Then 
there exists an $\epsilon>0$ such that $|\kappa(q)|>\epsilon$
for $\gg x/(\log x)^{k+1}$ primes $q\leq x$, as $x$ tends to infinity.
\end{Thm}

This result is an analogue of \cite[Theorem 1.9]{CG} 
(where the authors stated the analogous result, but for $mq+1$ only) concerning $r(q)$, but our version 
requires a stronger assumption, namely that 
$ \sum_{m\in M} 1/m  > 6$, while
Croot and Granville just needed $ \sum_{m\in M} 1/m > 4$.
This difference
is due to technical reasons depending of the different
set of weights used in the prime sums involved into the 
definitions of $\kappa(q)$ and $r(q)$,
see Remark \ref{Remark-prop-53} for more insights.

Ihara \cite{I} conjectured that always $\gamma_q>0$. 
Ford et al.~\cite{FLM} describe a set of 
admissible integers $M$ of cardinality $2089$ such
that $ \sum_{m\in M} 1/m  > 2$ and apply it to show that
under Conjecture \ref{HLconjecture} below, there are $\gg x/(\log x)^{2090}$ primes $q\le x$
for which $\gamma_q<0$. It seems likely that the set of primes $q$ with
$\gamma_q<0$ is very thin. Meanwhile several explicit 
such $q$ are known.

Recalling that $q$ is odd and hence $m$ must be even, and remarking
that $\sum_{n \le 12367} 1/(2n) = 6.0000215 \dotsc > 6$,
we can infer that 
$\#M \ge 12367$. For comparison,
we observe that $\sum_{n \le 227} 1/(2n) = 4.0021833 \dotsc > 4$.

The next result, together with \cite[Theorem~1.3]{CG}, also shows that $\kappa(q)$ and $r(q)$ behave in a similar way.

\begin{Thm}
\label{cr1}
Assume that the Elliott--Halberstam conjecture \ref{EHconjecture} is true.
\begin{enumerate}[label={(\roman*)}, wide, nosep, after=\vspace{4pt}]
\item Let $\varpi$ be a rational number. Then, under
the asymptotic version of the
Hardy--Littlewood conjecture \footnote{Inspecting the proof
of \cite[Theorem~1.3]{CG}, it is clear that
the Hardy--Littlewood conjecture \ref{HLconjecture-asymp} is only required for the case
in which $\varpi$ is rational.}, 
\ref{HLconjecture-asymp}, there exists an integer
$B(\varpi) \geq 1$ and a constant $C_\varpi > 0$, such that
there is an infinite set $Q_\varpi$ containing 
$\sim C_\varpi x/ (\log x)^{B(\varpi)}$ primes $q\leq x$
for which 
\begin{equation}
\label{Qomega}
\lim_{\substack{q \to \infty \\
q \in Q_{\varpi}}} \kappa(q) = \varpi.
\end{equation}
\item 
Let $A > 0$ be arbitrary. Let $\varpi$ be an irrational 
number. Then, any set of prime numbers $Q_{\varpi}$ 
for which \eqref{Qomega} holds, contains $\ll_{\varpi} \frac{x}{(\log x)^A}$
primes $q\le x$.
\end{enumerate}
\end{Thm}

Also the proof of Theorem \ref{cr1} will be presented in Section \ref{sec:CG-analogues}.

The last result we present in this Introduction (and is proved in
Section \ref{sec:CG-analogues} too), shows that the
distribution of $\kappa(q)$, just as that of $r(q)$, is closely 
connected with the existence of
infinitely many $q$ for which also $mq+b$ is prime (with $b\in \{- 1, 1\}$ fixed and $m$ an integer).

\begin{Thm}
\label{cr5}
Let $b \in \{-1, 1\}$.
If there exists $A\ge 1$ and $\epsilon >0$ such that
$b \kappa(q) > 3 +\epsilon$ for $\gg x/ (\log x)^A $ primes $ q \leq x$ for all
sufficiently large $x$,
then there exists an integer $m$ for which there are
$\gg x/ (\log x)^A $ primes $q \leq x$ such that  
$mq +b$ is also prime.
\end{Thm}

The inequality $b \kappa(q) > 3 +\epsilon$ in this result
is the analogue of $b r(q) > 2 +\epsilon$ in \cite[Theorem~1.8]{CG},
and the reason for the larger constant in the former case,
is the same as the one mentioned after Theorem \ref{noEHproof};
in this case too we refer to  Remark \ref{Remark-prop-53} for more details.

\subsection{Outline of our work.} 
The paper is organized as follows. In Section \ref{sec:prelim}
we will define the quantities and state the main theorems we will use
in this work. Moreover, we will also state the  conjectures about the distribution
of prime numbers we will need to prove some of our results.
Section \ref{sec:meangamma_plus} is dedicated to prove Theorem~\ref{VK1},
while in Section \ref{sec:Kummer} we will discuss about Kummer's conjecture
and present the proof of Theorem \ref{alerq}.

In Section \ref{sec:kappaq-initiol} we will start our discussion about $\kappa(q)$
and we will prove Theorem \ref{gq-direct}.
In Section \ref{sec:CG-analogues} we will present several
other results about $\kappa(q)$ and its similarity with $r(q)$.
In Section \ref{sec:spec}, we will show that a suitable heuristic
about the distribution of prime numbers in arithmetic progressions shows
that $\kappa(q) \sim r(q)$ for $q$ tending to infinity.
Moreover, we will prove that, under the assumption of a strong
version of the Elliott--Halberstam conjecture \ref{EH2}, that
$ \sum_{Q/2 < q \leq Q} \vert \kappa(q)\vert =o(\pi^*(Q))$, as $Q$ tends to infinity.

In Table \ref{table1} on p.\pageref{table1} we summarise the main
theorems and the conjectures used in the proofs of our results; in Table \ref{table2}
on p.\pageref{table2} we present the values of $\kappa(q)$ for every odd prime up to $1000$.

Section \ref{sec:numerical} 
presents  some relevant numerical data and graphical representations. 
They are consistent with the speculations
and conjectures about our results we make.
For example, the comparison of the histograms of $r(q)$ and $\kappa(q)$
provides, in our opinion, pretty compelling 
evidence that these quantities
should have the same asymptotic behaviour. Also, the presence of 
secondary ``spikes'' around $1/4$ and $1/8$ 
demonstrates in a beautiful way the contributions of the primes $q$ for which 
$2q\pm1$, or $4q\pm1$, are prime too.

\section{Preliminaries}
\label{sec:prelim}
\subsection{Prime number distribution}
In this section, we recall the material we need on the distribution of prime numbers, using the notations

\begin{equation*}
\pi(t) = \sum_{p \le t} 1, \quad \quad
\pi(t;d,a) = \sum_{\substack{p \le t \\ p \equiv a  \pmod*{d}}} 1,
\end{equation*}
\begin{equation*}
\theta(t) = \sum_{p \le t} \log p, \quad \quad
\theta(t;d,a) = \sum_{\substack{p \le t \\ p \equiv a  \pmod*{d}}} \log p,
\end{equation*}
and
\begin{equation*}
\psi(t) = \sum_{n \le t} \Lambda(n), \quad \quad
\psi(t;d,a) = \sum_{\substack{n \le t \\ n \equiv a  \pmod*{d}}} \Lambda(n),   
\end{equation*}
where $\Lambda$ denotes the von Mangoldt function.

For fixed coprime integers $a$ and $d$, we have asymptotic equidistribution: 
$$
\psi(t;d,a)\sim \psi(t)/\varphi(d)
\quad
\textrm{and}
\quad
\theta(t;d,a)\sim \theta(t)/\varphi(d),
\quad
$$
as $t\to \infty$.
The following result concerns the accuracy of the first approximation. 
For every  $A > 0$, we have
that
\begin{equation} 
\label{BV-type-ineq}
\sum_{d \le {\mathcal Q}}
\max_{t \le u} 
\max_{(a,d)=1}
\Bigl\vert 
\psi(t;d,a)-\frac{\psi(t)}{\varphi(d)}
\Bigr \vert 
\ll 
\frac{u}{(\log u)^{A}},
\end{equation}
where ${\mathcal Q}={\mathcal Q}(u)$ is a suitable function,
and the implicit constant may depend on $A$ and ${\mathcal Q}$.
Estimate \eqref{BV-type-ineq},  with 
${\mathcal Q}(u) = \sqrt{u}/(\log u)^B$, $B = B(A)>0$,
was independently proved by Bombieri 
and A.I. Vinogradov in 1965, see \cite[\S9.2] {CojocaruM2006}.
The very same statement for 
$t=u$ and ${\mathcal Q}(u) = u^{1-\epsilon}$ with $0<\epsilon<1$,
is unproved yet and it is commonly called
the \emph{Elliott--Halberstam conjecture}, see \cite{CojocaruM2006,ElliottH1968/69}, which we state below.
A statement equivalent to \eqref{BV-type-ineq} with the $\psi(t;d,a), \psi(t)$-functions 
replaced by the $\pi(t;d,a), \pi(t)$-functions, or the $\theta(t;d,a), \theta(t)$-ones, 
can be easily obtained. 

\begin{conj}
[Elliott--Halberstam conjecture, general case]
For every $\epsilon>0$ and $A>0$,
\[
\sum_{d\le u^{1-\epsilon}} \max_{(a,d)=1}\Bigl \vert \pi (u;d,a)-{\frac {\pi (u)}{\varphi (d)}}\Bigr \vert  \ll_{A,\epsilon} \frac{u}{(\log u)^A}.
\]
\end{conj} 
In this paper we will actually need a slightly different form involving
just the cases where
$a\in\{-1,1\}$ and $d=q$ is a prime:
\begin{conj}
[Elliott--Halberstam conjecture]
\label{EHconjecture}
For every $\epsilon>0$ and $A>0$,
\[
\sum_{q\le u^{1-\epsilon}} \Bigl \vert \pi (u;q,1)- \pi(u;q,-1)\Bigr \vert  \ll_{A,\epsilon} \frac{u}{(\log u)^A}.
\]
\end{conj}

In Section \ref{sec:spec}, we will prove that that the absolute value of 
$\kappa(q)$ tends to 
zero on average assuming, as in  \cite[\S10]{Gr},
that the previous conjecture holds in a wider range for $q$, namely

\begin{conj}[Strong form of the Elliott--Halberstam conjecture]
\label{EH2}
We have
$$
\sum_{u < q \leq 2u} \bigl\vert \pi(t;q,1) - \pi(t;q,-1) \bigr\vert \ll \frac{t}{(\log t)^3}
$$
uniformly for any 
$t \geq u \exp\bigl(\frac{\log u}{\sqrt{\log_2 u}}\bigr).$
\end{conj}

We will also use the Riemann Hypothesis (RH$_{\textrm{odd}}(q)$) for the Dirichlet $L$-series
attached to odd Dirichlet characters.
\begin{conj}[RH$_{\textrm{odd}}(q)$]
\label{RH-odd-conjecture}
Let $q$ be an odd prime.
The non-trivial zeros of the Dirichlet $L$-series $L(s,\chi)$, 
where $\chi$ runs over the set of the odd Dirichlet character$\pmod*{q}$, 
are on the line $\Re(s)=1/2$.
\end{conj}

We will also need the following prime pair estimate, which
is a type of result often encountered in sieve theory.
\begin{lem}
\label{lem1}
Let $a, t$ be fixed non-zero coprime 
integers with $t\ge 1$. The number of prime pairs $p,q \leq u$ with $p-a=tq$ is  
\[
\ll \frac{u}{\varphi(t) (\log (u/t))^2}\frac{|a|}{\varphi(|a|)}
\ll \frac{u\log_2 (|a|+1)}{\varphi(t) (\log (u/t))^2},
\]
where the implicit constants are absolute.
\end{lem}

\begin{proof}
The first estimate follows for example from \cite[Satz~4.2, p.~45]{Prachar}
with $a_1m+b_1=m$ and $a_2m+b_2=tm+a$ 
and from Examples \,9.4.6-9.4.8 in \cite{murtyy}.
The final
estimate is a consequence of the bound $\varphi(n) \gg n/ \log \log n$ 
(we refer to \cite[Theorem~13.14]{apos} for details).
\end{proof}

An important tool we will use is the following theorem. 
\begin{CThm}[Brun--Titchmarsh theorem]
\label{BT-thm}
Let $x,y>0$ and $a,d$ be positive integers such that $(a,d)=1$.
Then, uniformly for all $y>d$, we have
\begin{equation}
\label{BT-estim}
\pi(x+y;d,a) - \pi(x;d,a) < \frac{2y}{\varphi(d) \log(y/d)}.
\end{equation}
\end{CThm}
For a proof, see, e.g., Montgomery--Vaughan  \cite[Theorem~2]{MVsieve}.

\subsection{Admissible sets of large measure}
\par Let $\mathcal A=\{a_1,\ldots,a_s\}$ be a set
of $s$ distinct natural numbers.
We define the measure
\[
\mu(\mathcal A)=\sum_{k=1}^s\frac{1}{a_k}.
\]
Given a prime $p$, let 
$\omega(p)$ denote the number of solutions modulo 
$p$ of the equation
\begin{equation}
\label{omegaeq}
X\prod_{i=1}^s(a_iX+1)\equiv 0 \pmod*{p}.
\end{equation}
A set $\mathcal A$ is said to be \emph{admissible} if 
$\omega(p)<p$ for every prime $p$. As $\omega(p)\le s+1$,
we see that $\mathcal A$ is admissible, if and only if 
$\omega(p)<p$ for every prime $p\le s+1$. 
We observe that if we change $a_iX+1$ by
$a_iX-1$ in \eqref{omegaeq}, the number of solutions
is also still $\omega(p)$. 
\par The admissible sets
are relevant for determining which sets of linear forms can
(presumably) be infinitely often all simultaneously prime.

\begin{conj}[Hardy--Littlewood conjecture, lower bound version \cite{HardyL1923}]
\label{HLconjecture}
Suppose $\AA = \{a_1,\ldots,a_s\}$ is an admissible set.
Choose $b\in \{-1,1\}$.
Then the number of integers $n\le x$ such that the integers 
$n,a_1n+b,\ldots,a_sn+b$ are all prime is of cardinality 
$\gg_{\mathcal A} x/(\log x)^{s+1}$.
\end{conj}

Actually, the full Hardy--Littlewood conjecture gives an asymptotic, 
rather than a lower bound:
\begin{conj}[Hardy--Littlewood conjecture, asymptotic version \cite{HardyL1923}]
\label{HLconjecture-asymp}
Suppose $\AA = \{a_1,\dotsc,$ $a_s\}$ is an admissible set.
Choose $b\in \{-1,1\}$.
Then the number of integers $n$ in the interval $(x,2x]$ 
such that the integers 
$n,a_1n+b,\ldots,a_sn+b$ are all prime is of cardinality 
\[
\sim \prod_p \Bigl(1- \frac{\omega(p)}{p}\Bigr) 
\Bigl(1-\frac{1}{p}\Bigr)^{-s-1} 
\frac{x}{(\log x)^{s+1}},
\quad
\textrm{as} \ x\to \infty.
\]

\end{conj}
There is a specific case of Conjecture \ref{HLconjecture} that is especially famous.

\begin{conj}[Hardy--Littlewood conjecture for Sophie Germain primes \cite{HardyL1923}]
\label{HLsophie}
The number of primes $ p \leq x$ for which $2p+1$ is also prime 
is $\gg x/(\log x)^2$.
\end{conj}

\par The following result, conjectured 
by Erd\H{o}s (1988), shows that there are admissible
sets having  arbitrarily large measure $\mu$.

\begin{Thm}[Granville \cite{Gr}]
\label{unbounded}
There is a sequence of admissible sets $\mathcal A_1,\mathcal A_2,\ldots$ 
such that $\lim_{j\to  \infty}\mu(\mathcal A_j)=\infty$. 
We have $\overline{\mathcal M} = [0,\infty]$, with $\overline{\mathcal M}$ 
the closure of the set $ \{\mu(\mathcal{A}) : \mathcal{A}$ is an admissible set$\}$. 
\end{Thm}

Now, we record a useful result which plays an important role in the proof of Theorem \ref{cr1}.
\begin{lem} \label{CGlem} \cite[Lemma~6.1]{CG}
Let $M$ denotes a set of non-zero integers and $n$ be a natural number.
For the set $\Sigma(M, n) = \big \{\mu(M) : M \subset \mathbb{Z}^* , \vert M \vert < n\big \},$ the following holds.
\begin{enumerate}[label={(\roman*)}, wide, nosep, after=\vspace{4pt}]
     \item For any rational number $r$ there exists a positive constant $\nu(r)$ and a natural number $n(r)$
     such that if $s \in \Sigma(M, n(r)+1)$ with $\vert r-s \vert < \nu(r)$, then $r=s$.
     
      \item For any irrational number $r$ and for all $n$, there exists a positive constant
    $\nu(r,n)$ such that there are no elements $s \in \Sigma(M, n)$
    with $\vert r-s \vert < \nu(r,n).$
\end{enumerate}
\end{lem}

\section{Proof of Theorem~\ref{VK1}} 
\label{sec:meangamma_plus}

The proof of Theorem~\ref{VK1} exactly fits to what was done in~\cite{KM}.
In what follows $\chi$ denotes a Dirichlet character modulo $q$.
Logarithmic differentiation of $L(s,\chi)$ gives
\[
-\frac{L'(s,\chi)}{L(s,\chi)}=\sum_{n\ge 1}\frac{\Lambda(n)}{n^s}\chi(n), \quad  \Re(s) \ge 1.
\]
For any real large enough $x > 1$, we consider an
approximation
\begin{equation*}
\Phi_{\chi}(x) = \frac{1}{x} \int_{1}^{x} \Bigl( \sum_{n\le t}{\frac{\Lambda(n)}{n}}\chi(n)\Bigr)dt
\end{equation*}
of $-L'(1,\chi)/L(1,\chi)$.
This suggests to rewrite equation \eqref{gammaplus} as
\begin{equation*} 
\gamma_q^+ = \Euler 
+  
\sum_{\substack{\chi \neq \chi_0 \\ \chi(-1)=1}}\Bigl(\frac{L'(1,\chi)}{L(1,\chi)}+\Phi_{\chi}(x)\Bigr)
-
\sum_{\substack{\chi \neq \chi_0 \\ \chi(-1)=1}}\Phi_{\chi}(x). 
\end{equation*}
For $Q\geq 2$ and every $x>1$, we have
\begin{equation} \label{eq2}
\sum_{Q/2< q\leq Q} |\gamma_q^+ | \ll
\pi^*(Q) \log Q +\sum_{Q/2< q\leq Q}\Bigl \vert \sum_{\substack{\chi \neq \chi_0 \\ \chi(-1)=1} } \Phi_{\chi}(x)\Bigr|,
\end{equation}
where we  used \cite[Proposition~3.1]{KM} to obtain the bound
\[
\sum_{Q/2< q\leq Q}\Bigl \vert \sum_{\substack{\chi \neq \chi_0 \\ \chi(-1)=1} } 
\Bigl(\frac{L'(1,\chi)}{L(1,\chi)}+\Phi_{\chi}(x)\Bigr)\Bigr \vert  
\ll 
\pi^*(Q)\log Q.
\]
(Note that \cite[Proposition~3.1]{KM} can be 
applied here since its estimate is essentially given by the exceptional
zero contribution; hence the additional condition we have on the parity of the Dirichlet characters
does not change the picture.)
Thus it remains to obtain the same bound for the double sum in \eqref{eq2}.
By the definition of $\Phi_\chi(x)$, one clearly has
\begin{equation} \label{sumofPhi}
\sum_{\substack{\chi \neq \chi_0 \\ \chi(-1)=1} } \Phi_{\chi}(x) 
= 
\frac{1}{x} \int_{1}^{x} 
\Bigl( 
\sum_{n \le t} \frac{\Lambda(n)}{n} \sum_{\substack{\chi \neq \chi_0 \\ \chi(-1)=1} } \chi(n) 
\Bigr) dt.
\end{equation}
The orthogonality of even characters 
entails that
\begin{equation*}
\frac{2}{q-1} \sum_{\chi(-1)=1} \chi(a) \overline{\chi}(b) = \begin{cases}
1, & b \equiv \pm a  \pmod*{q},\\
0, &\text{otherwise}.
\end{cases}
\end{equation*}
Using this, we obtain
\[
\sum_{n \le t} \frac{\Lambda(n)}{n} \sum_{\substack{\chi \neq \chi_0 \\ \chi(-1)=1} } \chi(n) 
= 
\frac{q-1}{2} \sum_{\substack{n \le t \\ n \equiv \pm 1  \pmod*{q}}}\!\! \frac{\Lambda(n)}{n} 
- 
\sum_{\substack{n\leq t\\ (n, q)=1}} \frac{\Lambda(n)}{n}. 
\]
Since by assumption $q$ is prime, we clearly have
\[
\sum_{\substack{n\leq t\\ (n, q)=1}}\frac{\Lambda(n)}{n}
=\sum_{n\leq t}\frac{\Lambda(n)}{n}+O\Bigl(\frac{\log q }{q}\Bigr).
\]
Thus \eqref{sumofPhi} becomes
\begin{align} 
\notag
\sum_{\substack{\chi \neq \chi_0 \\ \chi(-1)=1} } \Phi_{\chi}(x) &
= \frac{1}{x} \int_{1}^{x} \Bigl(\frac{q-1}{2} 
\sum_{\substack{n \le t \\ n \equiv \pm 1  \pmod*{q}}}\!\! \frac{\Lambda(n)}{n} - \sum_{n \le t} \frac{\Lambda(n)}{n} 
+ O\Bigl(\frac{\log q }{q}\Bigr)\Bigr) dt
\\&  \label{sumofphigen}
= \frac{1}{x} \int_{1}^{x} \Bigl(\frac{q-1}{2} 
\sum_{\substack{n \le t \\ n \equiv \pm 1  \pmod*{q}}}\!\! \frac{\Lambda(n)}{n} - \sum_{n \le t} \frac{\Lambda(n)}{n} \Bigr) dt 
+ {O\Bigl(\frac{\log q}{q}\Bigr)}.
\end{align}
Further, by partial summation we have
\begin{align} 
\notag
\sum_{n\leq t}\frac{\Lambda(n)}{n}
&= \frac{\psi (t)}{t}+\int\limits_{1}^{t}\frac{\psi (u)}{u^2}\, du,\\
\label{pslam}
\sum_{\substack{n\leq t\\ n\equiv \pm 1 \pmod*{q}}}\!\! \frac{\Lambda(n)}{n} &=
\frac{\psi(t; q, \pm 1)}{t} + \int_{1}^{t}\frac{\psi(u;q,\pm 1)}{u^2}\, du,
\end{align}
where $\psi(t;q,\pm 1) = \psi(t;q, 1)+\psi(t;q,-1)$.
For a fixed non-zero integer $a$, define
\begin{equation*} 
S(Q, a, x)=\int_{1}^{x} \Bigl( \sum_{Q/2< q\leq Q} \Bigl \vert (q-1) \psi(u;q,a)- \psi(u)\Bigr \vert  \Bigr) \frac{du}{u^2}
\end{equation*}
and set $S(Q,\pm 1, x) = S(Q,-1,x) + S(Q,1,x)$.
Then on taking absolute values in \eqref{sumofphigen}
and summing over $Q/2 < q \le Q$,  
we see that the contribution of the first terms in \eqref{pslam} is bounded by
\[
\frac{1}{x} \int_{1}^{x} \Bigl(\sum_{Q/2< q\leq Q} \Bigl \vert \frac{q-1}{2} \psi(t;q,\pm 1)- \psi(t)\Bigr \vert  \Bigr) \frac{dt}{t} 
\ll S(Q,\pm 1, x).
\]
Working analogously with the integral terms of \eqref{pslam}, and enlarging the internal integration 
interval  from $[1,t]$ to $[1,x]$, we bound its contribution by
\[
\frac{1}{x} \int_{1}^{x}   \int_{1}^{t} \Bigl(\sum_{Q/2 < q \le Q} 
\Bigl \vert  \frac{q-1}{2} \psi(u;q,\pm 1)-\psi(u)\Bigr \vert  
\Bigr) \frac{du}{u^2}  dt \ll S(Q,\pm 1,x).
\]
Finally, on combining two equations above and estimating the last term of \eqref{sumofphigen} via Mertens' theorem as
\[
\frac{1}{x} \int_{1}^{x} \Bigl(\sum_{Q/2 < q \le Q} \frac{\log q}{q}\Bigr) dt \ll 1,
\]
we conclude that
\begin{equation*} 
\sum_{Q/2< q\leq Q}\Bigl \vert \sum_{\substack{\chi \neq \chi_0\\ \chi(-1)=1} } \Phi_{\chi}(x)\Bigr|
\ll S(Q,-1, x)+S(Q,1,x)+\pi^*(Q).
\end{equation*}

To complete the proof of Theorem~\ref{VK1}, it remains to 
show that $S(Q, a, x)\ll \pi^*(Q)\log Q$ for $a = \pm 1$ and for every $x>1$.
We only do this for $a=1$, the other case
being similar\footnote{The only difference being 
that in \eqref{defS_1} the integration range starts at $Q/2-1$ instead of $Q/2$ since $\psi(u;q,-1)=0$ for $u < q-1$.}
\par In the range $u>Q^3$, one can estimate {the contribution to 
$S(Q,1,x)$} using
\eqref{BV-type-ineq} with ${\mathcal Q}(u) =u^{1/3}$.
Notice, that the choice of exponent is arbitrary here, as we only need to satisfy the conditions 
of the Bombieri--Vinogradov theorem in this range, namely, we need $Q < u^{1/2-\varepsilon}$, 
hence any range $u > Q^a$ with $a >2+\varepsilon$ will work here.
This gives rise to
\[
\sum_{Q/2< q\leq Q} \Bigl \vert  (q-1) \psi(u;q,1)
- \psi(u)\Bigr \vert  \ll \frac{Qu}{(\log u)^3},
\]
which on enlarging the integration interval 
from $Q^3\le u\le x$ to $u\ge Q^3$, leads to a contribution
\begin{equation*} 
\ll Q\int_{Q^3}^\infty \frac{u}{(\log u)^{3}} \, \frac{du}{u^2} \ll Q \ll \pi^*(Q) \log Q.
\end{equation*}
(Here we used that $x$ is large enough, namely $x = Q^3 +z$, with $z>0$, large, but fixed.)
\par Next suppose that $u\le Q^3$. 
In this range, we estimate the contribution 
to $S(Q,1,x)$ trivially by
\begin{equation}
\label{StoS1}
\ll S_1(Q)+ 
\sum_{Q/2< q\leq Q}\int_{1}^{Q^3}  \psi(u)\frac{du}{u^2},
\end{equation}
where
\begin{equation} \label{defS_1}
S_1(Q)=\int_{1}^{Q^3} \sum_{Q/2< q\leq Q}\!\! (q-1) \psi(u;q,1) \, \frac{du}{u^2}
=
\int_{Q/2}^{Q^3} \sum_{Q/2< q\leq Q}\!\! (q-1) \psi(u;q,1) \, \frac{du}{u^2},
\end{equation}
on observing that $\psi(u;q,1)=0$ for $u\le q$ (in case $a=-1$ 
the integration range starts at $Q/2-1$ instead of $Q/2$ since $\psi(u;q,-1)=0$ for $u < q-1$).
Then, using that
$\psi(u)\ll u$, we see that the second term in \eqref{StoS1} is
$\ll \pi^*(Q)\log Q$ and it remains to establish the same bound for 
$S_1(Q)$. 
Notice that
\[
\sum_{Q/2 < q \le Q} (q-1) \psi(u;q,1) = \sum_{n \le u} \Lambda(n) \sum_{\substack{q \mid (n -1) \\ Q/2 < q \le Q}} (q-1).
\]
The contribution of prime powers to $S_1(Q)$ satisfies
\[
\sum_{\substack{p^k \le u \\k \ge 2}} \log p 
\sum_{\substack{Q/2 < q \le Q \\ q \mid (p^k - 1)}} (q-1)
 \ll 
Q 
\sum_{\substack{p^k \le u \\ k \ge 2}}  \log p
\sum_{\substack{Q/2 < q \le Q \\ q \mid (p^k - 1)}} 1
 \ll
Q \sum_{p \le \sqrt{u}} \log p\, \frac{\log u}{\log p} \ll Q 
\sqrt{u},
\]
where we invoked the
trivial estimate $q-1\le Q$ and used that the sum over $q$ has at most a uniformly bounded number of terms 
(since $q>Q/2$ and 
$q \mid (p^k-1)$ with $p^k-1\le Q^3$).
Thus the contribution of prime powers  for $u \le Q^3$ to 
$S_1(Q)$  is
\[
\ll Q \int_{Q/2}^{Q^3} \frac{\sqrt{u}}{u^2}du\ll \sqrt{Q}\ll \pi^*(Q) \log Q,
\]
and that of the primes is bounded
above by
\begin{equation}
\label{lemapp}
\sum_{p \le u} \log p 
\sum_{\substack{Q/2 < q \le Q \\ q \mid (p - 1)}} (q-1) 
= 
\sum_{t \le 2u/Q} 
\sum_{\substack{p \le u \\ p- 1 = tq \\Q/2 < q \le Q}} (q-1) \log p
\ll Q \sum_{t \le 2u/Q} \frac{u\log u}{\varphi(t) (\log(u/t))^2},
\end{equation}
where we invoked Lemma \ref{lem1}. 
In the fourth sum in \eqref{lemapp}, we have 
$p\le \min\{u,Qt+1\}$, which suggests to distinguish two subcases.
If $u < Qt$, then the contribution of primes to $S_1(Q)$ is bounded 
by 
\begin{align*}
&\ll \sum_{t \le 2Q^2} \int_{Qt/2}^{Qt} \frac{Q}{\varphi(t)} \frac{1}{(\log Q)^{2}} \frac{\log u}{u} du \ll 
\frac{Q}{(\log Q)^{2}} \sum_{t \le 2Q^2} \frac{1}{\varphi(t)} \bigl((\log (Qt))^2- (\log (Qt/2))^2\bigr) \\
&\ll Q \ll \pi^*(Q) \log Q,
\end{align*}
where we used $u\le Q^3$, $u>Qt/2$, the estimate $(\log (Qt))^2- (\log (Qt/2))^2 \ll \log(Qt)$ and 
\begin{equation*} 
\sum_{t \le x} \frac{1}{\varphi(t)} \ll \log x.
\end{equation*}

\noindent Further, for $u\ge Qt$
we have 
$p\le \min\{u,tQ+1\}\le tQ+1$ and so
$(q-1)\log p\ll Q\log(Qt)$, which on invoking \eqref{lemapp}
leads to  the bound  

\begin{align} 
\notag
&\ll \int_{Qt}^{Q^3} \Bigl( \sum_{t \le 2u/Q} \sum_{\substack{p \le tQ+1 \\ Q/2 < q \le Q \\ p-1=tq}}  (q-1) \log p \Bigr) \frac{du}{u^2} 
\ll \int_{Qt}^{Q^3} Q\log(Qt)\Bigl( \sum_{t \le 2u/Q} \sum_{\substack{p \le tQ+1 \\ Q/2 < q \le Q \\ p-1=tq}}  1 \Bigr) \frac{du}{u^2}
\\ \notag &
\ll \sum_{t \le 2Q^2} Q \log (Qt) \frac{Qt}{\varphi(t) (\log Q)^2} \int_{Qt}^\infty \frac{du}{u^2}
\ll \frac{Q}{(\log Q)^2}\sum_{t \le 2Q^2} \frac{\log (Qt)}{\varphi(t)} 
\\ 
\notag &
\ll Q \ll \pi^*(Q) \log Q,
\end{align}
for the contribution to $S_1(Q)$.
Adding the bounds for the contributions of the
$u<Qt$, respectively $u\ge Qt$, we obtain $S_1(Q)\ll \pi^*(Q)\log Q$, as 
desired.

\section{Kummer's conjecture and Proof of Theorem~\ref{alerq}}
\label{sec:Kummer}

\par The orthogonality property of odd characters, 
\begin{equation*} 
\label{ortodd}
\frac{2}{q-1} \sum_{\chi(-1)=-1} \chi(a) \overline{\chi}(b) = \begin{cases}
\pm 1, & b \equiv \pm a  \pmod*{q},\\
0, &\text{otherwise},
\end{cases}
\end{equation*}
gives us
$$\sum_{\chi(-1)=-1} \log(L(s,\chi)) = \frac{q-1}{2}\lim_{x\to  \infty}
\Bigl(
\sum_{\substack{m\ge 1 ;\ p^m\le x\\ p^m\equiv 1 \pmod*{q}}}\!\! \frac{1}{mp^{ms}} 
- 
\sum_{\substack{m\ge 1 ;\ p^m\le x \\ p^m\equiv -1 \pmod*{q}}}\!\! \frac{1}{mp^{ms}}\Bigr),
$$
where the sum is over all prime powers 
$p^m\le x$. In the papers \cite{deb,MasMon,Puchta} just mentioned, the authors consider the latter function
in a neigborhood of $s=1$ (but whereas
Debaene \cite{deb} and Puchta \cite{Puchta} take higher derivatives into account,
Masley and Montgomery \cite{MasMon} stopped at the first derivative). Here, we will actually set $s=1$,
which in combination with \eqref{hasse} yields
\begin{equation}
\label{loggie}
r(q)=\log R(q) =\frac{q-1}{2}\lim_{x\to  \infty}
\Bigl(\sum_{\substack{m\ge 1 ;\ p^m\le x\\ p^m\equiv 1 \pmod*{q}}}\!\! \frac{1}{m p^m}-
\sum_{\substack{m\ge 1; \ p^m\le x \\ p^m\equiv -1 \pmod*{q}}}\!\! \frac{1}{mp^m}\Bigr).
\end{equation}
\begin{Defi}
\label{fqdefi}
The argument in the limit we denote by $f_q(x)$ and $f_q=\lim_{x\to  \infty}f_q(x)$.
\end{Defi}
Note that Kummer's conjecture is equivalent with $f_q=o(1/q)$ as $q$ tends to infinity.
The idea is now to choose $x$ as small 
as possible so that the resulting error 
$f_q-f_q(x)$ is still reasonable. In 
attempting to decrease $x$, the Bombieri--Vinogradov Theorem \ref{BV-type-ineq} and 
the Brun--Titchmarsh Classical Theorem \ref{BT-thm} play a crucial role.
The main contribution to $f_q(x)$ comes from the
term with $m=1$, and is denoted by $g_q(x)$:
\begin{equation}
\label{geeq}
g_q(x)=\sum_{\substack{p\le x\\ p\equiv 1 \pmod*{q}}}\!\! \frac{1}{p}-
\sum_{\substack{p\le x \\ p\equiv -1 \pmod*{q}}}\!\! \frac{1}{p}.
\end{equation}
Taking all this into account
Granville \cite{Gr} showed that if Kummer's conjecture is true then for every $\delta>0$ we must have
\begin{equation*}
g_q(q^{1+\delta})=o\Bigl(\frac{1}{q}\Bigr),
\end{equation*}
for all but at most $2x/(\log x)^{3}$ primes $q\le x$. He used this to show that
$c^{-1}\le R(q)\le c$
for a positive proportion $\rho(c)$ of primes $p\le x$, where $\rho(c)\to  
1$ as $c$ tends to infinity. 
\par Murty and Petridis \cite{MP} improved this as follows.
\begin{Thm}
There exists a constant $c>1$ such that for a sequence of primes with natural density 
1 we have $$ \max\{R(q),R(q)^{-1}\}\le c.$$ 
If the Elliott--Halberstam conjecture \ref{EHconjecture} is true, then
we can take $c=1+\epsilon$ for any fixed $\epsilon>0$.
\end{Thm}

Thus typically $R(q)$ is close to $1$, but conjecturally very different behavior 
is also possible (on very thin sets of primes), see also Section \ref{sec:spec}.
\begin{Thm}[Granville \protect{\cite[Theorems 2 and 4]{Gr}}]
If the lower bound version of the Hardy--Littlewood conjecture \ref{HLconjecture} is true, and also
the Elliott--Halberstam conjecture \ref{EHconjecture}, then for any admissible set 
${\mathcal A}$, the numbers
$e^{\mu(\mathcal A)/2}$ and $e^{-\mu(\mathcal A)/2}$ are both
limit points of the sequence $\{R(q):q\text{~is~prime}\}$.
Furthermore, this sequence has $[0,\infty]$ as set of limit points.
\end{Thm}

We end this section with the proof of Theorem \ref{alerq} in which we will 
use some classical analytic number theory techniques to show that $r(q)$ is absolutely bounded on average.

\subsection{Proof of Theorem \ref{alerq}}
From \eqref{loggie} we have that
\begin{equation*}
r(q) = 
\frac{q-1}{2}\lim_{x\to  \infty}
\Bigl(
\sum_{\substack{m\ge 1 ;\ p^m\le x\\ p^m\equiv 1 \pmod*{q}}}\!\! \frac{1}{mp^{ms}}
- 
\sum_{\substack{m\ge 1 ;\ p^m\le x \\ p^m\equiv -1 \pmod*{q}}}\!\! \frac{1}{mp^{ms}}\Bigr).
\end{equation*} 
Lemma 1 of \cite{paper1-rq}
implies that the contribution of the prime powers is negligible.
In fact, recalling the definition \eqref{geeq},  one has
\begin{equation*}
r(q) \ll 
(q-1) \lim_{x\to  \infty}  \vert g_q(x) \vert  + 1.
\end{equation*}

We split now the summation ranges in $p\le Q^3$ and $p>Q^3$ and we define accordingly
the quantities $T_{1}(q)$ and $T_{2}(q)$.
By using the partial summation formula, we estimate these as
$$
T_{1}(q)  \ll \frac{\vert\theta(Q^3; q, -1) - \theta(Q^3; q, 1)\vert}{Q^3 \log Q} + \int_{2}^{Q^3} 
\bigl \vert \theta(u; q, 1) - \theta(u; q, -1)\bigr\vert \frac{\log u +1}{(u \log u)^2} du, 
$$
respectively
$$
T_{2}(q) \ll \frac{\vert\theta(Q^3; q, -1) - \theta(Q^3; q, 1)\vert}{Q^3 \log Q} 
+ 
\int_{Q^3}^{\infty} \bigl\vert\theta(u; q, 1) - \theta(u; q, -1)\bigr\vert \frac{\log u +1}{(u \log u)^2} du.
$$
Using the following weak form of the Prime Number Theorem in arithmetic progressions
\[
\theta( Q^3;q, \pm 1) \ll \frac{Q^3}{q-1},
\]
where $\theta(t;q,\pm 1) = \theta(t;q, 1)+\theta(t;q,-1)$,
we have
\begin{align} \label{om}
\sum_{Q/2 < q \leq Q} (q-1) T_{1}(q) 
\ll \pi^*(Q) + \frac{1}{\log Q} \int_{Q/2}^{Q^3} 
\Bigl( \sum_{Q/2 < q \leq Q} (q-1) \theta(u; q, \pm 1)\Bigr) \frac{du }{u^2 },
\end{align}
where we also used that $\theta(u;q,\pm 1)=0$ for $u\le q$.
Using the estimate of $S_1(Q)$ in the proof of Theorem \ref{VK1}, for the case in which no prime powers are involved,
the quantity in the right hand side of equation \eqref{om} is $\ll \pi^*(Q)$. 
Hence
\begin{equation*}
\sum_{Q/2 < q \leq Q} (q-1) T_{1}(q) \ll \pi^*(Q).
\end{equation*}
Similarly, for $a\in \{- 1, 1\}$,
\begin{align*} 
\sum_{Q/2 < q \leq Q} (q-1) T_{2}(q) \ll \pi^*(Q) 
+ \frac{1}{\log Q} \int_{Q^3}^{\infty} \Bigl( \sum_{Q/2 < q \leq Q} |  (q-1) \theta(u; q, a) - \theta(u) |\Bigr) \frac{du }{u^2}.
\end{align*}
As in the proof of Theorem \ref{VK1},  we apply the Bombieri--Vinogradov theorem. 
As this integral is $\ll Q$, it follows that 
\begin{equation*}
\sum_{Q/2 < q \leq Q} (q-1) T_{2}(q) \ll \pi^*(Q).
\end{equation*}
By collecting all these estimates, Theorem \ref{alerq} follows. \qed

\section{Distribution of \texorpdfstring{$\kappa(q)$}{gq-gq+}
and Proof of Theorem \ref{gq-direct}}
\label{sec:kappaq-initiol}

\subsection{Strengthening the analogy with $r(q)$} 

We begin our investigation into the properties of $\kappa(q)$
that will lead us to establish analogies with the
asymptotic behavior of $r(q)$.

We start by recalling some relevant identities for 
$\gamma_q$ and $\gamma_q^+$, respectively defined in \eqref{gammaq} and \eqref{gammaplus}.
For a proof and a more general variant, see Ciolan et al.\,\cite[Proposition~4]{CLM}. 
The second formula for $\gamma_{q,k}$ can be
derived by combining the splitting laws of rational primes in $\mathbb Q(\zeta_q)$ and ${\mathbb Q}^+(\zeta_q)$ 
with \eqref{hickup}.
\begin{prop}
Let $k\in \{1,2\}$. Putting $\gamma_{q,1}=\gamma_q$ and $\gamma_{q,2}=\gamma_q^+$ we have
$$\gamma_{q,k} = - \frac{\log q}{q-1}  + 
\lim_{x\to\infty} \Bigl( \log x - \frac{q-1}{k}\sum_{\substack{n\le x 
\\ n^k\equiv 1\pmod*q}} \frac{\Lambda(n)}{n} \Bigr).$$
Furthermore,
\begin{align*}
\gamma_{q,k} 
&=- \frac{\log q}{q-1} - \frac{q-1}{k}S_k(q) + 
\lim_{x\to\infty} \Bigl( \log x - \frac{q-1}{k}\sum_{\substack{p\le x 
\\ p^k\equiv 1 \pmod*{q}}} \frac{\log p}{p-1} \Bigr),
\end{align*}
where
\begin{equation*}
\label{Sdefinition}
S_1(q)
= 
\sum_{\text{ord}_q(p) \geq 2}{\frac{\log p} {p^{\text{ord}_q(p)}-1}},\quad S_2(q)
= 
\sum_{\text{ord}_q(p^2) \geq 2}{\frac{\log p} {p^{\text{ord}_q(p^2)}-1}},
\end{equation*}
where ord$_q(a)$ denotes the multiplicative order of $a$ modulo $q$.
\end{prop}

From the previous proposition, we  deduce the following corollary.
\begin{cor}
\label{elegantdifference}
We have $$\gamma_{q}^+-\gamma_q = \frac{q-1}{2} 
\lim_{x\to\infty} 
\Bigl( \sum_{\substack{n\le x \\ n\equiv 1\pmod*q}} \!\! \frac{\Lambda(n)}{n} -
\sum_{\substack{n\le x \\ n\equiv -1\pmod*q}} \!\! \frac{\Lambda(n)}{n}\Bigr),$$
and
\[
\gamma_{q}^+-\gamma_q =
(q-1) \Bigl(S_1(q) -\frac{S_2(q)}{2} \Bigr)
+
\frac{q-1}{2} 
\lim_{x\to\infty} 
\Bigl( \sum_{\substack{p\le x \\ p\equiv 1\pmod*q}} \!\! \frac{\log p}{p-1}  -
\sum_{\substack{p\le x \\ p\equiv -1\pmod*q}} \!\! \frac{\log p}{p-1} \Bigr).
\]
\end{cor}

\subsection{The behavior of $\kappa(q)$}
\label{sub5.2}

In this subsection, we will see the interrelation between $\kappa(q)$ and $r(q)$.
It follows from \eqref{gq-def} and Corollary \ref{elegantdifference} that
\begin{equation}\label{watowat}
\kappa(q) =  \frac{q-1}{2} \bigl(v_q + w_q\bigr),
\end{equation}
where
\begin{equation}
\label{alpha-def}
v_q  = \frac{1}{\log q} \Bigl(\sum_{\substack{m \geq 2\\ p^{m} \equiv 1 \pmod*{q}}} \!\! \frac{\log p}{p^m}
-
\sum_{\substack{m \geq 2\\ p^{m} \equiv -1 \pmod*{q} }} \!\! \frac{\log p}{p^m}\Bigr)
\end{equation}
and
\begin{equation}
\label{beta-def}
w_q = \lim_{x\to  \infty} w_q(x), \quad w_q(x)  
= \frac{1}{\log q} \Bigl(\sum_{\substack{p\le x\\ p\equiv 1 \pmod*{q}}}\!\! \frac{\log p}{p}
-\sum_{\substack{p\le x\\p\equiv -1 \pmod*{q}}} \!\! \frac{\log p}{p}\Bigr).
\end{equation}
Clearly
\begin{equation}
\label{alpha-d}
|v_q|\le 
\frac{1}{\log q} 
\max\Bigl\{\sum_{\substack{m \geq 2\\ p^{m} \equiv 1 \pmod*{q}}} \!\! \frac{\log p}{p^m},
\sum_{\substack{m \geq 2\\ p^{m} \equiv -1 \pmod*{q} }} \!\! \frac{\log p}{p^m}\Bigr\}.
\end{equation}

Observe that for each odd prime power $p^m$ contributing to the right hand side
in \eqref{alpha-d} we have $\frac{\log p}{\log q}\ge \frac{1}{m}$, 
and hence  we cannot directly use \cite[Lemma 1]{paper1-rq}.
Nevertheless, we can estimate $\vert v_q \vert$ following
the same argument already used there.
We note that
$(43-18\zeta(3))/13=1.6433058\dotsc$
\begin{prop}\label{nonsplit}
The estimate
$v_q = o(1/q)$ holds except for at most $O(\sqrt{x} (\log x)^{2})$ primes $q \leq x$.
Further,
there exists a constant $c>0$ such that for every prime $q$ we have
$$
| v_q |  \leq \frac{1}{q} \Bigl(\frac{43}{13}-\frac{18}{13}\zeta(3)\Bigr) + \frac{c \log_2 q}{q\log q}.
$$
\end{prop}

\begin{proof}
The proof is similar to that of \cite[Lemma 1]{paper1-rq}.
We split each of the sums in \eqref{alpha-d} into three subsums 
according to $p^m \leq q  (\log q)^{2}$, $q(\log q)^{2} < p^m < q^2$
and $p^m>q^2$.
The contribution to the final estimate of the sums over the first range
will be less than $c_1/q$, with a constant $c_1>0$ that will be explicitly determined, while the
others contribute $\ll \log_2 q/(q\log q)$.

Arguing as in \cite[Lemma 1]{paper1-rq},
for $q \geq 3$, we have
\begin{align}
\notag
\sum_{\substack{m \geq 2;\ p^m>q^2\\ p^{m} \equiv \pm 1 \pmod*{q}}}
\frac{\log p}{\log q}\,\frac{1}{p^m}
& \leq 
\frac{1}{\log q} \sum_{p > q} \sum_{ m \geq 2 } \frac{\log p}{p^m} +
\frac{1}{\log q} \sum_{p \le q} \sum_{\substack{m \geq 2 \\ p^m > q^2}} \frac{\log p}{p^m}\\
\label{first-interval-new}
& \leq \frac{4}{3\log q} \sum_{p> q} \frac{\log p}{p^2} +   \frac{4}{3\log q}  \sum_{p \le q}  \frac{\log p}{q^2} 
\ll
 \frac{1}{\log q}  \int_q^{\infty} \frac{\theta(t)}{t^3} dt,
\end{align}
where in the last step we used the partial summation formula.
Using  the Chebyshev bound in the weaker form $\theta(t) \ll t$, 
one obtains
\[
\int_q^{\infty} \frac{\theta(t)}{t^3} dt
\ll
\frac{1}{q} ,
\]
and hence \eqref{first-interval-new} becomes
\begin{equation}
\label{first-interval-estim}
\sum_{\substack{m \geq 2;\ p^m>q^2\\ p^{m} \equiv \pm 1 \pmod*{q}}}
\frac{\log p}{\log q}\,\frac{1}{p^m}
\ll
\frac{1}{q\log q}.
\end{equation}
Further,
\begin{align}  
\label{second-interval}
\sum_{\substack{m \geq 2;\ q(\log q)^{2} < p^m < q^2\\ p^{m} \equiv \pm 1 \pmod*{q}}}
\frac{\log p}{\log q}\,\frac{1}{p^m}
\leq  
\sum_{\substack{m \geq 2;\ q(\log q)^{2} < p^m < q^2\\ p^{m} \equiv \pm 1 \pmod*{q}}}
\frac{2}{m p^m} 
\ll
\frac{1}{q\log q},
\end{align}
where the final step follows from 
eq.~(15) of \cite{paper1-rq}.
Finally, since $\log p \le (1/m) (\log q + 2 \log_2 q)$ 
for $p^m \leq q  (\log q)^{2}$,
we obtain, using eq.~(18) of \cite{paper1-rq}, that
\begin{align}\label{dual}
\sum_{\substack{m \geq 2; \, p^m \leq q  (\log q)^{2} \\ p^m\equiv b \pmod*{q}}}
\frac{\log p}{\log q}\,\frac{1}{p^m} 
& \leq 
\Bigl( 1 +  \frac{2 \log_2 q}{\log q}\Bigr)
\sum_{\substack{m \geq 2; \, p^m \leq q  (\log q)^{2} \\ p^m\equiv b \pmod*{q}}}
\frac{1}{m p^m}
\\
\notag 
& \leq \frac{1}{q} \Bigl(\frac{43}{13}-\frac{18}{13}\zeta(3)\Bigr) 
\Bigl( 1 +  \frac{2 \log_2 q}{\log q}\Bigr)\Bigl(1+ \frac{c_1}{q}\Bigr)\\ 
& =
\frac{1}{q} \Bigl(\frac{43}{13}-\frac{18}{13}\zeta(3) \Bigr) + \frac{c_2\log_2 q}{q\log q} 
\label{dual2}
\end{align}
for every $b\in\{-1,1\}$,
where $c_1,c_2$ are suitable positive constants.
The second part of Proposition \ref{nonsplit} follows on inserting
\eqref{first-interval-estim}-\eqref{dual2} into \eqref{alpha-d}.

From the proof of \cite[Proposition 1]{Gr}, we have
\begin{equation}\label{granpf}
\sum_{x/2 < q \leq x} \sum_{\substack{m \geq 2; \, p^m \leq q  (\log q)^{2} \\ p^m\equiv b \pmod*{q}}} 
\frac{1}{m p^m}
\ll \frac{1}{\sqrt{x}},
\end{equation}
and hence the sum in the intermediate estimate into \eqref{dual} is equal to 
$ o(1/q)$ for all but $O(\sqrt{x} (\log x)^{2})$ primes $q \leq x$. Therefore, 
\begin{equation}
\label{third-interval} 
\frac{1}{\log q} \Bigl(\sum_{\substack{m \geq 2\\ p^{m} \equiv 1 \pmod*{q}}} \!\! \frac{\log p}{p^m}
+\sum_{\substack{m \geq 2\\ p^{m} \equiv -1 \pmod*{q} }} \!\! \frac{\log p}{p^m}\Bigr)
=
o\Bigl(\frac{1}{q}\Bigr),
\end{equation}
for all but $O(\sqrt{x} (\log x)^{2})$ primes $q \leq x.$
The proof is concluded on 
inserting the upper bounds given in \eqref{first-interval-estim}-\eqref{second-interval}
and \eqref{third-interval} into \eqref{alpha-d}.
\end{proof}

\begin{Rem}
\label{sharp-constant-2}
Arguing as explained in 
\cite[Remark 5]{paper1-rq},
the  constant $\frac{43}{13}-\frac{18}{13}\zeta(3)$ in the leading term of Proposition \ref{nonsplit}
can be replaced by $1.601$.
\end{Rem}

\begin{cor}\label{firstsum1}
As $Q$ tends to infinity, we have
\begin{equation*}
    \sum_{Q/2 < q \leq Q} q\vert v_q \vert = o(\pi^*(Q)).
\end{equation*}
\end{cor}

\begin{proof}
Immediate on using equations \eqref{first-interval-estim}-\eqref{dual} and \eqref{granpf}.
\end{proof}

\subsection{Proof of Theorem \ref{gq-direct}}

Recalling \eqref{gq-def} and \eqref{watowat}, we have
$
\kappa(q) 
=
\frac{q-1}{2} (v_q + w_q),
$
where $v_q, w_q$ are respectively defined in \eqref{alpha-def} and \eqref{beta-def}.
We note that  Proposition \ref{nonsplit} implies
\begin{equation}\label{131}
\frac{q-1}{2} |v_q|  \leq 
\frac{43}{26}-\frac{9}{13}\zeta(3) 
+ \frac{c \log_2 q}{\log q},
\end{equation}
for every odd prime $q$, where $c>0$ is a suitable constant.
Let us now write, for brevity,
\begin{equation}
\label{sigma-def}
\Sigma = \frac{q-1}{2} w_q 
=\frac{q-1}{2\log q} \lim_{x\to  \infty}  \Bigl(\sum_{\substack{p\le x\\ p\equiv 1 \pmod*{q}}}\!\! \frac{\log p}{p}
-\sum_{\substack{p\le x\\p\equiv -1 \pmod*{q}}} \!\! \frac{\log p}{p}\Bigr).
\end{equation}
Let $\ell(q)$ be any monotonic function 
tending to infinity with $q$.
We split $\Sigma$ in three 
subsums $S_1, S_2, S_3$, defined according
to whether $p\le x_1$, $x_1 < p \le x_2$ and $p\ge x_2$, where
$x_2 = e^q$ and $x = x_1= q^{\ell(q)}$.

By the partial summation formula and 
the Siegel--Walfisz theorem, see, e.g., \cite[ch.~22]{Davenport}, we have
\begin{align}\label{S3-estim}
\notag
S_3 &= \frac{q-1}{2\log q} \int_{x_2}^{\infty} \frac{1}{u} d\{\theta(u;q,1) - \theta(u;q,-1)\}\\
& \ll \frac{q}{\log q} \Bigl( \frac{1}{e^{c_1 \sqrt{q}}} + \int_{x_2}^{\infty} 
\frac{1}{u e^{c_1 \sqrt{\log u}}} du\Bigr) \ll \frac{q^2}{(\log q) e^{c_1 \sqrt{q}}}, 
\end{align}
where $c_1>0$ is an absolute constant.  

We now proceed to estimate $S_2$.
Recall (see, e.g., \cite[ch.~19]{Davenport}) that if $\chi$ is a non-principal 
character to the modulus $q$ and $2 \leq T \leq x$, then
\begin{equation*}
\theta(x,\chi) = \sum_{p\le x} \chi(p) \log p = - \delta_q 
\Bigl(\ \frac{x^{\beta_0}}{\beta_0} +  \ \frac{x^{1 - \beta_0}}{1- \beta_0}  \Bigr) - 
\sideset{}{'}\sum_{|\gamma| \leq T} \frac{u^{\rho}}{\rho} 
+ O\Bigl( \frac{x (\log qx)^2 }{T} + \sqrt{x}\Bigr),
\end{equation*}
where $\delta_q = 1$ if the Siegel zero $\beta_0$ exists and it is zero otherwise, and
$\sum^\prime$ is the sum over all non-trivial zeros 
$\rho = \beta + i \gamma$ of $L(s,\chi)$ except  $\beta_0$ 
and its symmetric zero $1-\beta_0$
\footnote{We consider here both the contribution of $\beta_0$ 
and $1-\beta_0$ in order to better compare our result with that of Dixit-Murty  \cite[Theorem~1.2]{DixitM2023}.}.
Using the partial summation formula and choosing $T=q^4$, we have
\begin{align}\label{late2}
\notag
 ( \log q ) S_2 & = \sum_{\chi(-1) = -1} 
\sum_{x_1 < p \le x_2} \frac{\chi(p) \log p}{p} 
 = \sum_{\chi(-1) = -1} \Bigl( \frac{\theta(x_2, \chi)}{x_2 } - \frac{\theta(x_1, \chi)}{x_1 } + \int_{x_1}^{x_2} \theta( u,\chi) \frac{du}{u^2}  \Bigr)\\
& = - \delta_q\int_{x_1}^{x_2}  \bigl(u^{\beta_0-2} + u^{- \beta_0 - 1} \bigr)du - \int_{x_1}^{x_2}\Bigl( \sum_{\chi(-1) = -1} 
\sideset{}{'}\sum_{|\gamma| \leq q^4} u^{\rho -2}\Bigr) du + \frac{q-1}{2} E_q,
\end{align}
where 
\[
E_q   \ll  
\int_{x_1}^{x_2}  \Bigl(\frac{(\log qu)^2}{q^4 u} + \frac{1}{u^{3/2}}\Bigr) du
\ll \frac{1}{q}  ,
\]
for sufficiently large $q$. 
By using  \cite[Lemmas~1, 7 and~8]{LuZhang}\footnote{
Note that \cite[Lemma~7]{LuZhang} holds for every $T,x_1$ such that 
$\lim_{q \to \infty} \log(qT) /\log x_1 = 0$. This  allows us to choose $T=q^4$ and 
$x_1= q^{\ell(q)}$, where $\ell(q)$ tends to infinity
arbitrarily slowly and monotonically as $q$ tends to 
infinity. The final error term in 
\cite[Lemma~8]{LuZhang} is then $\ll 1/\ell(q)=o(1)$.}, we obtain 
\begin{equation}
\label{Lemma7-LuZhang2}
\int_{x_1}^{x_2}\Bigl( \sum_{\chi(-1) = -1} 
\sideset{}{'}\sum_{|\gamma| \leq q^4} u^{\rho -2}\Bigr) du
\ll \frac{\log q}{\ell(q)}.
\end{equation}
A direct computation, recalling that 
$x_1= q^{\ell(q)}$ 
and $x_2 = e^q$, yields
\begin{align}
\label{siegel-zero-term2}
\notag
\int_{x_1}^{x_2}  u^{\beta_0-2}du &
= \int_{1}^{\infty} u^{\beta_0-2}du  -\int_{1}^{x_1} u^{\beta_0-2}du -
\int_{x_2}^{\infty} u^{\beta_0-2}du \\
& \geq \frac{1}{1- \beta_0} 
- (\log q) \ell(q)
- \frac{1}{(1- \beta_0)e^{q(1- \beta_0)}}.
\end{align}
Similarly,
\begin{align}
\label{siegel-zero-term3}
\int_{x_1}^{x_2}  u^{- \beta_0-1}du 
& \geq \frac{1}{\beta_0} - 2 - \frac{1}{\beta_0 e^{q \beta_0}}.
\end{align}

Inserting 
\eqref{Lemma7-LuZhang2}-\eqref{siegel-zero-term3} into  \eqref{late2}, we finally obtain
\begin{equation}\label{f2}
\Big \vert  S_2 + \frac{\delta_q}{\beta_0(1- \beta_0)\log q} \Big \vert
\leq  
\delta_q \ell(q) + \frac{c}{\ell(q)}, 
\end{equation}
for every sufficiently large prime $q$ and some positive constant $c$.

We now proceed to estimate $S_1$.
The first ingredient is the following inequality (valid
for any $x>0$):
\begin{equation}
\label{bounds-sigma1-trunc}
- \frac{q-1}{2} \sum_{\substack{p \leq x 
\\ p \equiv -1 \pmod*{q}} } \frac{\log p}{p} \leq  \sum_{\chi(-1)=-1} \sum_{p \leq x}  \frac{\chi(p) \log p}{p} \leq \frac{q-1}{2}
\sum_{\substack{p \leq x \\ p \equiv 1 \pmod*{q}} } \frac{\log p}{p}.
\end{equation}

We proceed by evaluating the two 
prime sums over $p \equiv b \pmod*{q}$, $b \in \{-1, 1\}$.
Letting $x \geq q^2$ and $b \in \{-1, 1\}$, by
partial summation and the Brun--Titchmarsh theorem, see Classical Theorem \ref{BT-thm}, we have
\begin{align}
\notag  \sum_{\substack{kq< p \leq x \\ p \equiv b \pmod*{q}}}\!\!& \frac{\log p}{p}  
 = \frac{\pi(x;q,b)\log x}{x} - \frac{\pi(kq; q,b)\log(kq)}{kq} 
+ \int_{kq}^{x} \frac{( \log u -1) \pi(u; q,b)}{u^2} du 
 \\&
 \label{largeprimes}
 \le 
 \frac{2} {q-1} \Bigl(
 \frac{\log x}{\log (x/q)}
 + \log q\log_2 \bigl( \frac{x}{q} \bigr)
 + \log \bigl( \frac{x}{q} \bigr)
 - (\log q) \log_2 k \Bigr),
\end{align}
where $k\ge 3$ is an odd integer we will choose later.
Since $k$ is odd and $q\ge 3$, the primes $p \leq kq$, $p \equiv b \pmod*{q}$, 
$b \in \{-1, 1\}$, are at most the ones of the type $p = 2jq + b$, where $j=1,\dotsc,(k-1)/2$.
So 
\begin{equation}
\label{smallprimes}
\frac{q-1}{2}
\sum_{\substack{p \leq kq \\ p \equiv b \pmod*{q}} } \frac{\log p}{p}  
 \le
 \frac{q-1}{2}
  \sum_{j=1}^{(k-1)/2}  \frac{\log(2jq-1)}{2jq-1}
  <
\frac{1}{4}   \sum_{j=1}^{(k-1)/2}  \frac{\log(2jq-1)}{j}.
\end{equation}

Combining \eqref{largeprimes}-\eqref{smallprimes}, we obtain
\begin{equation}
 \label{addition}
 \frac{q-1}{2}\!\!\!\!   \sum_{\substack{p \leq x \\ p \equiv b \pmod*{q}}}\!\! \frac{\log p}{p}   
 \leq  \frac{\log x}{\log (x/q)} 
+ (c_1(q,k) - \log_2 k) \log q
+\log q \log_2 \bigl( \frac{x}{q} \bigr)
+  \log \bigl(\frac{x}{q} \bigr),
\end{equation}
where $$c_1(q,k)= \frac{1}{4}   \sum_{j=1}^{(k-1)/2}  \frac{\log(2jq-1)}{j \log q}.$$
Letting $x=x_1= q^{\ell(q)}$ into \eqref{addition} leads to
\begin{align*}
\frac{q-1}{2 \log q}\!\!\!\! 
\sum_{\substack{p \leq x_1 \\ p \equiv b \pmod*{q}}}\!\! \frac{\log p}{p}  
&\le 
\log_2 q 
+\ell(q) + \log \ell(q)
+ c_1(q,k) - \log_2 k -1
+ \frac{c }{\log q},
\end{align*}
for every sufficiently large prime $q$ and some positive constant $c$.
Moreover, since  
$$c_1(q,k) 
\le 
\frac{1}{4}   \sum_{j=1}^{(k-1)/2}  \frac{1}{j} \Bigl(1+ \frac{\log (2j)}{1400}\Bigr)
\quad
\text{for every}\ q\ge \exp(1400)\approx 1.02866\cdot10^{608},$$
using \eqref{bounds-sigma1-trunc}
for sufficiently large $q$ we have
\begin{equation}\label{F-estim}
\vert S_1 \vert 
\le \log_2 q 
+\ell(q) + \log \ell(q)
+ c_2(k) - 1 + \frac{c}{\log q},
\end{equation}
where  
$$c_2(k)=\frac{1}{4}   \sum_{j=1}^{(k-1)/2}  \frac{1}{j} \bigl(1+ \frac{\log(2j)}{1400}\bigr) - \log_2 k.$$
The minimal value for $c_2(k)$
is attained for  $k=55$ and is $<-0.413812$.

Using \eqref{S3-estim}, \eqref{f2} and \eqref{F-estim} for $k=55$, we obtain
\begin{equation}
\Big \vert \Sigma + \frac{\delta_q}{\beta_0(1- \beta_0)\log q} \Big \vert
\le \log_2 q 
+ ( \delta_q+1) \ell(q) + \log \ell(q) + c_2(55) -1
\label{estim4}
+ \frac{c}{\ell(q)}, 
\end{equation}
for every sufficiently large prime $q$ and some positive constant $c$.

Combining \eqref{watowat}, \eqref{131} and \eqref{estim4}, we arrive at
for sufficiently large prime $q$ at
$$
\Big|\kappa(q) + \frac{\delta_q}{\beta_0(1- \beta_0)\log q}\Big| 
<  
\log_2 q 
+ ( \delta_q+1) \ell(q) + \log \ell(q)
-0.59,
$$
since $\frac{43}{26}-\frac{9}{13}\zeta(3) + c_2(55)-1 
< -0.59$.
The first part of Theorem \ref{gq-direct} immediately follows.

To prove the second part of Theorem \ref{gq-direct}, 
we split $\Sigma$ in \eqref{sigma-def} in two subsums $S_1, S_2$ 
defined ac\-cording to $p\le x_1$ and $p>x_1$. 
Let $A>0$ be a constant to be chosen later and let $x = x_1 = q^{2}(\log q)^A$ in \eqref{addition}.
We immediately obtain
\begin{equation} \label{128}
\vert S_1 \vert < \log_2 q  + c_2(55) +1 +  2(A+1) \frac{\log_2 q}{\log q}
\end{equation}
for sufficiently large $q$.
Moreover, arguing as in \eqref{late2}, we can write
\begin{equation}\label{late-rh-odd}
S_2
 = \frac{1}{\log q} \sum_{\chi(-1) = -1} 
 \lim_{y \to \infty} 
 \Bigl( \frac{\theta(y, \chi)}{y} - \frac{\theta(x_1, \chi)}{x_1} 
 + \int_{x_1}^{y} \theta( u,\chi) \frac{du}{u^2} \Bigr).
\end{equation}
Assuming Conjecture \ref{RH-odd-conjecture} holds,
we have that $\psi(x,\chi) \ll \sqrt{x}\, (\log x)^2$ for every odd Dirichlet character $\chi \pmod*{q}$,
see, e.g., \cite[p.~125, ch.~20]{Davenport}.
Recalling $\psi(x,\chi) = \theta(x,\chi) +O(\sqrt{x})$, we deduce from
\eqref{late-rh-odd} that
\begin{equation}
\label{S2-rh-estim}
S_2
\ll 
\frac{q}{\log q} \frac{(\log x_1)^2}{\sqrt{x_1}}
\ll_A
(\log q)^{1-A/2} = o(1),
\end{equation}
for every $A>2$.
Choosing $A=3$ we have,  by combining \eqref{128} and \eqref{S2-rh-estim}, 
\begin{equation}
\label{sigma-rh-estim}
\vert \Sigma \vert 
<\log_2 q  +  c_2(55) +  1 + \frac{c}{\sqrt{\log q}},
\end{equation}
where $c$ is a suitable positive constant.
Using \eqref{watowat}, \eqref{131} and \eqref{sigma-rh-estim} we finally obtain
$$
|\kappa(q)| 
< \log_2 q + 1.41
$$
for sufficiently large $q$, on observing that $\frac{43}{26}-\frac{9}{13}\zeta(3) + c_2(55) +1  
< 1.41$.
This proves the second part of Theorem \ref{gq-direct}.
\qed

\begin{Rem}[On the role of the Brun--Titchmarsh theorem in our results, I]
The leading constant of the $\log_2 q$ term
in the estimates for $\kappa(q)$ in Theorem \ref{gq-direct} 
directly depends on eq.~\eqref{F-estim} 
that follows from using the Brun--Titchmarsh theorem (Classical Theorem \ref{BT-thm})
in \eqref{largeprimes}.
In particular, a key role in \eqref{largeprimes} is played by the constant $2$ present in \eqref{BT-estim}; 
any improvement of this constant to, e.g., $2-\xi$,
$\xi\in(1,2)$, will lead to an improvement of the leading constant $1$ to $1-\xi/2$ 
in Theorem \ref{gq-direct}.
From the works of Motohashi \cite{Motohashi1979}, Friedlander--Iwaniec
\cite{FriedlanderI1997}, Ramar\'e \cite[Theorem~6.5]{Ramare2009}  and Maynard
\cite[Proposition~3.5]{Maynard2013}, it is well known that replacing such a constant with any value less
than $2$ is equivalent to assuming the non-existence of the Siegel zero for the 
Dirichlet $L$-series defined using the whole set of Dirichlet characters modulo $q$.
However, since in our problem just the odd Dirichlet characters are 
involved, this is not quite what we 
need. It seems very challenging to pursue this road further, which we leave to the courageous reader.

\end{Rem}
\begin{Rem}[On the role of the Brun--Titchmarsh theorem in our results, II]
Montgomery and Vaughan, see \cite[Theorem~2]{MVsieve}, under the same hypotheses 
of Classical Theorem \ref{BT-thm}, also proved that there exists a
constant $C>0$ such that if $y>Cq$ then
\begin{equation}
\label{BT-alternative-MV}
\pi(x+y;q,a) - \pi(x;q,a) < \frac{2y}{\varphi(q) (\log(y/q) + 5/6)}.
\end{equation}
Usage of this estimate potentially allows
us to decrease the value of $C_1$ in Theorem \ref{gq-direct},
leading to improvements of the constants $-0.59$ and $1.41$.
Once $C$ has been made explicit, \eqref{BT-estim} 
can be replaced by \eqref{BT-alternative-MV}.
Selberg \cite[vol.~2, p.~233]{SelbergCollected2}  obtained
\eqref{BT-alternative-MV}
with $2.8$ instead of 
$\frac{5}{6}$, but also did not make $C$ explicit.
\end{Rem}

\begin{prop}\label{goneplus}
For any $ \delta > 0, B \geq 1$ we have
$$
\bigl\vert w_q - w_q(q^{1+\delta}) \bigr\vert \leq \frac{6+o(1)}{q-1} ,
$$
for all primes $q$, with
the possible exception of $\ll x/(\log x)^{B}$ primes $q \leq x$.
Under the Elliott--Halberstam conjecture \ref{EHconjecture}, we have that the stronger 
conclusion $w_q - w_q(q^{1+\delta})=o(1/q)$ holds for all primes $q$, with
the possible exception of $\ll x/(\log x)^{B}$ primes $q \leq x$.
\end{prop}

\begin{proof}
We will first prove the final assertion.
\par
Set $ E(t;q) = \pi(t;q,1) - \pi(t;q,-1).$
From \eqref{beta-def} and
using the Riemann--Stieltjes integration, we have, for $z \geq y \geq 3$, that
\begin{equation}
\label{swt}
\bigl(w_q(z) - w_q(y)\bigr)\log q  =  \int_{y}^{z} \frac{\log t} {t} d(E(t;q))
=  \frac{E(t;q) \log t}{t}\Big|_{y}^{z} 
+ \int_{y}^{z} \Bigl(\frac{\log t - 1} {t^2}\Bigr) E(t;q) dt.
\end{equation}
Letting
$S(t;x) = \sum_{\substack{x/2 < q \leq x\\ q \, \text{prime}} } | E(t;q)|,$
we obtain from \eqref{swt} that
\begin{align}\label{EHappl}
\sum_{\substack{x/2 < q \leq x\\ q \, \text{prime}} }|w_q - w_q(q^{1+\delta})|\log q 
& \ll   \lim_{z \to \infty} 
\Bigl(\frac{S(z;x) \log z}{z}\Bigr) + \frac{S(q^{1+\delta};x) (1+\delta)\log x}{(x/2)^{1+\delta}}  \notag\\
&\hskip1cm +
\int_{(x/2)^{1+\delta}}^{\infty} \frac{ S(t; x)\log t} {t^2} dt \notag \\
& \ll \frac{1}{(\log x)^{B+1}} + \int_{(x/2)^{1+\delta}}^{\infty} \frac{dt}{t(\log t)^{B+1}} \ll \frac{1}{(\log x)^{B}}.
\end{align}

The penultimate step follows from the Elliott--Halberstam conjecture \ref{EHconjecture}, taking $0 < \epsilon < \delta/ (1+\delta)$ and $A = B+2$. Since $t \geq q^{1+\delta}$, we have $q \le t^{1-\epsilon}$. 
Now, if $w_q - 
w_q(q^{1+\delta}) \geq \epsilon'/q$, for any $\epsilon' > 0$ and 
for  $ \gg x/(\log x)^{B}$ primes $q$ in the range $ x/2 < q \leq x$, then
$$ 
\frac{1}{(\log x)^{B-1}}\ll \sum_{\substack{x/2 < q \leq x\\ q \, \text{prime}} }|w_q - w_q(q^{1+\delta})|\log q \ll \frac{1}{(\log x)^{B}},
$$
which is impossible.
The proof of the final assertion is completed on summing over 
the appropriate dyadic intervals $[2^{-i-1}x, 2^{-i}x]$ with
$i=0,\ldots,\lfloor{\frac{\log x}{\log 2}}\rfloor$.

If instead of assuming the Elliott--Halberstam conjecture \ref{EHconjecture},
we invoke the Bombieri--Vinogradov theorem, we obtain
\begin{equation}\label{firstest}
w_q - w_q(q^{2 +\delta/2}) = o\Bigl(\frac{1}{q}\Bigr),
\end{equation}
for any $B \geq 1$ and $\delta > 0$, and for all but $O(x/(\log x)^{B})$ primes $q \leq x$.
Note that
\begin{equation}\label{wqmax}
    \vert w_q(q^{2 +\delta/2}) - w_q(q^{1 +\delta}) \vert 
    \leq 
    \frac{1}{\log q} \max \Big\{\sum_{\substack{p \equiv 1 \pmod*{q} \\ q^{1+\delta} < p \leq q^{2+\delta/2} }} \frac{\log p}{p}, 
    ~ \sum_{\substack{p \equiv -1 \pmod*{q} \\ q^{1+\delta} < p \leq q^{2+\delta/2} }} \frac{\log p}{p} \Big\}.
\end{equation}
As before let $b\in \{-1,1\}$. Using the partial summation
formula and \cite[Corollary~3.2]{CG}, we have
\begin{align}
  \notag
   \frac{1}{\log q}\sum_{\substack{p \equiv b \pmod*{q} \\ q^{1+\delta} < p \leq q^{2+\delta/2} }} \frac{\log p}{p}
   & =  
   (2+\delta/2) \frac{\pi(q^{2+\delta/2};q,b)}{q^{2+\delta/2}} 
   - 
   (1+\delta)\frac{\pi(q^{1+\delta};q,b)}{q^{1+\delta}} 
   \\
   \notag
  & \hskip1cm 
  + \frac{1}{\log q} \int_{q^{1+\delta}}^{q^{2+\delta/2}}  \frac{(\log u -1) \pi(u;q,b)}{u^2}  du  \notag 
  \\
  \notag
   & \leq  
   \frac{4+o(1)}{(q-1) (\log q)(\log (q/2))} \int_{q^{1+\delta}}^{q^{2+\delta/2}}  \frac{\log u }{u} du + o\Bigl(\frac{1}{q}\Bigr) \notag 
   \\
   \label{hooleyapp}
   & \leq \frac{6+o(1)}{q-1},
   \end{align}
for any $B \geq 1$ and $\delta > 0$, and for all but $O(x/(\log x)^{B})$ primes $q \leq x$. 
Combination of \eqref{wqmax} and \eqref{hooleyapp} yields
\begin{equation}\label{secest}
\vert w_q(q^{2 +\delta/2}) - w_q(q^{1 +\delta}) \vert \leq  \frac{6+o(1)}{q-1},
\end{equation}
for any $B \geq 1$ and $\delta > 0$, and for all but $O(x/(\log x)^{B})$ primes $q \leq x$.

The first assertion is now obtained on combining the estimates in \eqref{firstest} and \eqref{secest}.
\end{proof}

\begin{Rem}\label{Remark-prop-53}
In \cite[p.~591]{CG} the value $4$ instead of $6$ is obtained;
this difference is due to the presence of the $\log p/\log q$ weight 
in our sum $w_q$, defined in \eqref{beta-def}, while 
they have no weight in their sum $g_q$.
In the critical range 
$q^{1+\delta}< p \le q^{2+\delta/2}$, such a weight assumes
 values between $1+\delta$ and $2+\delta/2$. 
Hence for sure the value $4$ cannot be obtained in our case and straightforward estimates
would give a value of $8+2\delta$. Instead, we were able to obtain $6$, by arguing 
as we do above to get from \eqref{wqmax} to \eqref{secest}. 
\end{Rem}

We need the next corollary to prove Proposition  \ref{imp}.

\begin{cor}\label{phoneplus}
For any $ \delta > 0, B \geq 1$ the inequality
$$
\Big\vert \kappa(q)- \frac{q-1}{2}w_q(q^{1+\delta}) \Big\vert \leq 3 +o(1)
$$
holds for every prime $q$, with
the possible exception of $\ll x/(\log x)^{B}$ primes $q \leq x$.
Under the Elliott--Halberstam conjecture \ref{EHconjecture} one can replace $3+o(1)$ with $o(1)$.
\end{cor}
\begin{proof}
The proof is immediate on combining \eqref{watowat}, Propositions \ref{nonsplit} and \ref{goneplus}.
\end{proof}

\begin{lem} \label{kaya}
For given $\lambda, \epsilon > 0$, there exists a $\delta > 0$ such that for sufficiently 
large values of $x$, there are $ \leq \lambda x/(\log x)^{2}$ primes $q \leq x$ such that $2q+1$ is prime and 
$$
\Big| w_q(q^{1+\delta}) - \frac{1}{2q+1} \Big| \geq \frac{\epsilon}{q}.
$$
\end{lem}

\begin{proof}
Given any positive integer $k$ we define 
$$
N^{\pm}_k =\#\{q : x < q \leq 2x \ \text{and}\ q, 2q+1 \ \text{and} \ kq\pm 1 \ \text{are all prime}\}.
$$
By applying both the Brun and the Selberg sieve  \cite[Theorem~5.7]{Hal} we obtain 
\begin{equation}\label{nplus}
N^{+}_k(x)  \ll \frac{x}{(\log x)^{3}} \prod_{p \mid k(k-2)} \frac{p}{p-1},  \text{~for~}k \geq 3,
\end{equation}
and
\begin{equation}\label{nminus}
N^{-}_k(x)  \ll \frac{x}{(\log x)^{3}} \prod_{p \mid k(k+2)} \frac{p}{p-1},  \text{~for~}k \geq 1.
\end{equation}
Clearly
 \begin{align}\label{muk}
\sum_{\substack{x/2 < q \leq x \\ q,\, 2q+1\, \text{primes}}} q \Big| w_q(q^{1+\delta}) - \frac{1}{2q+1} \Big| \leq 
\sum_{\substack{x/2 < q \leq x \\ q,\, 2q+1\, \text{primes}}} q \Big| w_q(q^{1+\delta}) - \frac{1+\delta}{2q+1} \Big|
+  
\delta \sum_{\substack{x/2 < q \leq x \\ q,\, 2q+1\, \text{primes}}} 1.
\end{align}
The third term is
$\ll \delta x/(\log x)^{2}$ by Lemma \ref{lem1} with $a=1$ and $t=2$.
It is clear from
 the definition of 
 $w_q(q^{1+\delta})$ that
\begin{align}\label{most}
\notag
\sum_{\substack{x/2 < q \leq x \\ q,\, 2q+1\, \text{primes}}} q \Big| w_q(q^{1+\delta}) - \frac{1+\delta}{2q+1} \Big| 
& \le \sum_{\substack{x/2 < q \leq x \\ q,\, 2q+1\, \text{primes}}} \frac{q}{\log q} 
\Bigl(   \sum_{\substack{p\le q^{1+\delta}\\ p\equiv \pm 1 \pmod*{q}}}\frac{\log p}{p} - \frac{(1+\delta) \log q}{2q+1 } 
\Bigr)\\ \notag
& \ll (1+\delta) \sum_{\substack{x/2 < q \leq x \\ q,\, 2q+1\, \text{primes}}}
\sum_{\substack{p \leq q^{1+\delta}\\ p \equiv \pm 1 \pmod*{q} \\
p \neq 2q + 1}} \frac{q}{p \mp 1}\\
\notag
& \ll (1+\delta) \Bigl(\sum_{k =2}^{x ^{\delta}}
\frac{N^+_{2k}(x)}{2k} + \sum_{k = 1}^{x ^{\delta} - 1}
\frac{N^-_{2k}(x)}{2k} \Bigr) \\
& \ll (1+\delta) \frac{x}{(\log x)^{3}} \sum_{k=1}^{x^{\delta}} \frac{1}{k} \prod_{p \mid k(k+1)} \frac{p}{p-1},
\end{align}
where the final estimate 
is a consequence 
of \eqref{nplus} and \eqref{nminus}.
By \cite[p.~328]{Gr} we have 
$$ \sum_{k=1}^{y} \prod_{p \mid k(k+1)} \frac{p}{p-1} = (C_1 + o(1)) y,$$
where 
$$C_1 = \prod_{p}\Big( 1+ \frac{2}{p(p-1)}\Big) = 3.279577\dots,$$  which by partial summation gives
$$
\sum_{k=1}^{x^{\delta}} \frac{1}{k} \prod_{p \mid k(k+1)} \frac{p}{p-1} = \delta(C_1 + o(1)) \log x.
$$
Inserting this estimate in \eqref{most}, we infer from \eqref{muk} that there exists a positive constant $C_2$ such that
\begin{equation}\label{hina}
\sum_{\substack{x/2 < q \leq x \\ q,\, 2q+1\, \text{primes}}} q \Big| w_q(q^{1+\delta}) 
- \frac{1}{2q+1} \Big| \leq \frac{C_2 \delta(2+\delta) x}{(\log x)^{2}}.
\end{equation}
Now, if for any $\lambda, \eps,\delta  > 0$, there are $ > \lambda x/(\log x)^{2}$ primes $x/2 < q \leq x$ such that $2q+1$ is prime and 
$$
\Big| w_q(q^{1+\delta}) - \frac{1}{2q+1} \Big| \geq \frac{\epsilon}{q},
$$
then
$$ 
\sum_{\substack{x/2 < q \leq x \\ q,\, 2q+1\, \text{primes}}} q \Big| w_q(q^{1+\delta}) - \frac{1}{2q+1} \Big| 
\geq 
\frac{\lambda \epsilon x}{(\log x)^{2}},
$$
which contradicts \eqref{hina},
provided that $\delta( 2+\delta) < \epsilon \lambda/C_2$.
The result follows from summing over the appropriate dyadic intervals.
\end{proof}

The next proposition shows the asymptotic distribution of $\kappa(q)$ and $r(q)$
under the assumption of the Elliott--Halberstam conjecture \ref{EHconjecture} and of the
Hardy--Littlewood conjecture for Sophie Germain primes \emph{\ref{HLsophie}}. 
Later, in Theorem \ref{cr4}, again assuming the 
Elliott--Halberstam conjecture \ref{EHconjecture}, we will see more 
refined asymptotic information about $\kappa(q)$ and how this is related with
having infinitely many prime pairs of the form $(q,mq+1)$, and of the form $(q,mq-1)$.

\begin{prop}\label{mainprop}
Assuming the Elliott--Halberstam conjecture \emph{\ref{EHconjecture}} and the Hardy--Littlewood conjecture
for Sophie Germain primes \emph{\ref{HLsophie}}, 
we have that
$$\kappa(q)=\frac{1}{4}+o(1)\quad \text{and} \quad r(q) = \frac{1}{4} + o(1),$$
for all primes $q$ belonging to some infinite sets $P_1$, respectively $P_2$, each containing $\gg x/(\log x)^{2}$ primes $q \leq x$.
\end{prop}

\begin{proof}
The Hardy--Littlewood conjecture for Sophie Germain primes, i.e., Conjecture \ref{HLsophie},
implies that there exists a constant
$C_3$ such that there are $\geq C_3 x/(\log x)^{2}$ primes $q \leq x$ for which $2q+1$ is also prime.
For any $\epsilon > 0$, Lemma \ref{kaya} confirms that there exists a $\delta > 0$ such that for $ \geq \frac{1}{2}C_3 x(\log x)^{-2}$ 
primes $q \leq x$ such that $2q+1$ is prime and 
$$
\Big| w_q(q^{1+\delta}) - \frac{1}{2q+1} \Big| \leq \frac{\epsilon}{q},
$$ 
and for any $\delta > 0$,
Corollary \ref{phoneplus} implies that

$$
\kappa(q) -  \frac{q-1}{2}w_q(q^{1+\delta}) = o(1),
$$
assuming the Elliott--Halberstam conjecture \ref{EHconjecture}. 
The assertion on $\kappa(q)$ follows on combining the last two estimates.

Assuming the Elliott--Halberstam conjecture \ref{EHconjecture} and the Hardy--Littlewood 
conjecture for Sophie Germain primes \ref{HLsophie}, the statement of Proposition 3 of  
\cite{Gr} is that for $\gg x/(\log x)^{2}$ primes $q \leq x$, one has
$qf_q =\frac{1}{2} + o(1)$ (recall Definition \ref{fqdefi}).
Using \eqref{loggie}, i.e.,
$r(q) = \frac{q-1}{2} f_q$, our second assertion is clearly equivalent to this.
\end{proof}

A variation of the above method allows one to prove  
the following rather more general proposition.

\begin{prop}\label{gen}
Assume that both the Elliott--Halberstam conjecture \ref{EHconjecture} and 
the lower bound version of the Hardy--Littlewood conjecture \ref{HLconjecture} 
are true. Let 
$\mathcal A=\{a_1,\ldots,a_s\}$ be any admissible set and choose
$b \in \{-1, 1\}$.
We have that
$$\kappa(q) = \frac{b}2 
\mu(\mathcal A)+o(1) \quad \text{and} \quad r(q) = \frac{b}2 
\mu(\mathcal A)+o(1),
$$
for all primes $q$ belonging to some infinite sets $P_3,P_4$, respectively,
each containing $\gg_{\epsilon} x/(\log x)^{s+1}$ primes $q \leq x$.
\end{prop}

\begin{proof}
By \eqref{loggie} we have $r(q) = \frac{q-1}{2} f_q$.
The proof now follows on invoking
Proposition \ref{mainprop} and  \cite[Proposition~4]{Gr}.
\end{proof}

\begin{Thm}
\label{gqrange}
Assume that both the Elliott--Halberstam conjecture \ref{EHconjecture} and 
the lower bound version of the
Hardy--Littlewood conjecture \ref{HLconjecture} are true.
Let $E$ be the sequence
$\{\kappa(q)\}_q$,
where $q$ ranges over the primes.
For any admissible set ${\mathcal A}$, 
the numbers
${\mu(\mathcal A)/2}$ and ${-\mu(\mathcal A)/2}$ are both
limit points of $E$. Moreover, the sequence $E$ has $[- \infty,\infty]$ as set of 
limit points.
\end{Thm}

\begin{proof}
The first statement directly follows from Proposition 
\ref{gen}.  Using \cite[Proposition~7]{Gr}, 
the closure of the set $ \{\mu(\mathcal{A}) : \mathcal{A}\ \text{is an admissible set}\}$ 
is seen to be $[ 0, \infty]$. Combining this with Proposition
\ref{gen}, the second statement follows too.  
\end{proof}

\section{Deeper analysis of \texorpdfstring{$\kappa(q)$}{kqCG} exploiting Croot and Granville's method}
\label{sec:CG-analogues}
In 2002, Croot and Granville \cite{CG} published a follow up of \cite{Gr}.
In it they address a more refined question on the value distribution of
$r(q)$, namely for how many $q$ the Kummer ratio $r(q)$ is very close to a prescribed real number
assuming that both the Elliott--Halberstam conjecture \ref{EHconjecture} and 
the asymptotic version of the
Hardy--Littlewood conjecture \ref{HLconjecture-asymp} are true.
Moreover, they were able to show that Kummer's conjecture 
that $r(q)=o(1)$ cannot be true assuming a different set
of hypotheses than the ones used in \cite[Theorem~1.9]{Gr}: 
they replaced the Elliott--Halberstam and the Hardy--Littlewood
conjecture for Sophie Germain primes 
(Conjectures \ref{EHconjecture} and \ref{HLsophie}, respectively) 
with a stronger Hardy--Littlewood conjecture, namely Conjecture \ref{HLconjecture-asymp}.
Given a positive rational number $r$, this then entails studying for how many $\mathcal A$ we have $\mu(\mathcal A)=r$.

Applying the ideas  used in \cite{CG} to the value
distribution of $\kappa(q)$, we obtain the following  results.

\begin{Thm}
\label{cr2}
Assume that both the Elliott--Halberstam conjecture \ref{EHconjecture} and 
the asymptotic version of the
Hardy--Littlewood conjecture \ref{HLconjecture-asymp} are true. Let $h_1(q)$ and $h_2(q)$ be any two functions tending to zero with increasing $q$.
For any  $\varpi >0$, there exists a constant $C_{\varpi} > 0$, an integer
$B(\varpi) \geq 1$ such that asymptotically there
are
$C_{\varpi} x/(\log x)^{B(\varpi)}$ primes $q\leq x$ for which 
$$ \kappa(q) >  \varpi +h_1(q).$$
Analogously, there exists a constant $D_{\varpi} > 0$, an integer $ A(\varpi) \geq 1$ such that asymptotically there are
$D_{\varpi} x/ (\log x)^{A(\varpi)}$ primes $q\leq x$, for which
$$ \kappa(q)< - \varpi + h_2(q).$$
\end{Thm}
The proof of the following corollary is not directly obvious and will be 
given later.
\begin{cor}\label{cr3}
Assume that both the Elliott--Halberstam conjecture \ref{EHconjecture} and 
the asymptotic version of the
Hardy--Littlewood conjecture \ref{HLconjecture-asymp} are true.
Then, for any fixed $\varpi > 0$, one has
$$
|\kappa(q)| <  \varpi,
$$
for all but $ O(x/ (\log x)^{E(\varpi)})$
primes $q \leq x$,  where 
$E(\varpi) = \exp(c^{\varpi})$ and $c>1$ is some constant.
\end{cor}

One can also wonder how large $\kappa(q)$ can be as a function of $q$.
This is rather outside the realm of the current tools in analytic
number theory and is discussed in
Section \ref{sec:spec}.

The proofs of Theorem \ref{cr2},  Corollary 
\ref{cr3} and Theorem \ref{cr5} 
require several ingredients and hence 
we will present them later in this section.
We start with our first ingredient.

\begin{prop}\cite[Proposition~5.1\,(i)]{CG}\label{very}
Fix an integer $n \geq 2$ and an arbitrarily small $\delta > 0$.
For all but $O(x/(\log x)^{n+ 1/2})$ of the primes
$x < q \leq 2x$,
there are no more than $n-1$ primes $< q/\delta$ which are $\equiv \pm 1 \pmod*{q}.$
\end{prop}

This result was stated in  \cite{CG} assuming the 
truth of the Elliott--Halberstam conjecture \ref{EHconjecture}; 
but, inspecting the proof, it is clear that this hypothesis is only needed 
into the second part of \cite[Proposition~5.1]{CG}, while the first one is unconditional. 
For this reason we stated Proposition \ref{very} 
without assuming the  Elliott--Halberstam conjecture \ref{EHconjecture}.

We need the next proposition in order to prove Theorem \ref{cr5}.

\begin{prop}\label{imp}
Fix an integer $n \geq 2$ and select $\delta > 0$ arbitrarily small. 
For all but $O(x/(\log x)^{n+ 1/2})$ of the primes
$x < q \leq 2x$, we have, as $q\to \infty$,
\[
\Big\vert \kappa(q) - \frac{q-1}{2} w_q \Bigl(\frac{q}\delta \Bigr) \Big \vert \leq 3 + O(\delta).
\]
In addition, under the Elliott--Halberstam conjecture \ref{EHconjecture}, 
we can replace $3 + O(\delta)$ with $O(\delta)$.
\end{prop}

\begin{proof}
The result follows from Corollary \ref{phoneplus}, provided
we can show that $w_q(q^{1+\delta})$ can be replaced
by $w_q(q/\delta)$ at the cost of an error $O(\delta)$.
Thus, it is enough to show that,
as $q$ tends to infinity,
\[ 
\bigl\vert w_q(q^{1+\delta})-w_q(\frac{q}{\delta}) \bigr \vert
\le
\frac{1}{\log q}\sum_{\substack{q/\delta < p \le q^{1+\delta} \\ p\equiv \pm 1 \pmod*{q}}}\frac{\log p}{p}
\ll \frac{\delta}{q},
\]
for all but $O(x/(\log x)^{n+ 1})$ of the primes
$x < q \leq 2x$. Since
\[ 
\frac{1}{\log q}\sum_{\substack{q/\delta < p \le q^{1+\delta} \\ p\equiv \pm 1 \pmod*{q}}}\frac{\log p}{p} \le
(1+\delta) \sum_{\substack{q/\delta < p \le q^{1+\delta} \\ p\equiv \pm 1 \pmod*{q}}}\frac{1}{p},
\]
it suffices to show that the right hand side is $\ll \delta /q$
for all but $O(x/(\log x)^{n+1})$ of the primes
$x < q \leq 2x$. 
{But this is precisely what is proved in \cite[Proposition~5.1\,(ii)]{CG}.}  
\end{proof}

Before proceeding further with the proofs of our results,
we need a deeper study of $w_q(q/\delta)$.
Hence, recalling that $b\in\{-1,1\}$, we proceed now to relate
$w_q(q/\delta)$ with the measure of the set 
$M=M(b,\delta)$ of integers $m$ for which there exists a prime 
$p$ of the form $p=mq+b$ and  with $0< bm< 2/\delta$.
Clearly every prime $p< q/\delta$ satisfying $p \equiv \pm 1 \pmod*{q}$ 
is either of the form $mq+1$ with
$0 < m < 2/\delta$, or of the form $-mq - 1$ with $-2/\delta < m < 0$.
Recall that
\begin{equation}
 \quad w_q(x)  
= \frac{1}{\log q} \Bigl(\sum_{\substack{p\le x\\ p\equiv 1 \pmod*{q}}}\!\! \frac{\log p}{p}
-\sum_{\substack{p\le x\\p\equiv -1 \pmod*{q}}} \!\! \frac{\log p}{p}\Bigr)\notag.
\end{equation}
We have, as $ q \to \infty$,
\begin{align}
\notag
 \notag (q-1)w_q\Bigl(\frac{q}{\delta}\Bigr) & \leq 
\frac{q-1}{\log q} \Bigl(\sum_{\substack{0 < m < 2/\delta\\ m\in M}}\frac{\log (q/\delta)}{mq+1}
+\sum_{\substack{-2/\delta < m < 0\\ m\in M}} \frac{\log (2q-1)}{mq+1}\Bigr)
= \sum_{m \in M}  \frac{1}{m} + o(1) 
\end{align}
and 
\begin{align}
\notag
 \notag (q-1)w_q\Bigl(\frac{q}{\delta}\Bigr) & \geq 
\frac{q-1}{\log q} \Bigl(\sum_{\substack{0 < m < 2/\delta\\ m\in M}}\frac{\log (2q+1)}{mq+1}
+\sum_{\substack{-2/\delta < m < 0\\ m\in M}} \frac{\log (q/\delta)}{mq+1}\Bigr)
= \sum_{m \in M}  \frac{1}{m} + o(1),
\end{align}
and hence
\begin{align}
\notag(q-1)w_q\Bigl(\frac{q}{\delta}\Bigr)
& =  \sum_{m \in M}  \frac{1}{m} + o(1).
\end{align}

Using this estimate and Proposition \ref{imp}, for fixed $n \geq 2$ and $\delta > 0$ arbitrary small, for all but $O(x/(\log x)^{n+ 1/2})$ of the primes
$x < q \leq 2x$, we have
\begin{equation} \label{funda}
\Big \vert \kappa(q) - \sum_{m \in M} \frac{1}{2m} \Big \vert \leq 3 + O(\delta),\quad 
\textrm{as}\ q\to \infty.
\end{equation} 
In addition, assuming the Elliott--Halberstam conjecture \ref{EHconjecture}, for fixed $n \geq 2$ and $\delta > 0$ arbitrary small, for all but $O(x/(\log x)^{n+ 1/2})$ of the primes
$x < q \leq 2x$, we have
\begin{equation} \label{fund}
\kappa(q)  = \sum_{m \in M} \frac{1}{2m} + O(\delta),
\quad 
\textrm{as}\ q\to \infty.
\end{equation} 

We are now ready to prove the following result which also expresses, under the assumption
of the Elliott--Halberstam conjecture \ref{EHconjecture}, that the asymptotic behavior $\kappa(q) \sim b/(2m)$
as $q$ tends to infinity,
where  $b\in\{-1,1\}$ and $m>0$ even are both fixed integers,  
is equivalent to the existence of infinitely many primes of the form $mq + b$.

\begin{Thm}\label{cr4}
Assume that the Elliott--Halberstam conjecture \ref{EHconjecture} holds,
let $m$ be an even positive integer, and
$b \in \{-1, 1\}$. There are $\gg_m x/ (\log x)^2$ primes $ q \leq x$ for which $mq + b $ is also prime if
and only if $$\kappa(q) \sim \frac{b}{2m}, 
\quad \textrm{as} \ q\to \infty.$$
\end{Thm}

\begin{proof}
We follow the proof of Theorem 1.7 in \cite{CG},
and take $n=2$ in Proposition \ref{very}. This result guarantees that
for small $\delta > 0$, for all but $O( x /(\log x)^2)$ primes $x < q \leq 2x$, there is at most 
one prime $p_0$ of the form  $p_0=m_0q + b$ with 
$0 < m_0 < 2/\delta$ if $b=1$, or with $-2/\delta < m_0 < 0$ if $b=-1$.
In other words, the two conditions can be grouped together into $0<bm_0<2/\delta$.
Let $\Delta_0=1$ if $p_0$ exists and $\Delta_0=0$ otherwise.
Recalling \eqref{fund} and the definition of the set $M$, we have
$$
\kappa(q)  = \sum_{m \in M} \frac{1}{2m} + O(\delta) 
= 
\sum_{\substack{0 < bm < 2/\delta \\ mq+b ~ \text{is prime}}} \frac{b}{2m} + O(\delta)
=
\frac{b\Delta_0}{2m_0} + O(\delta),
$$ 
as $q$ tends to infinity.
Letting  $\delta$ tend to zero and $x$ 
to infinity, we have the result.
\end{proof}

Theorem \ref{cr4} shows a conjectural symmetrical distribution
of $\kappa(q)$ with respect to $b$ for primes congruent to $b\in\{-1,1\}$ modulo $q$.
This is consistent with the distribution of the values of $\kappa(q)$ shown 
into the histograms of Section \ref{sec:numerical};
on the other hand, the conjectural symmetrical distribution of the values of  $r(q)$ 
is provided by \cite[Theorem~1.7]{CG} (the source of
inspiration for Theorem \ref{cr4}).

We are now in the position to prove Theorems \ref{cr5} and \ref{cr2}, and Corollary \ref{cr3}.

\begin{proof}[Proof of Theorem \ref{cr5}]
If 
$\kappa(q) > 3 + \epsilon$ for $\gg x/ (\log x)^A $ primes $ q \leq x$ for all
sufficiently large $x$, for some $A\ge 1$, then using Proposition \ref{imp}, for
sufficiently small positive $\delta$, we have 
$$ 
0< 2\epsilon  +O(\delta) \leq   (q-1) w_q(q/\delta).
$$
It follows from this estimate that there exists an integer 
$0< m < 2/\delta$ for which there are
$\gg x/ (\log x)^A $ primes $q \leq x$ with $mq + 1$ also prime.
If 
$-\kappa(q) > 3 + \epsilon$ for $\gg x/ (\log x)^A $ primes $ q \leq x$ for all
sufficiently large $x$, for some $A\ge 1$, then using Proposition \ref{imp}, for sufficiently
small $\delta > 0$, we have 
$$ 
(q-1) w_q(q/\delta) \leq - 2\epsilon  + O(\delta)< 0.
$$
We infer that there exists an integer 
$0<m < 2/\delta$ for which there are
$\gg x/ (\log x)^A $ primes $q \leq x$ for which $mq - 1$ is also prime.
\end{proof}

\begin{proof}[Proof of Theorem \ref{cr2}]
The goal is to determine how often $b \kappa(q) >  b\varpi +o(1)$ for $b\in\{-1,1\}$.
Following the proof of \cite[Theorem~1.4]{CG}, we select $n= B(b,\varpi)$ and $\delta$ sufficiently small in Propositions \ref{very} and \ref{imp}.  
Assume $b\kappa(q) >  b\varpi +o(1)$ for some set of primes $q$ in $(x, 2x].$ 
Using \eqref{fund}, for all but $O(x/(\log x)^{B(b,\varpi)})$ of the primes $x < q \leq 2x$, we have no more than $B(b,\varpi)-1$ primes
of the form $ mq + b$ with $0 < bm < 2/\delta$  and
\begin{equation}\label{est2}
    \sum_{\substack{0 < bm < 2/\delta \\ mq+b ~ \text{is prime}}} \frac{1}{m} \geq 2\varpi.
\end{equation}
By the Hardy--Littlewood conjecture \ref{HLconjecture-asymp},
this set of primes has cardinality $\sim C_{b,\varpi} x/(\log x)^{B(b,\varpi)}$ 
for some constant $C_{b,\varpi}>0$. This completes the proof.
\end{proof}

\begin{proof}[Proof of Corollary \ref{cr3}]

If $ 0 < 2 \varpi \leq 1$, then $E(\varpi) = 2$. If $E(\varpi) = N$, then 
$ 2 \varpi \leq \sum_{n=1}^{N} 1/n \leq \log N+1.$ This implies that
$ E(\varpi) = N \geq \exp(2\varpi -1)$ and hence $E(\varpi)$ 
is a non-decreasing function. By the partial summation formula, we have
$$
\sum_{\substack{0 < bm < 2/\delta \\ mq+b ~ \text{is prime}}} \frac{1}{m} \ll \log \log  E(\varpi).
$$
Combining this estimate with \eqref{est2}, we have
$ 2 \varpi \ll \log \log E(\varpi).$ 
So $E(\varpi) \geq \exp(c^{\varpi})$ for some constant $c>1$.
\end{proof}

It remains to prove  Theorems \ref{cr1} and \ref{noEHproof}.

\begin{proof}[Proof of Theorem \ref{cr1}]
Let us first consider the case when $\varpi$ is irrational.
Take $n+1/2 > A$ and $\delta >0$ sufficiently small in
\eqref{fund} and Lemma \ref{CGlem}\,(ii). 
It follows that for all, but at most $O(x/(\log x)^{A})$ of the primes
$x < q \leq 2x$, we have
\begin{equation}\label{eq133}
\kappa(q)  = \sum_{m \in M} \frac{1}{2m} + O(\delta),
\end{equation} 
as $q$ tends to infinity. If $\kappa(q) \sim \varpi$, as $q$ 
runs through this set of primes, then \eqref{eq133} implies that
$$
\Big \vert 2 \varpi -  \sum_{m \in M} \frac{1}{m} \Big \vert \ll \delta, 
$$
contradicting Lemma \ref{CGlem}(ii).

The proof for the rational case can be carried out by mimicking  the proof of \cite[Theorem~1.3\,(i)]{CG}. 
Here one simply replaces the variable $p$ with $q$ and the expression $pf_p$ with $\kappa(q).$
\end{proof}

\begin{proof}[Proof of Theorem \ref{noEHproof}]
As stated in Theorem \ref{noEHproof}, let $M$ be a set of $k$ distinct positive integers with 
$\Sigma_{m \in M} 1/m > 6$, such that there are 
$\gg x/(\log x)^{k+1}$ primes $q\leq x$ for which $mq \pm 1$ is also prime for every $m \in M$
(in fact, the existence of such sets depends on Theorem \ref{unbounded} and Conjecture \ref{HLconjecture}).
Then, taking $n=k+1$ in equation \eqref{funda}, for sufficiently small $\epsilon>0$, we obtain
$$
\vert \kappa(q) \vert > \epsilon ,
$$
for $\gg x/(\log x)^{k+1}$ primes $q\leq x$. 
\end{proof}

We end this section by elaborating on the link of these results with those
in \cite{CG} about $r(q)$.
Recall Definition \ref{fqdefi} of $f_q$. From \cite[\S6]{CG}, 
under the same hypothesis and for the same set of primes, we have, as $q\to \infty$,
$$
q f_q = \sum_{m \in M} \frac{1}{m} + O(\delta).
$$ 
Combining this estimate with \eqref{fund},
it is clear that assuming the 
Elliott--Halberstam conjecture \ref{EHconjecture}, for fixed $n \geq 2$, 
for all but $O(x/(\log x)^{n+ 1/2})$ of the primes
$x < q \leq 2x$, we have
\begin{equation}\label{kapparzer0}
   r(q) - \kappa(q) = o(1),
\end{equation}
in line with
Figure \ref{fig6}.

\section{Speculations}
\label{sec:spec}
It seems very plausible that $r(q)$ is close to $\kappa(q)$ for a
wide range of prime numbers $q$. We have not been able to
prove a result in this direction beyond the modest and
conditional Propositions \ref{mainprop}-\ref{gen} and equation \eqref{kapparzer0}.
But, following the heuristic argument in  \cite[\S9]{Gr}, we 
can obtain some conjectural result. First, we have to recall that 
Granville assumed that 
two very strong hypotheses on the size of the error
term for the Prime Number Theorem in arithmetic progressions
and on a Brun--Titchmarsh type estimate both hold; namely
 that for $a\in\{-1,1\}$ one has 
\begin{equation}\label{strong1}
\pi(x;q,a)-\frac{\pi(x)}{\varphi(q)} 
\ll \sqrt{\frac{x}{q}}\exp\bigl(\sqrt{\log x}\,\bigr) \quad\textrm{for}\ x> q,
\end{equation}
which can be seen as a corrected version of a conjecture of Montgomery,
see \cite[eq.~(15.9), p.~136]{Montgomery} and Friedlander--Granville
\cite{FG1989},
and
\begin{equation}\label{strong2}
\pi(x;q,a) \ll \frac{x}{q\log x} \quad\textrm{for}\ x> q(\log q)^3.
\end{equation}

Using \eqref{strong1} for $x> q\exp\bigl(3\sqrt{\log q}\, \bigr)$,
\eqref{strong2} for $q(\log q)^3< x \le q\exp\bigl(3\sqrt{\log q}\, \bigr)$
and, in the remaining range, a conjectural form of
the Brun--Titchmarsh Classical Theorem \ref{BT-thm} 
in which the constant $2$ is replaced by $1+\epsilon$,
Granville  \cite[\S9]{Gr} was led to the following speculation.

\begin{Speculation}
\label{Speculation1}
We have, as $q$ tends to infinity,
\begin{equation*}
\bigl(-1+o(1)\bigr) \log_3 q \le  2r(q)\le \bigl(1+o(1)\bigr) \log_3 q.
\end{equation*}
Both bounds are best possible in the sense that they are attained
for some infinite sequence of primes $q$.
\end{Speculation}
We will now indicate what the same line of reasoning gives for
$\kappa(q)$.
Proposition \ref{nonsplit} implies that $qv_q \ll 1$ as $q$ tends to infinity.
Using 
the partial summation formula and a conjectural form of
the Brun--Titchmarsh Classical Theorem \ref{BT-thm} 
in which the constant $2$ is replaced by $1+\epsilon$,
we deduce that 
$$
\Big \vert \frac{q-1}{2} w_q( q (\log q)^3) \Big \vert \leq \Bigl(\frac{1}{2} +\epsilon \Bigr)\log_3 q + O(1),
$$ 
for any arbitrary 
small $\epsilon >0$. Following  the argument in \cite[\S9]{Gr}, we can easily 
obtain that $q(w_q - w_q( q (\log q)^3)) \ll 1$, 
assuming conjectures \eqref{strong1} and \eqref{strong2}.
Summarizing, we arrive at the following speculation analogous to Speculation \ref{Speculation1}.
\begin{Speculation}
\label{spec2}
We have, as $q$ tends to infinity,
\begin{equation*}
\bigl(-1+o(1)\bigr) \log_3 q \le 2\kappa(q) \le \bigl(1+o(1)\bigr) 
\log_3 q.
\end{equation*}
Both bounds are best possible in the sense that they are attained
for some infinite sequence of primes $q$.
\end{Speculation}

These heuristics in the setting of $\gamma_q$ led
Moree \cite{Msurvey} to the following speculation.
\begin{Speculation}[Moree]
We have, as $q$ tends to infinity,
\begin{equation}\label{FLMbound}
\notag
\frac{\gamma_q}{\log q}\ge \bigl(-1+o(1)\bigr) \log_3 q.
\end{equation}
This bound is best possible in the sense that there exists an infinite 
sequence of primes for which it is attained.
\end{Speculation}

We explain now how such an heuristics works for $r(q) - \kappa(q)$.
Evaluating this quantity means, as is 
clear from looking at \eqref{watowat}-\eqref{beta-def} and the  proof of
Theorem 1 in \cite{paper1-rq}, amounts to estimating 
\begin{align}
\notag
\Bigl(\sum_{\substack{m \geq 1\\ p^{m} \equiv 1 \pmod*{q}}} 
\frac{1}{mp^m}
&
-
\sum_{\substack{m \geq 1\\ p^{m} \equiv -1 \pmod*{q} }} \!\! \frac{1}{mp^m}\Bigr)
- 
\frac{1}{\log q} \Bigl(\sum_{\substack{m \geq 1\\ p^{m} \equiv 1 \pmod*{q}}} 
\frac{\log p}{p^m}-\sum_{\substack{m \geq 1\\ p^{m} \equiv -1 \pmod*{q} }} \!\! \frac{\log p}{p^m}\Bigr)
\\ &
\label{double-difference}
=
\sum_{\substack{m \geq 1\\ p^{m} \equiv 1 \pmod*{q}}} 
\frac{1}{p^m} \Bigl(\frac{1}{m}-\frac{\log p}{\log q}\Bigr)
-
\sum_{\substack{m \geq 1\\ p^{m} \equiv -1 \pmod*{q}}} 
\frac{1}{p^m} \Bigl(\frac{1}{m}-\frac{\log p}{\log q}\Bigr).
\end{align}
This we studied already, but without the 
weighting factor $w(p,q,m)=\frac{1}{m}-\frac{\log p}{\log q}$.

We now prove that the prime powers contribute $o(1/q)$ to \eqref{double-difference} (for them $m\ge 2$).
Using \eqref{first-interval-estim}-\eqref{second-interval} 
and \cite[eq.~(14)-(15)]{paper1-rq}, the contribution to \eqref{double-difference} of the ranges
$p^m > q  (\log q)^{2}$, $m\ge 2$, is $o(1/q)$.  We are left with
estimating what happens in the range $p^m \le q  (\log q)^{2}$, with $m\ge 2$.
Letting $b\in\{-1,1\}$ and recalling that in \cite[Lemma 2]{Gr}
it is proved  that
\[
\sum_{\substack{m \geq 2; \, p^m \leq q  (\log q)^{2} \\ p^m\equiv b\pmod*{q}}}
\frac{1}{mp^m}  
\ll
\frac{1}{q},
\]
we see that it suffices to show that $w(p,q,m) = o(1)$ for 
$p^m \le q  (\log q)^{2}$, $m\ge 2$.
If $q+1 \le p^m \leq q  (\log q)^{2}$, then
\[
\vert w(p,q,m) \vert
\le
\frac{2\log_2 q}{m\log q},
\]
while,  in the remaining case  $p^m=q-1$, $m\ge 2$, we must have $p=2$ and  
\[
w(p,q,m) 
\le
\frac{1}{m(q-1)\log q}.
\]
Summarizing, we have proved that
\[
\sum_{\substack{m \geq 2; \, p^m \leq q  (\log q)^{2} \\ p^m\equiv b \pmod*{q}}}
\frac{\vert w(p,q,m)\vert}{p^m} 
\ll
\frac{\log_2 q}{q\log q} 
= o \Bigl(\frac{1}{q} \Bigr),
\]
and this implies that the contribution of the prime powers to \eqref{double-difference} is $o(1/q)$
as $q$ tends to infinity.

It remains to deal with the case $m=1$ in \eqref{double-difference}.
It is easy to remark that $w(p,q,1) = o(1)$ for $2q-1 \le p \le q\exp{(c f(q))}$, $c>0$,
where $f(q) = o(\log q)$ as $q$ tends to infinity.  
Exploiting this we can avoid the use of 
\eqref{strong2} and replace \eqref{strong1}
with the following hypothesis which, both as regards the quality of the estimate and the uniformity, is slightly weaker than \eqref{strong1}: 

\noindent
letting $f(x) = \frac{\log x}{(\log_2 x)^2}$ 
and  $t_1 = q \exp(3f(q))$, we assume, for $a\in\{-1,1\}$, that
\begin{equation}\label{strong1-relaxed}
\pi(x;q,a)-\frac{\pi(x)}{\varphi(q)} 
\ll 
\sqrt{\frac{x}{q}}\exp\bigl(f(x)\bigr) \quad\textrm{for}\ x> t_1.
\end{equation}
Recalling $b\in\{-1,1\}$, by using the conjectural estimate \eqref{strong1-relaxed} 
for $p > t_1$ and the Riemann--Stieltjes integration (see for example, equation \eqref{swt}), we have
\begin{equation}\label{inso2}
\sum_{\substack{p > t_1\\p\equiv  b \pmod*{q}}}  \frac{w(p,q,1)}{p}
= o\Bigl(\frac{1}{q}\Bigr).
\end{equation}

Combining equations \eqref{double-difference} and \eqref{inso2}
and the fact that the prime powers contribution is $o(1/q)$,
we obtain that
\begin{equation}
\label{granville-plus-weight}
r(q)-\kappa(q)
=  \frac{q-1}{2} \Bigl(\sum_{\substack{p\le t_1\\ p\equiv 1 \pmod*{q}}} \frac{w(p,q,1)}{p} 
-\sum_{\substack{p\le t_1\\p\equiv -1 \pmod*{q}}}  \frac{w(p,q,1)}{p}
\Bigr)
+ o(1).
\end{equation}
We argue as in the displayed equation before eq.~(13) on \cite[p.~335]{Gr} 
and use the  Brun--Titchmarsh Classical Theorem \ref{BT-thm} to conclude that, for $b\in\{-1,1\}$,
\begin{equation}\label{BTeqn13}
\sum_{\substack{2q-1 \le p\le t_1\\ p\equiv b \pmod*{q}}} \frac{1}{p}
\ll  
\frac{\log_2 (t_1/q)}{q} 
\ll
\frac{\log f(q)}{q},
\end{equation}
Using that
\[
w(p,q,1)  
\ll
\frac{f(q)}{\log q}
\quad
\textrm{for}
\quad
2q-1 \le p\le t_1,
\]
we immediately obtain from \eqref{granville-plus-weight}-\eqref{BTeqn13} that
\[
r(q)-\kappa(q) \ll 
\frac{f(q)\log f(q)}{\log q} +o (1)
\ll
 \frac{1}{\log_2 q} +o (1)
 = o(1).
\]
This motivates us to make the following speculation.

\begin{Speculation}
\label{specmain}
We have, as $q$ tends to infinity,
\begin{equation*}
r(q)\sim\kappa(q).
\end{equation*}
\end{Speculation}

It is likely an easier question to evaluate the average of $r(q)-\kappa(q)$ 
and $|r(q)-\kappa(q)|$,  which, according to the previously described heuristics, is zero again.
Before stating the speculations, we recall that $\pi^*(Q)$ counts the prime numbers between $Q/2$ and $Q.$
\begin{Speculation}
\label{specmain-avg}
We have, as $Q$ tends to infinity,
\begin{equation*}
\sum_{Q/2 <q \le Q} \vert r(q)-\kappa(q)\vert= o(\pi^*(Q)).
\end{equation*}
\end{Speculation}
Figure \ref{fig6} clearly supports Speculations \ref{specmain} and \ref{specmain-avg}.

We now show that if the strong form of the Elliott--Halberstam conjecture \ref{EH2} holds,
then $\kappa(q) = o(1)$ on average.

\begin{Thm}\label{kappaqzeroEH2}
Assuming the strong form of the Elliott--Halberstam conjecture \ref{EH2}, we have, as $Q$ tends to infinity,
\begin{equation}\notag
 \sum_{Q/2 < q \leq Q} \vert \kappa(q) \vert
 = o(\pi^*(Q)).
\end{equation}
\end{Thm}

\begin{proof}
Letting 
$t_0 = Q \exp\bigl(\frac{\log Q}{\sqrt{\log_2 Q}}\bigr)$, we have
\begin{align}
\label{threesums}
\sum_{Q/2 < q \leq Q} \vert \kappa(q) \vert 
& 
= \frac{q-1}{2}  \sum_{Q/2 < q \leq Q} \vert v_q + w_q \vert \notag 
\\
& \leq  \sum_{Q/2 < q \leq Q} q\vert v_q \vert 
+ 
\sum_{Q/2 < q \leq Q} q\vert w_q - w_q(t_0)\vert + \sum_{Q/2 < q \leq Q} q\vert w_q(t_0) \vert.
\end{align}
By replacing the Elliott--Halberstam conjecture \ref{EHconjecture} with 
its strong version (stated in Conjecture \ref{EH2}) into equation \eqref{EHappl}, 
we have, as $Q$ tends to infinity,
 \begin{equation}\label{secondsum2}
     \sum_{Q/2 < q \leq Q} q\vert  w_q - w_q(t_0)\vert = o(\pi^*(Q)).
 \end{equation}
Using the Cauchy-Schwarz inequality, we have, as $Q$ tends to infinity,
\begin{align}\label{noEH2}
\Bigl(\sum_{Q/2 < q \leq Q} q\vert &w_q(t_0) \vert \Bigr)^2
 \leq 
    \pi^*(Q) \sum_{Q/2 < q \leq Q} q^2\vert w_q(t_0) \vert ^2 
    \notag \\
    & \ll \pi^*(Q) \sum_{k,~ l < t_0/Q} \frac{1}{kl} ~ \# \{Q/2 < q \leq Q: q, kq\pm 1 ~\text{and} ~ lq\pm1 ~ \text{are all prime} \} \notag \\
    & \ll \frac{\pi^*(Q) Q}{(\log Q) \log_2 Q} = o((\pi^*(Q))^2),
\end{align}
The final estimate in \eqref{noEH2} follows from 
\cite[p.~336]{Gr}. The proof is completed on combining \eqref{threesums}-\eqref{noEH2} and using Corollary \ref{firstsum1}. 
\end{proof}

We recall that, assuming that the strong form of the Elliott--Halberstam conjecture \ref{EH2} holds 
in the larger interval $t \geq u \exp{\bigl(\sqrt{\log u}\,\bigr)}$,
it was proved in Granville \cite[\S10, p.~ 336]{Gr} that \begin{equation}\label{rqzeroEH2}
\sum_{Q/2 <q\le Q} \vert r(q)\vert= o(\pi^*(Q)),
\end{equation}
as $Q$ tends to infinity.
Inspecting his proof, it not hard to see that it holds even if one assumes the Elliott--Halberstam conjecture \ref{EH2}
in the range
$t \geq u \exp\bigl(\frac{\log u}{\sqrt{\log_2 u}}\bigr)$, as we did.

It is then clear that Speculation \ref{specmain-avg} follows from \eqref{rqzeroEH2} 
and Theorem \ref{kappaqzeroEH2} (under the strong form of the Elliott--Halberstam conjecture \ref{EH2}).
 
Table \ref{table1} gives a summary of the roles of conjectures and theorems used in studying the estimates for $\kappa(q)$.

\begin{table}[htp]
 \scalebox{0.85}{
 \renewcommand{\arraystretch}{1.8}
 	\begin{tabular}{|c|c|c|}
 		\hline
 		Range of application  & Main ingredient 
 		& \text{Reference}\\
 		\hline
 		$ p > q^{1+\delta}$ & Elliott--Halberstam conjecture \ref{EHconjecture} & Proposition \ref{goneplus}\\ \hline
 		$ p > q^{2+\delta}$ & Bombieri--Vinogradov theorem \ref{BV-type-ineq} & Proposition \ref{goneplus} \\ 
 		\hline
 		 $ q^{1+\delta}< p \leq q^{2+\delta}$ & Hooley's Brun--Titchmarsh ``almost all'' theorem \cite[Cor.~3.2]{CG} & Proposition \ref{goneplus}\\ 
 	  \hline
 	  $ q(\log q)^3 < p \le q^{1+\delta}$ & Hooley's 
 	  Brun--Titchmarsh ``almost all'' theorem \cite[Cor.~3.2]{CG} & Proposition \ref{imp} \\
 	  \hline
 		$2q-1 \le p \leq q \exp(3\frac{\log q}{(\log_2 q)^2})$ & Brun--Titchmarsh Classical Theorem \ref{BT-thm} & Speculation \ref{specmain}\\ \hline
 		 $ p > q \exp(3\frac{\log q}{(\log_2 q)^2})$ & Conjecture \eqref{strong1-relaxed} & Speculation \ref{specmain}\\ \hline
 		  $p >  Q \exp\bigl(\frac{\log Q}{\sqrt{\log_2 Q}}\bigr)$ &  Strong form of the Elliott--Halberstam conjecture \ref{EH2} & Theorem \ref{kappaqzeroEH2}\\ \hline
 		  	 $p\leq  Q \exp\bigl(\frac{\log Q}{\sqrt{\log_2 Q}}\bigr)$ & Cauchy-Schwarz inequality and sievings  & Theorem \ref{kappaqzeroEH2}\\ \hline
 	\end{tabular}
 } 
 \caption{\label{table1}  
 Theorems or conjectures used in this paper and where they are used.
 }
\end{table}

\section{Comments on the plots and on the histograms}
\label{sec:numerical}

\newcommand{\pariversion}{PARI/GP (v.~2.15.4)}
\newcommand{\pythonversion}{Python (v.~3.11.6)}
\newcommand{\matplotlibversion}{Matplotlib (v.~3.8.2)}
\newcommand{\pandasversion}{Pandas (v.~2.1.3)}
\newcommand{\fftwversion}{FFTW (v.~3.3.10)}
\newcommand{\ubuntuversion}{Ubuntu~22.04.3~LTS} 
\newcommand{\clusteraddress}{\url{https://hpc.math.unipd.it}}
\newcommand{\capriaddress}{\url{https://capri.dei.unipd.it}}
\newcommand{\optiplexmachine}{Dell OptiPlex-3050, equipped with an Intel i5-7500 processor, 3.40GHz, 
32 GB of RAM}

\newcommand{\bound}{10^7}

The actual values of $r(q)$, $\kappa(q)$, and of all the other 
relevant quantities presented in the 
plots and histograms presented below
were obtained for every odd prime $q$ up to $\bound$
using the FFTW \cite{FrigoJ2005} software library
set to work with the \textit{long double precision} (80 bits). 
The algorithms used are presented in 
\cite{Languasco2021a}-\cite{LanguascoR2021} and in
\cite{paper1-rq}\footnote{The results presented in Table \ref{table2} were obtained by implementing
the very same algorithms in the scripting language of Pari/GP.}

The results were then collected in some \emph{comma-separated values} (csv) files and then
all the plots and the histograms were
obtained running on such stored data some suitable designed scripts written
using \pythonversion\ and making use of the packages \pandasversion\ and \matplotlibversion.
The normal function used in the histograms is defined as
\[
\mathcal{N}(x,\mu,\sigma) = \frac{1}{\sigma \sqrt{2\pi}} \exp\Bigl(- \frac{(x-\mu)^2}{2\sigma^2} \Bigr),
\]
where $\mu$, $\sigma$ respectively represent the mean and the
standard deviation of the plotted data.

To better  demonstrate the similarities between the behaviour of $\kappa(q)$
and $r(q)$,  we insert here also some plots and histograms about
$r(q)$ that were already presented in \cite{paper1-rq}.
The first important remark is that Figures \ref{fig1}-\ref{fig2}
show that both $r(q)$ and $\kappa(q)$ have a symmetrical distribution having average equal to $0$;
this symmetry is only partially explained by the known theoretical results.
We also remark that the ``concentration'' of the computed data
around the values $\pm1/4$ depends on 
the contribution of the primes $q$ such that
$2q\pm1$ are primes, see Figure \ref{fig3}. 
Indeed, we would expect a concentration to occur for primes $q$ related
to admissible sets $\mathcal{A}$ that are of  cardinality as small as possible (leading to 
relatively many corresponding primes $q$) and for which
$\mu(\mathcal{A})$ is large. Clearly the best candidate here is $\mathcal{A}=\{2\}$.
This is somehow predicted
by a result stated
in \cite{Gr}, assuming there are 
at least $\gg x/(\log x)^2$ Sophie Germain primes.
In Figure \ref{fig4} one sees
that the other value concentrations are less pronounced, for example 
those around the values $\pm1/8$ (which are due the contribution of the primes $q$ such that
$4q\pm1$ are primes). In fact the global effect of such ``secondary'' concentrations close to $\pm1/(2m)$, $m\ge4$,
is the reason why the central ``peak'' 
around the origin is steeper 
in comparison to the normal curve.

This is confirmed by the observation that the histograms show
distributions for $\kappa(q)$ that are closer
to the normal one  if only the contribution of the primes $q$ such that
$2q\pm 1$ and  $4q\pm 1$ are all four composite are taking 
into account, see  Figure \ref{fig5};
a similar phenomenon is observed for $r(q)$ too.
Finally, the behaviour 
of $r(q)-\kappa(q)$ shown in Figure \ref{fig6} supports the remarks 
about $r(q)$ being close to $\kappa(q)$ made in  Section \ref{sec:spec}.

Our computer programs developed to compute the values of $\kappa(q)$ (already used in \cite{Languasco2021a,LanguascoR2021}), the programs designed to analyse these values,
and the results obtained are available on the webpage 
\url{https://www.math.unipd.it/~languasc/kq-comput.html}

\medskip
\noindent \textbf{Acknowledgment}. 
Part of the work was done during the postdocs and multiple visits of the
first, fourth and fifth author
at the Max Planck Institute for Mathematics (MPIM) under the
mentorship of third author. They thank MPIM for the invitations, 
the hospitality of the staff and the excellent working conditions. The fourth author is supported by the Austrian Science Fund (FWF): P 35863-N.
The fifth author wishes to thank the INI and LMS for the financial support.
The computational work was carried out on machines
of the cluster located at the Dipartimento di Matematica ``Tullio Levi-Civita'' of 
the University of Padova, see \url{https://hpc.math.unipd.it}. 
The authors are grateful for having had such 
good computing facilities 
at their disposal. 

\begin{table}[htp]
\scalebox{0.5875}{
\begin{tabular}{|c|r|}
\hline
$q$  & $\kappa(q)$\phantom{01234567890123456} \\ \hline
$3$ &   $-0.335224373301549299654816272011\dotsc$ \\ \hline
$5$ &   $-0.196173501603961235980346083986\dotsc$ \\ \hline
$7$ &   $-0.067031676838182540620959092690\dotsc$ \\ \hline
$11$ &  $ 0.102860286443295216553377535220\dotsc$ \\ \hline
$13$ &  $ 0.112609625562397183644619140260\dotsc$ \\ \hline
$17$ &  $-0.123599726784128717636050456728\dotsc$ \\ \hline
$19$ &  $-0.483413418826496988129263309946\dotsc$ \\ \hline
$23$ &  $ 0.304501778642862588117192982929\dotsc$ \\ \hline
$29$ &  $ 0.202174400853472811939673145368\dotsc$ \\ \hline
$31$ &  $-0.167096940024034810849977929755\dotsc$ \\ \hline
$37$ &  $-0.116722996141948523834462113315\dotsc$ \\ \hline
$41$ &  $-0.019078448817730487665286522915\dotsc$ \\ \hline
$43$ &  $-0.001063229955915806485603063750\dotsc$ \\ \hline
$47$ &  $-0.004178628154399978902142786696\dotsc$ \\ \hline
$53$ &  $-0.068102981731655092877815133182\dotsc$ \\ \hline
$59$ &  $ 0.077561177537401718955224637346\dotsc$ \\ \hline
$61$ &  $-0.085600677063319552995034524968\dotsc$ \\ \hline
$67$ &  $ 0.048195303630233224012827541336\dotsc$ \\ \hline
$71$ &  $-0.054112902001196238161718707411\dotsc$ \\ \hline
$73$ &  $ 0.349480310713616884139356667045\dotsc$ \\ \hline
$79$ &  $-0.156621195958631413503001970935\dotsc$ \\ \hline
$83$ &  $ 0.232654418724714289189129416232\dotsc$ \\ \hline
$89$ &  $ 0.286113111172147083063576959599\dotsc$ \\ \hline
$97$ &  $-0.097406509380448235796556853350\dotsc$ \\ \hline
$101$ & $ 0.137944974577464046144537684905\dotsc$ \\ \hline
$103$ & $ 0.083885279170642135997577692479\dotsc$ \\ \hline
$107$ & $-0.021577355804406921497164663554\dotsc$ \\ \hline
$109$ & $-0.132210409721544731056437843608\dotsc$ \\ \hline
$113$ & $ 0.146243585139370860703027524805\dotsc$ \\ \hline
$127$ & $ 0.040394010134348489183018827920\dotsc$ \\ \hline
$131$ & $ 0.298449023408991783321689774462\dotsc$ \\ \hline
$137$ & $ 0.047930965876723141064902285324\dotsc$ \\ \hline
$139$ & $-0.148432025631579335054711906883\dotsc$ \\ \hline
$149$ & $ 0.074744597741851346493266689740\dotsc$ \\ \hline
$151$ & $ 0.124630596839923975050779380013\dotsc$ \\ \hline
$157$ & $-0.342221786939673495594702901369\dotsc$ \\ \hline
$163$ & $-0.075930049357489864149813121665\dotsc$ \\ \hline
$167$ & $-0.240143659440212808338520030893\dotsc$ \\ \hline
$173$ & $ 0.263674051277298658168490819274\dotsc$ \\ \hline
$179$ & $ 0.333190239006673566447646671954\dotsc$ \\ \hline
$181$ & $ 0.073654489425694545929682768408\dotsc$ \\ \hline
$191$ & $ 0.290035630415210093718516692609\dotsc$ \\ \hline
$193$ & $ 0.222651642202729866484149916054\dotsc$ \\ \hline
$197$ & $-0.156566809829687617064016329304\dotsc$ \\ \hline
$199$ & $-0.282430739458251133741875904647\dotsc$ \\ \hline
$211$ & $-0.429168347924233047597587411558\dotsc$ \\ \hline
$223$ & $-0.155603071391591932610289432189\dotsc$ \\ \hline
$227$ & $-0.353136988360127612810701881267\dotsc$ \\ \hline
$229$ & $-0.362422753876730085842820149560\dotsc$ \\ \hline
$233$ & $ 0.433532695766378853227670116542\dotsc$ \\ \hline
$239$ & $ 0.164184459843550091782890956397\dotsc$ \\ \hline
$241$ & $ 0.159253022504160842909844706024\dotsc$ \\ \hline
$251$ & $ 0.146972144654037063050493044384\dotsc$ \\ \hline
$257$ & $-0.136199726917943335811627159604\dotsc$ \\ \hline
$263$ & $-0.074624164045906778442642890380\dotsc$ \\ \hline
$269$ & $ 0.013472786285611176254289149447\dotsc$ \\ \hline
\end{tabular}
}
\scalebox{0.5875}{
\begin{tabular}{|c|r|}
\hline
$q$  & $\kappa(q)$\phantom{01234567890123456} \\ \hline
$271$ & $-0.189411243976062586851108632756\dotsc$ \\ \hline
$277$ & $ 0.263202439115100139126518544356\dotsc$ \\ \hline
$281$ & $ 0.057724833014438081725834568053\dotsc$ \\ \hline
$283$ & $-0.017818517564561211690503522917\dotsc$ \\ \hline
$293$ & $ 0.306747668430650885617781010764\dotsc$ \\ \hline
$307$ & $-0.069456295139958127570388779996\dotsc$ \\ \hline
$311$ & $ 0.198785581009518830204116721399\dotsc$ \\ \hline
$313$ & $-0.058335428782145955572971075926\dotsc$ \\ \hline
$317$ & $-0.306162250315812940131841982805\dotsc$ \\ \hline
$331$ & $-0.221876535391440609759086790484\dotsc$ \\ \hline
$337$ & $-0.137585019298129547983000969948\dotsc$ \\ \hline
$347$ & $ 0.125972980141740521593554243788\dotsc$ \\ \hline
$349$ & $-0.021683748794506169104750643621\dotsc$ \\ \hline
$353$ & $-0.184452367141133812078478630767\dotsc$ \\ \hline
$359$ & $ 0.147038825471371984998505014858\dotsc$ \\ \hline
$367$ & $-0.069313097303205530130205666424\dotsc$ \\ \hline
$373$ & $ 0.076596129588533742221629674966\dotsc$ \\ \hline
$379$ & $-0.388797187453176730278714669405\dotsc$ \\ \hline
$383$ & $-0.230682223036936316348691095841\dotsc$ \\ \hline
$389$ & $-0.223032768871024628694255146810\dotsc$ \\ \hline
$397$ & $ 0.000748118855371051701377580116\dotsc$ \\ \hline
$401$ & $ 0.185572827154491322428329890735\dotsc$ \\ \hline
$409$ & $ 0.248052237047556669352046991132\dotsc$ \\ \hline
$419$ & $ 0.157332613388133412852639687710\dotsc$ \\ \hline
$421$ & $-0.174178764178552548930059030768\dotsc$ \\ \hline
$431$ & $ 0.128634340886396804242241451176\dotsc$ \\ \hline
$433$ & $ 0.084269733696337528785456854927\dotsc$ \\ \hline
$439$ & $-0.447877263969389723214762662837\dotsc$ \\ \hline
$443$ & $ 0.427464638596097761454083371858\dotsc$ \\ \hline
$449$ & $-0.139890854597819111185944638526\dotsc$ \\ \hline
$457$ & $-0.245371236271918975132113711427\dotsc$ \\ \hline
$461$ & $ 0.038620450496409517715179339278\dotsc$ \\ \hline
$463$ & $-0.036333074275340449015217357910\dotsc$ \\ \hline
$467$ & $-0.121566043211421853161767135819\dotsc$ \\ \hline
$479$ & $ 0.145182917495206913077999600130\dotsc$ \\ \hline
$487$ & $ 0.140413509949441485073692928488\dotsc$ \\ \hline
$491$ & $ 0.262904594663507382291183569451\dotsc$ \\ \hline
$499$ & $-0.219260694237723093748867207657\dotsc$ \\ \hline
$503$ & $ 0.155264822023478652427167746436\dotsc$ \\ \hline
$509$ & $ 0.396770359727893150134302886340\dotsc$ \\ \hline
$521$ & $-0.407742993248390943508928391886\dotsc$ \\ \hline
$523$ & $ 0.011011396537186197983927501061\dotsc$ \\ \hline
$541$ & $-0.056550025070562764651312824901\dotsc$ \\ \hline
$547$ & $-0.362750306999037956812425242973\dotsc$ \\ \hline
$557$ & $-0.002498502607772693313640828116\dotsc$ \\ \hline
$563$ & $-0.082061002005529751422431640627\dotsc$ \\ \hline
$569$ & $-0.189592641309933980463255439634\dotsc$ \\ \hline
$571$ & $ 0.006762564496705295384412330712\dotsc$ \\ \hline
$577$ & $-0.071366532593572706997338782409\dotsc$ \\ \hline
$587$ & $-0.275331295594215260975351328532\dotsc$ \\ \hline
$593$ & $ 0.047479327097910440102634939653\dotsc$ \\ \hline
$599$ & $-0.042594548496272768170014013524\dotsc$ \\ \hline
$601$ & $-0.027460621240038493801636892036\dotsc$ \\ \hline
$607$ & $-0.183552581670057320622328056234\dotsc$ \\ \hline
$613$ & $-0.174524935685109163930471540817\dotsc$ \\ \hline
$617$ & $-0.208966374151513622926768311323\dotsc$ \\ \hline
\end{tabular}
}
\scalebox{0.5875}{
\begin{tabular}{|c|r|}
\hline
$q$  & $\kappa(q)$\phantom{01234567890123456} \\ \hline
$619$ & $-0.259629282624073772565452173821\dotsc$ \\ \hline
$631$ & $ 0.190116056570891865159528180383\dotsc$ \\ \hline
$641$ & $ 0.325939155356921431799220348265\dotsc$ \\ \hline
$643$ & $ 0.019886733189759427898093298165\dotsc$ \\ \hline
$647$ & $-0.137685395879823224137370997201\dotsc$ \\ \hline
$653$ & $ 0.273543758131422511493206807416\dotsc$ \\ \hline
$659$ & $ 0.383633957411155840228259251017\dotsc$ \\ \hline
$661$ & $-0.200105122199871976054565725904\dotsc$ \\ \hline
$673$ & $ 0.019385166425234903510111836048\dotsc$ \\ \hline
$677$ & $-0.094849834462250448508720537392\dotsc$ \\ \hline
$683$ & $ 0.119530228514915593721765831729\dotsc$ \\ \hline
$691$ & $-0.290881331274701683551158981060\dotsc$ \\ \hline
$701$ & $-0.087188812276913764965555516461\dotsc$ \\ \hline
$709$ & $ 0.064919413083159646396701312314\dotsc$ \\ \hline
$719$ & $ 0.181990826991575152578239392835\dotsc$ \\ \hline
$727$ & $ 0.046655812079948584302413751120\dotsc$ \\ \hline
$733$ & $-0.016642504300237507705799731068\dotsc$ \\ \hline
$739$ & $ 0.115713008144331724148361011007\dotsc$ \\ \hline
$743$ & $ 0.007343718971347873496945247142\dotsc$ \\ \hline
$751$ & $ 0.030925590148469386161588553586\dotsc$ \\ \hline
$757$ & $-0.052022312742514114454595810552\dotsc$ \\ \hline
$761$ & $ 0.444249104121725236751924267608\dotsc$ \\ \hline
$769$ & $-0.159282715034694377684631604887\dotsc$ \\ \hline
$773$ & $ 0.096632555915274919995984608182\dotsc$ \\ \hline
$787$ & $-0.042575685211717990064154229651\dotsc$ \\ \hline
$797$ & $ 0.058819660486877031151706605185\dotsc$ \\ \hline
$809$ & $ 0.315145023051549625660887462917\dotsc$ \\ \hline
$811$ & $-0.235354497663854842492900197703\dotsc$ \\ \hline
$821$ & $ 0.096861803511194817694770007831\dotsc$ \\ \hline
$823$ & $-0.034388994619879620513739570904\dotsc$ \\ \hline
$827$ & $-0.173652349209065218039268176978\dotsc$ \\ \hline
$829$ & $-0.195235656323517329297617035637\dotsc$ \\ \hline
$839$ & $-0.115638368729781225060211708094\dotsc$ \\ \hline
$853$ & $ 0.065511643256111743290917202609\dotsc$ \\ \hline
$857$ & $ 0.061215218394828596351971508897\dotsc$ \\ \hline
$859$ & $-0.159681951250709570182943546895\dotsc$ \\ \hline
$863$ & $ 0.070158552601870909036548126563\dotsc$ \\ \hline
$877$ & $-0.363272663845270843179334892982\dotsc$ \\ \hline
$881$ & $ 0.151983175755910192350668847190\dotsc$ \\ \hline
$883$ & $ 0.158424230922396137325711547541\dotsc$ \\ \hline
$887$ & $-0.034198720345486031699283336100\dotsc$ \\ \hline
$907$ & $-0.145266048805493201042546684519\dotsc$ \\ \hline
$911$ & $ 0.058290635536702959098297788733\dotsc$ \\ \hline
$919$ & $ 0.046137246025493658435384125040\dotsc$ \\ \hline
$929$ & $ 0.063037917470448604187531863614\dotsc$ \\ \hline
$937$ & $-0.072912862331977817247297121763\dotsc$ \\ \hline
$941$ & $ 0.113504707610325586382805828841\dotsc$ \\ \hline
$947$ & $ 0.268372969161553904131938250228\dotsc$ \\ \hline
$953$ & $ 0.136016292884486881751392072599\dotsc$ \\ \hline
$967$ & $-0.359730379227549872320819915874\dotsc$ \\ \hline
$971$ & $ 0.094543904158954407172899643305\dotsc$ \\ \hline
$977$ & $-0.211593426266351284571256570972\dotsc$ \\ \hline
$983$ & $-0.300766180700610757960740884883\dotsc$ \\ \hline
$991$ & $-0.126006012080074567946120981737\dotsc$ \\ \hline
$997$ & $-0.152411646339118425456419731685\dotsc$ \\ \hline
\phantom{} & \phantom{}  \\ 
\hline
\end{tabular}
} 
\caption{\label{table2}
Values of $\kappa(q)$ (truncated) for every odd prime up to $1000$.
}
\end{table}

\mbox{}
\vfill\eject
\mbox{}

\ifthenelse{\boolean{plots_included}}
{
\begin{figure}[H]
\begin{minipage}{0.48\textwidth} 
\includegraphics[scale=0.32,angle=0]{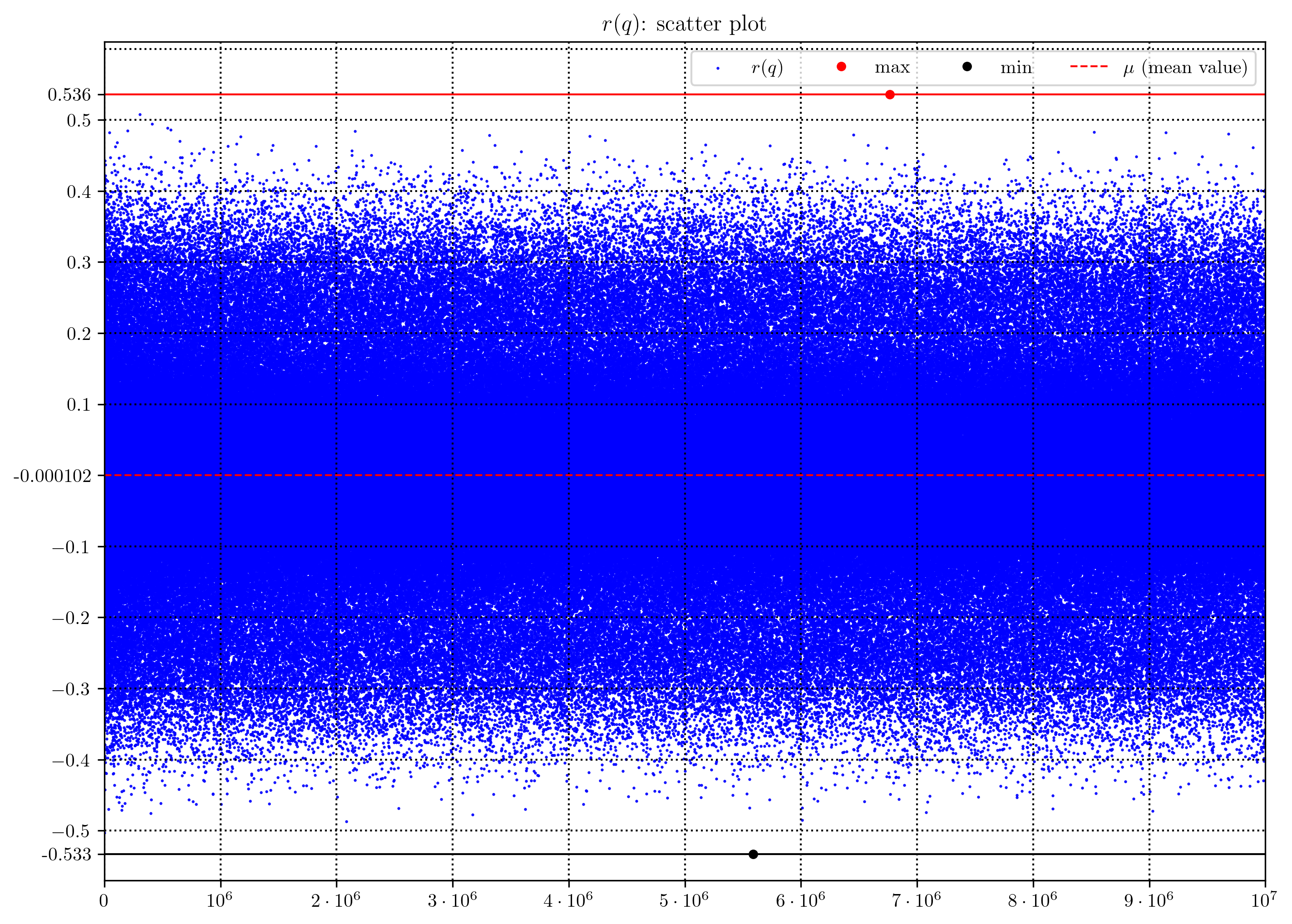}  
\end{minipage} 
\begin{minipage}{0.48\textwidth} 
\includegraphics[scale=0.32,angle=0]{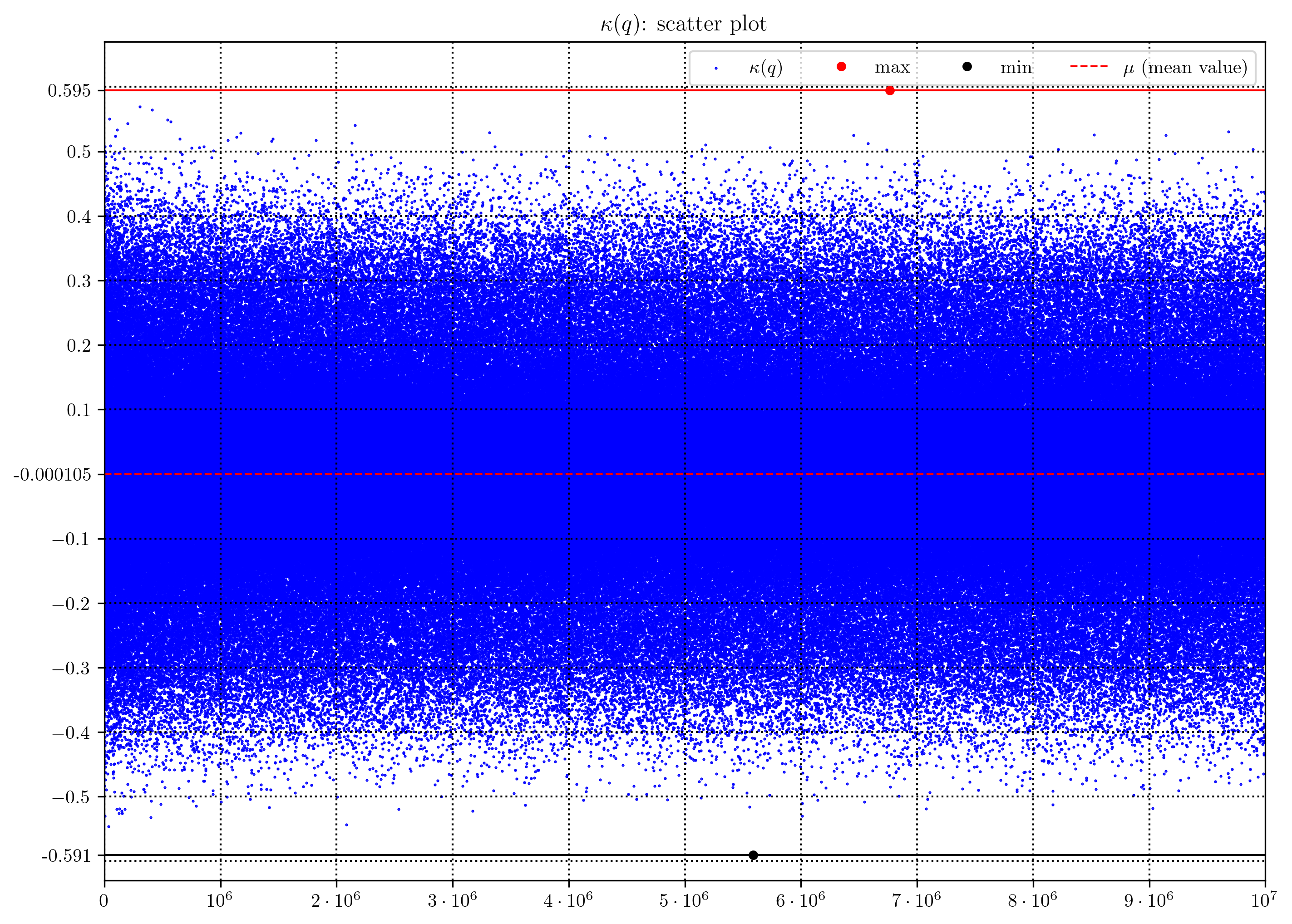}
\end{minipage} 
\caption{{\small On the left: the values of $r(q)$;
on the right: the values of $\kappa(q)$, $q$ prime, $3\le q\le  \bound$.
The red dashed lines represent the mean values.}}
\label{fig1}
\end{figure}

\begin{figure}[H]
\hskip0.5cm
\begin{minipage}{0.48\textwidth} 
\includegraphics[scale=0.5,angle=0]{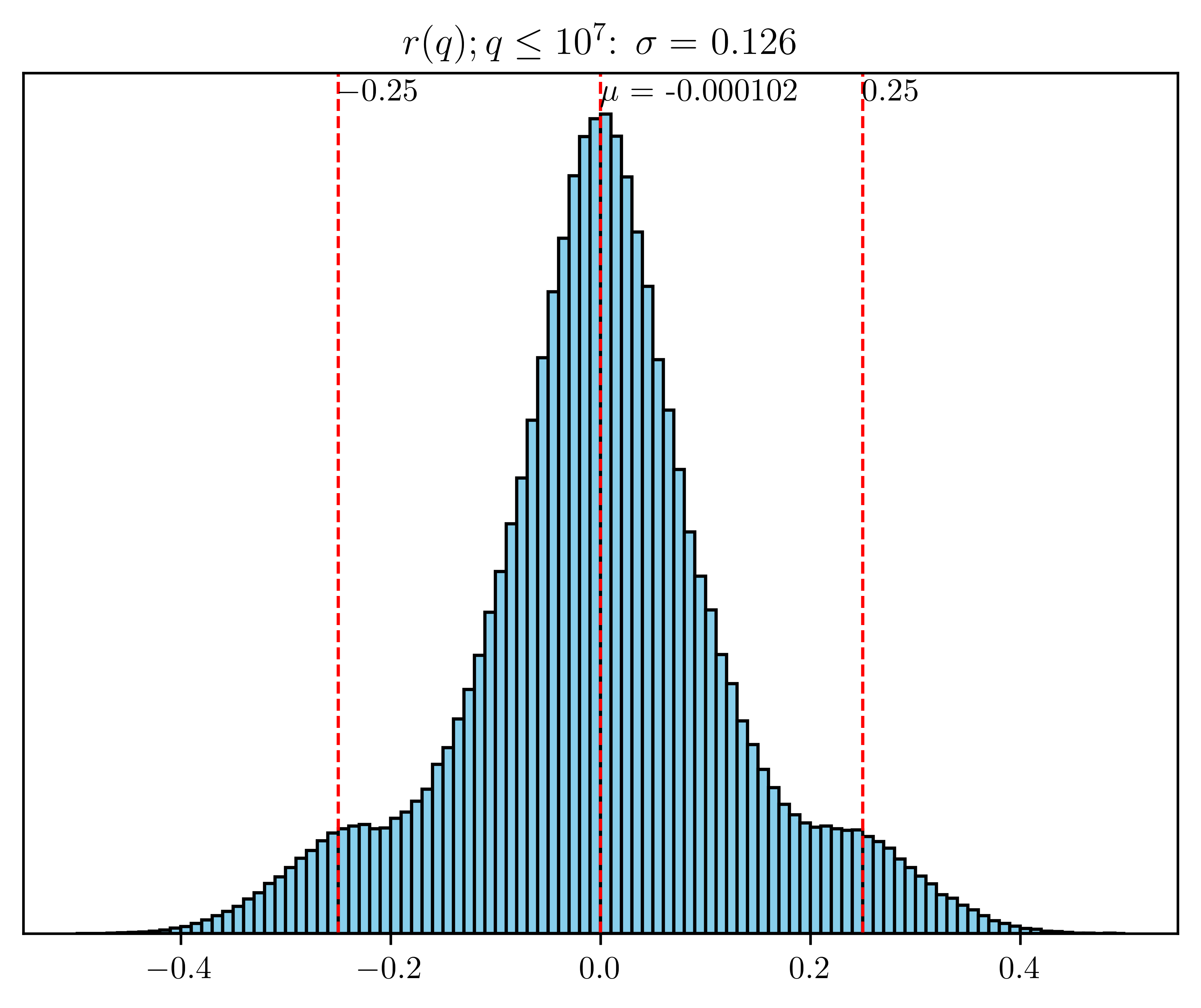} 
\end{minipage} 
\begin{minipage}{0.48\textwidth} 
\includegraphics[scale=0.5,angle=0]{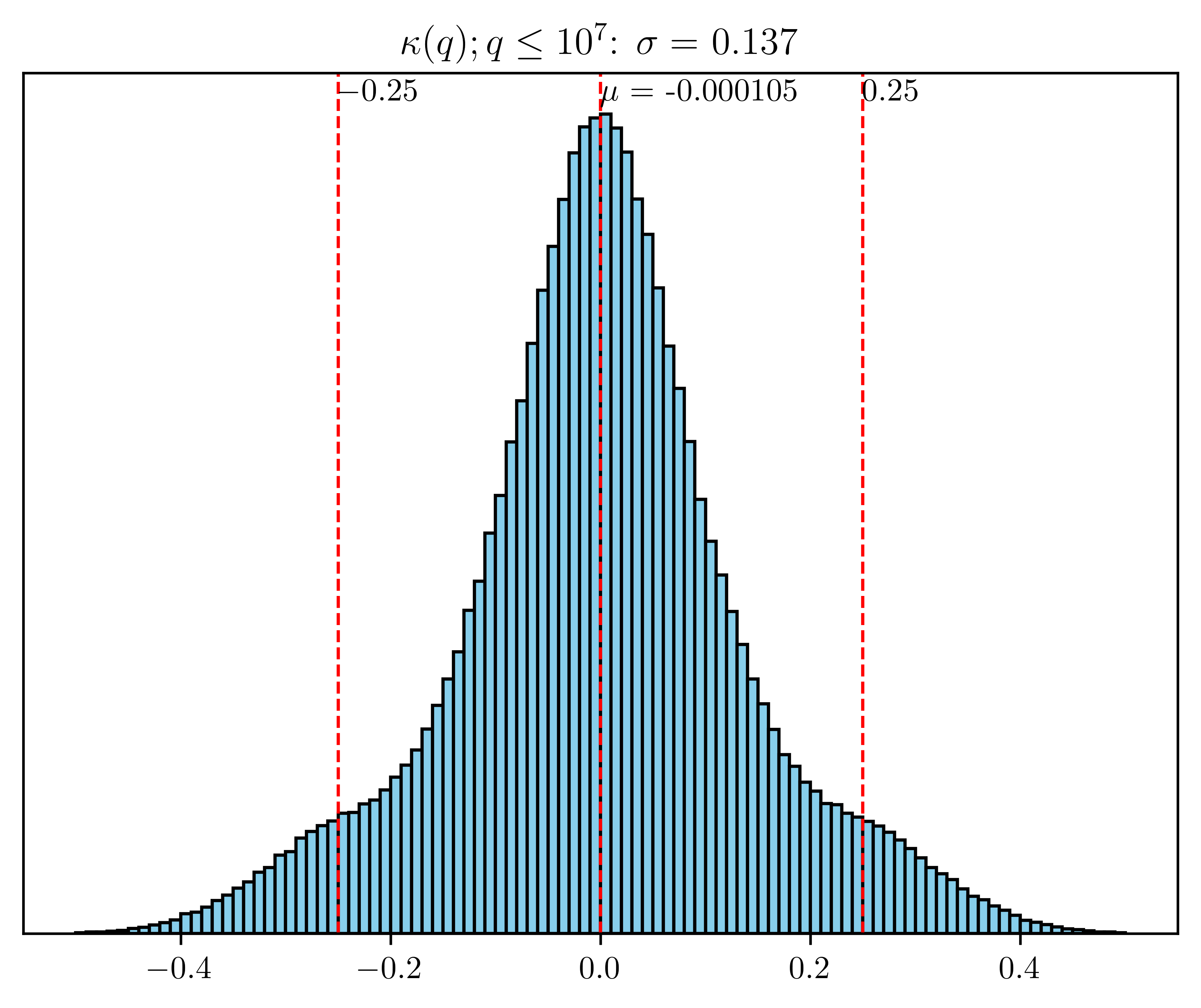}
\end{minipage} 
\caption{{\small On the left:  the histogram of $r(q)$;
on the right: the histogram of $\kappa(q)$, $q$ prime, $3\le q\le  \bound$.
The red dashed lines represent the mean values, and $\pm1/4$.}}  
\label{fig2}
\end{figure}

\begin{figure}[H]
\hskip0.5cm
\begin{minipage}{0.48\textwidth} 
\includegraphics[scale=0.5,angle=0]{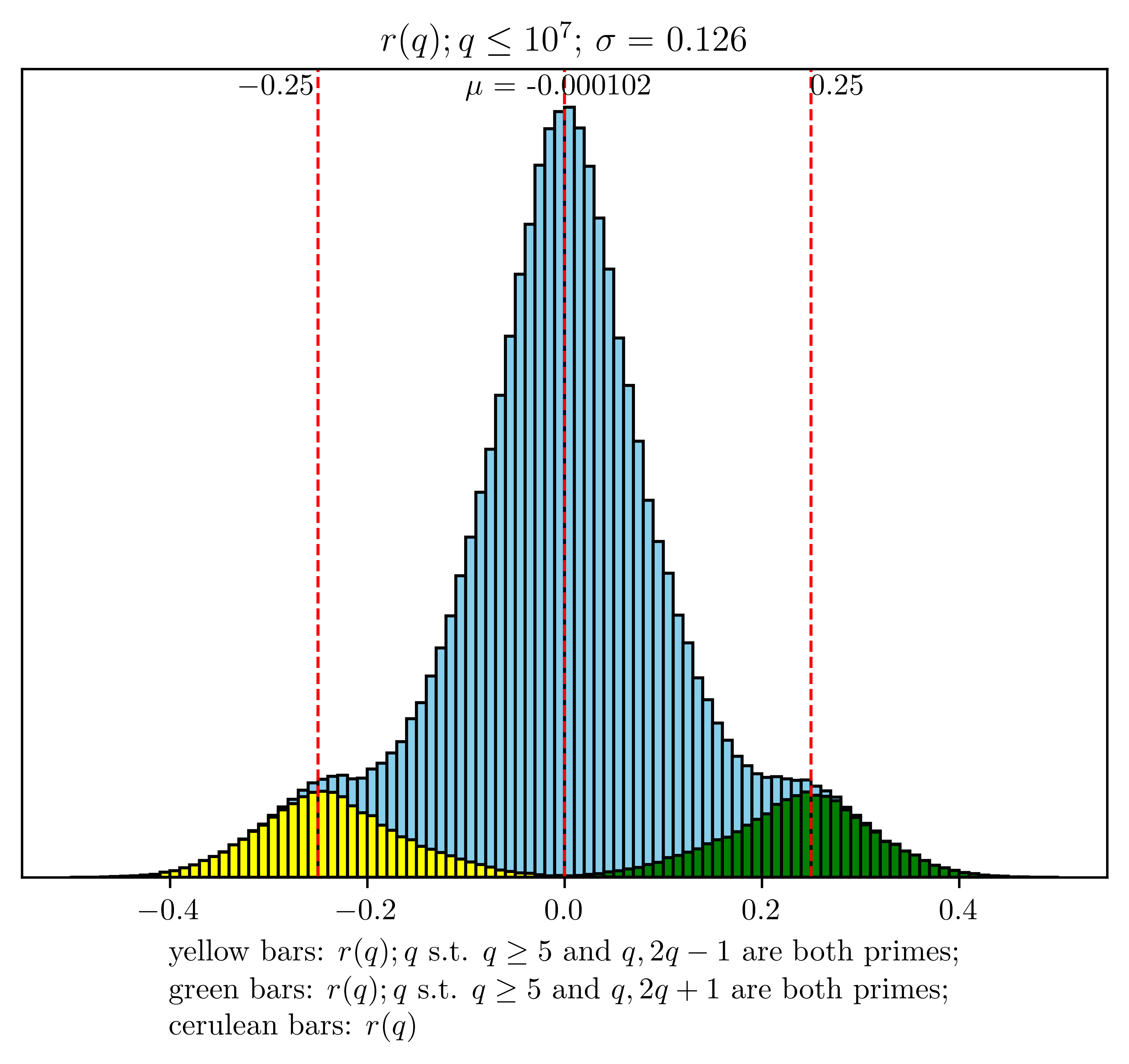}  
\end{minipage} 
\begin{minipage}{0.48\textwidth} 
\includegraphics[scale=0.5,angle=0]{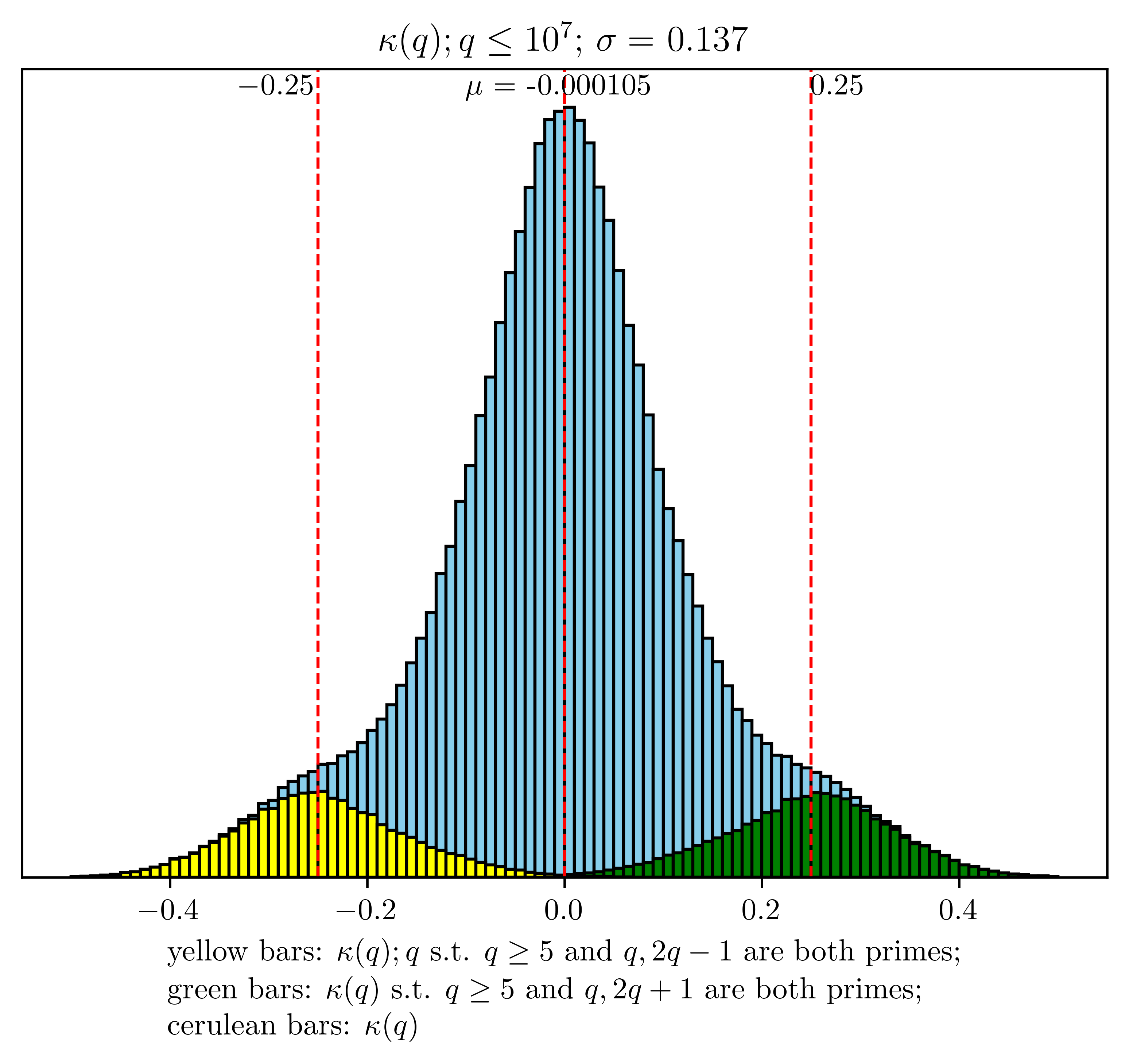}
\end{minipage} 
\caption{The same histograms of Figure \ref{fig3} but the contributions of the primes $q$ such that 
$2q+1$ is prime or $2q-1$ is prime (the ``spikes'') are superimposed.
}
\label{fig3}
\end{figure}  

\begin{figure}[H]
\scalebox{0.9}{
\begin{minipage}{0.48\textwidth} 
\includegraphics[scale=0.5,angle=0]{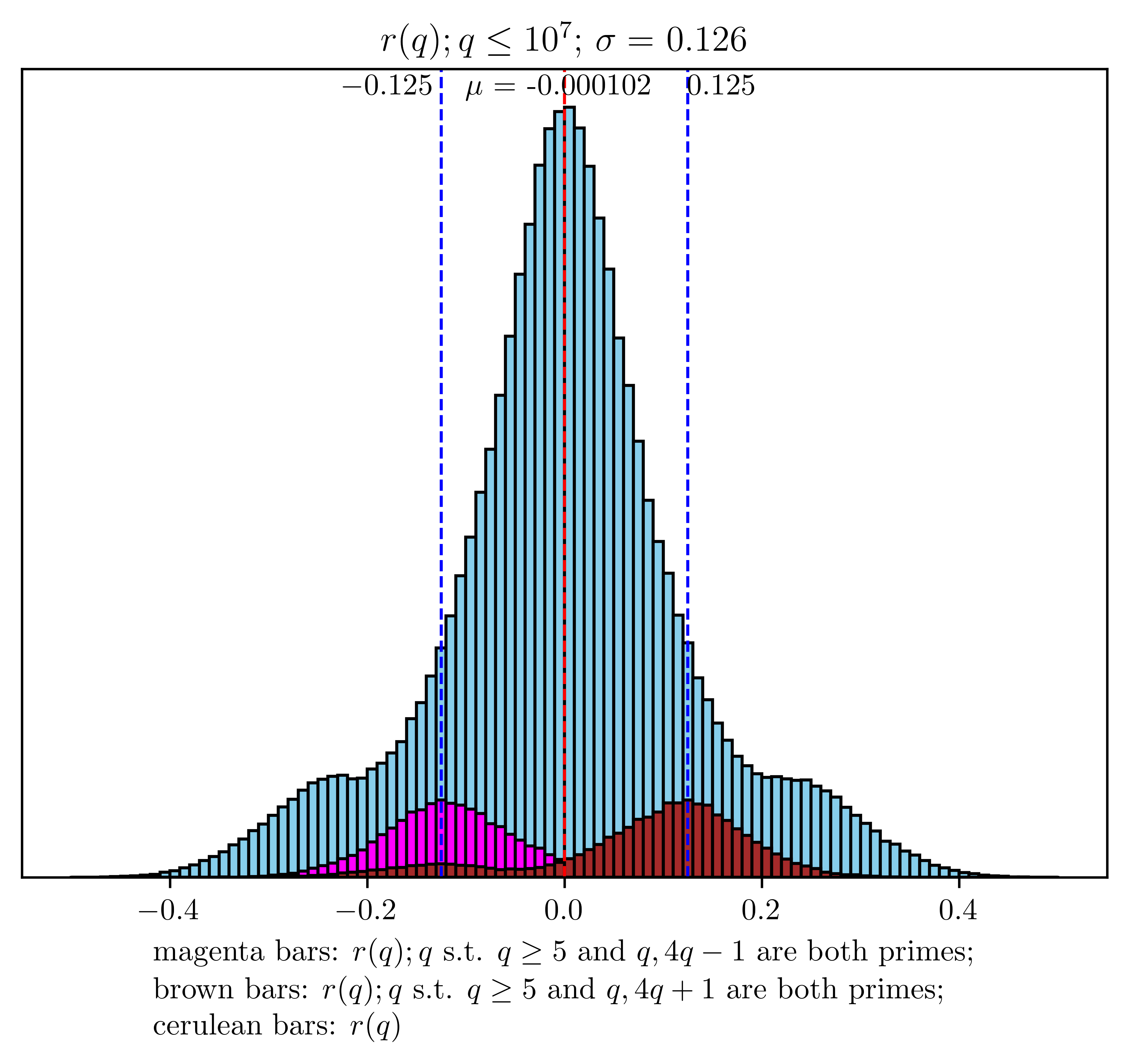}  
\end{minipage} 
\begin{minipage}{0.48\textwidth} 
\includegraphics[scale=0.5,angle=0]{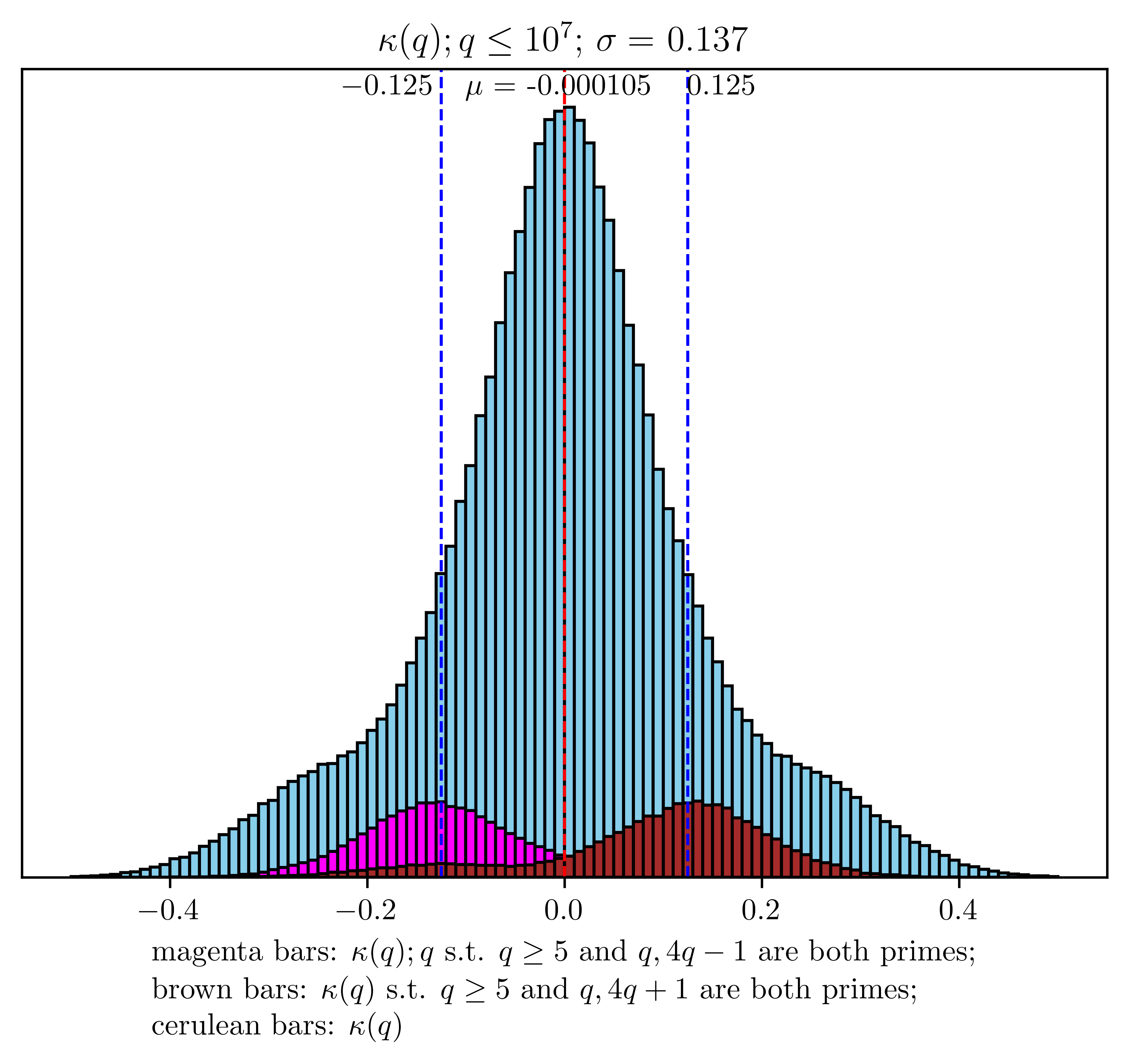}
\end{minipage} 
}
\caption{The same histograms of Figure \ref{fig3} but the contributions of the primes $q$ such that 
$4q+1$ is prime or $4q-1$ is prime (the ``spikes'') are superimposed.
}
\label{fig4}
\end{figure}

\begin{figure}[H]
\scalebox{0.9}{
\begin{minipage}{0.48\textwidth} 
\includegraphics[scale=0.5,angle=0]{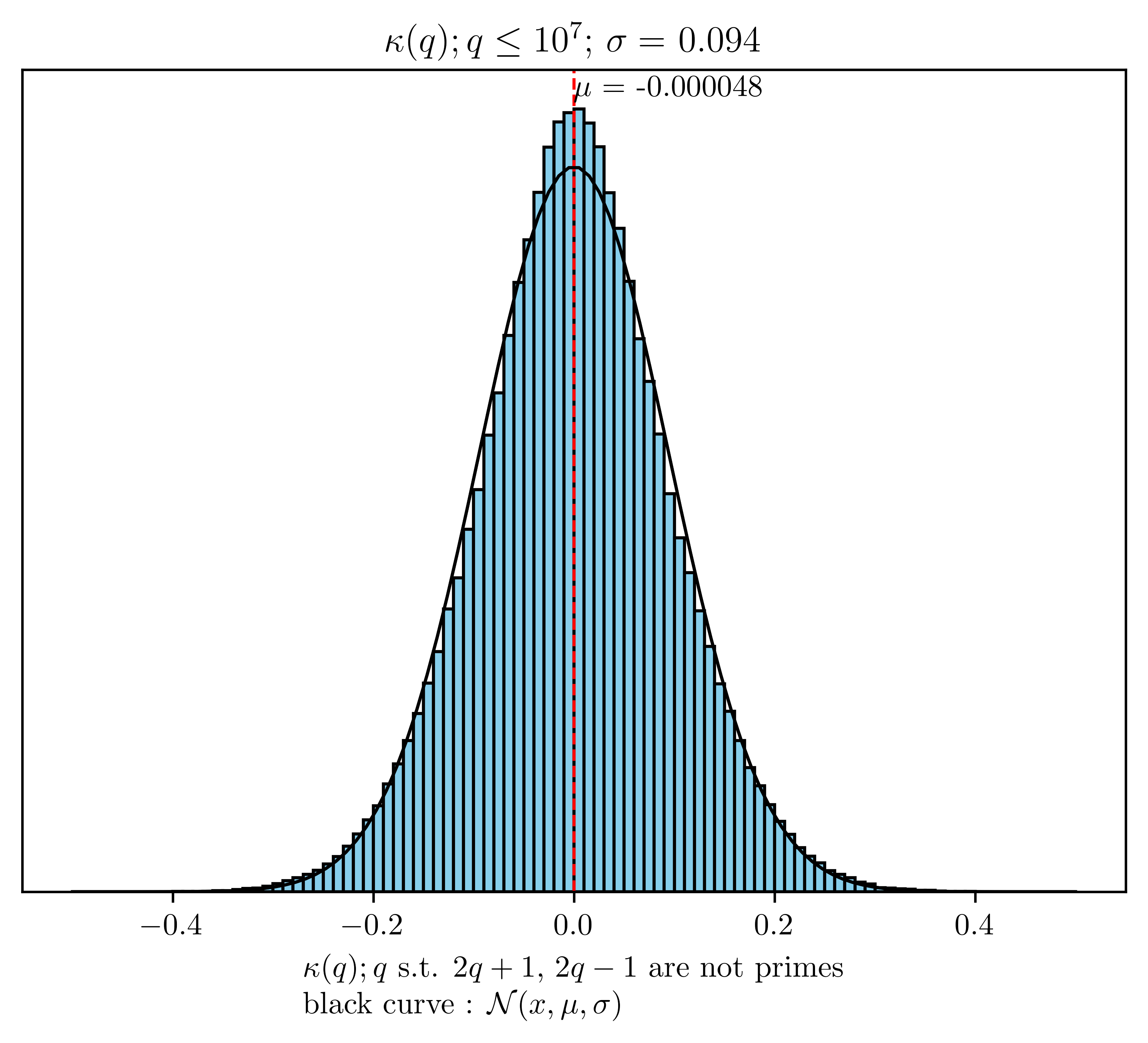}
\end{minipage} 
\begin{minipage}{0.48\textwidth} 
\includegraphics[scale=0.5,angle=0]{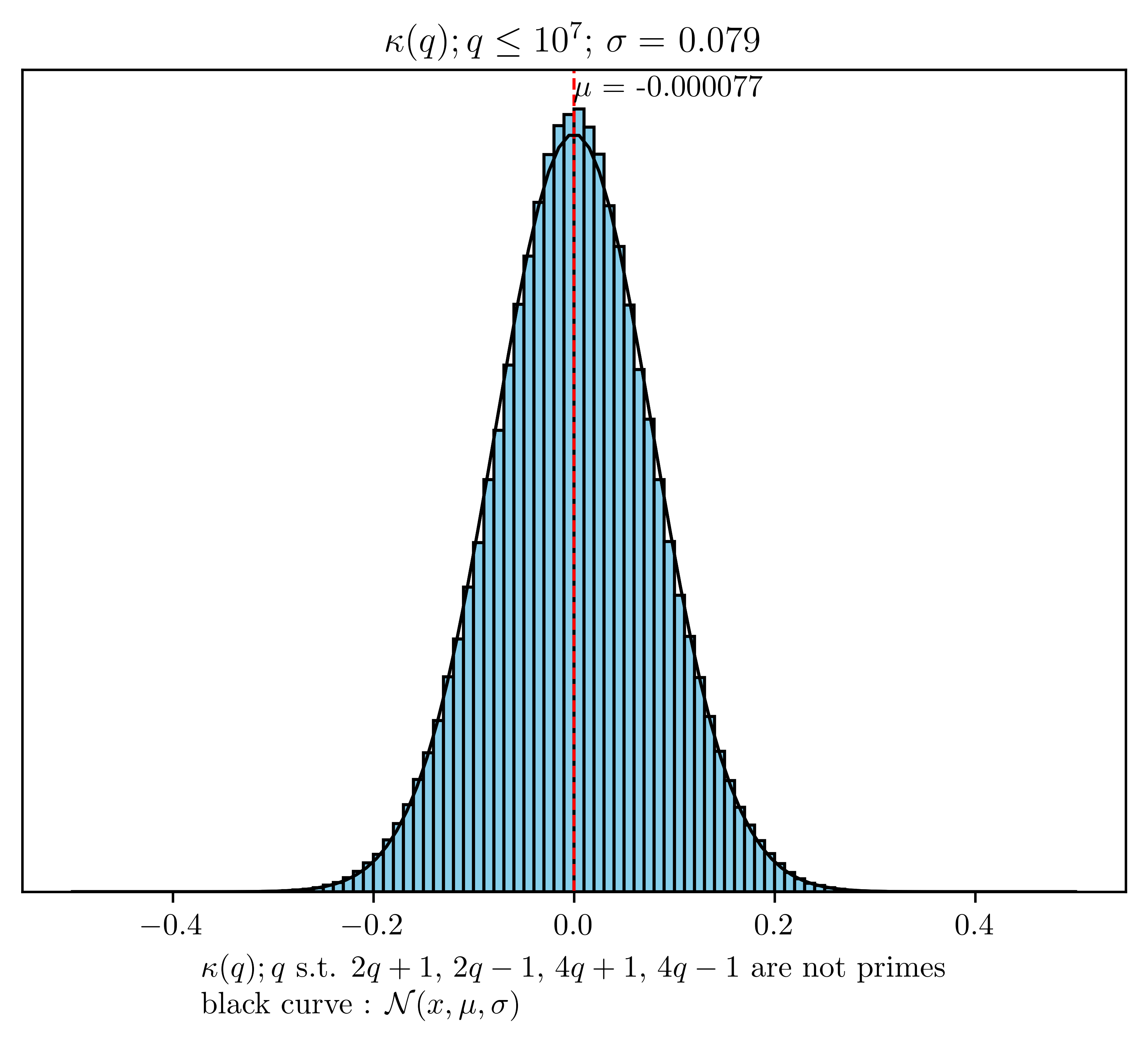}
\end{minipage} 
}
\caption{On the left: the histogram for $\kappa(q)$, $q$ prime, $5\le q\le  \bound$, such that
$2q\pm1$ are composite; on the right: idem with both $2q\pm1$ and $4q\pm1$ that are composite numbers.
The red dashed lines represent the mean values.}
\label{fig5}
\end{figure}

\begin{figure}[H]
\hskip-1cm
\scalebox{0.9}{
\begin{minipage}{0.48\textwidth} 
\includegraphics[scale=0.30675,angle=0]{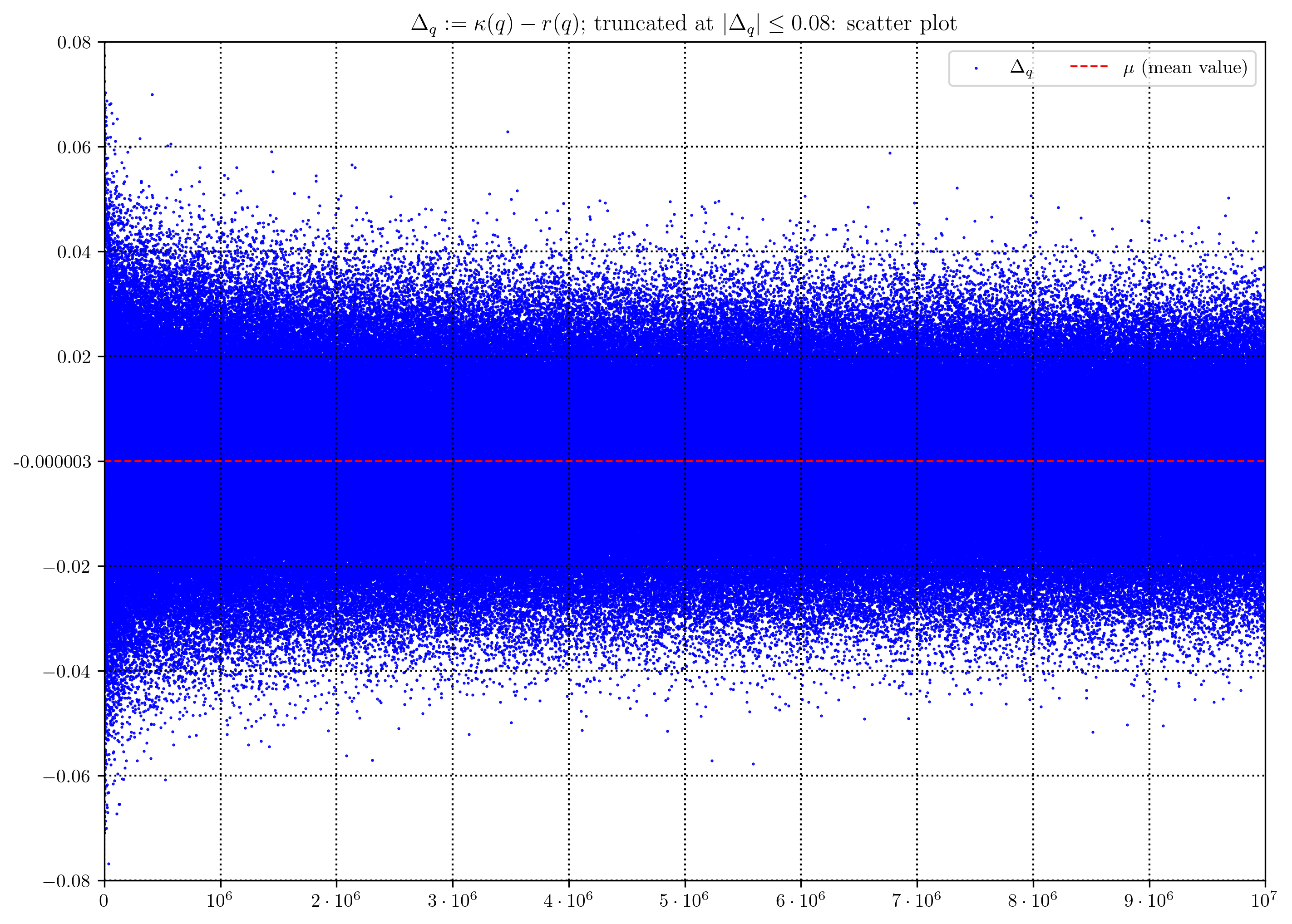}
\end{minipage} 
\hskip1cm
\begin{minipage}{0.48\textwidth} 
\includegraphics[scale=0.2975,angle=0]{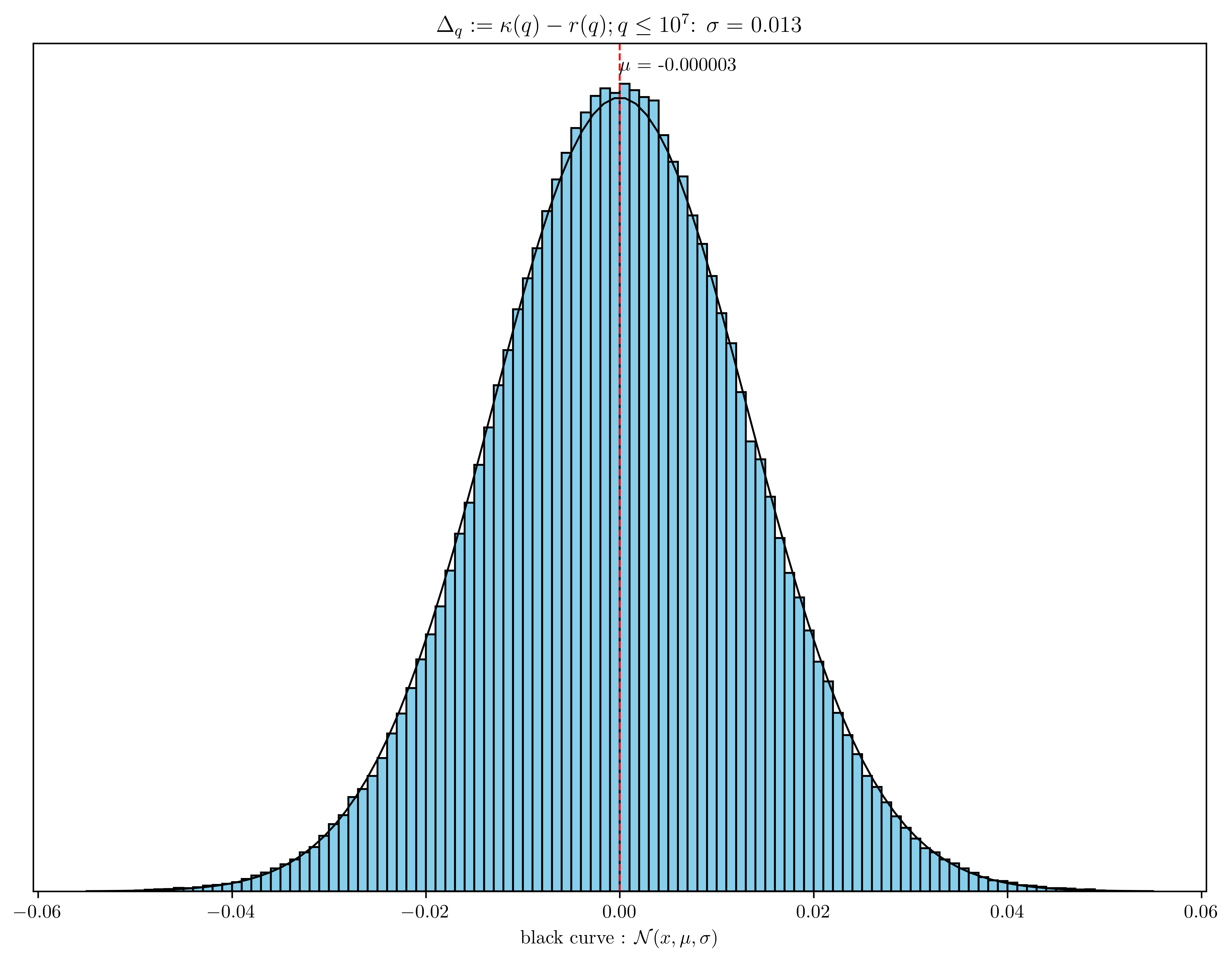}
\end{minipage} 
}
\caption{On the left:  the values of $\Delta_q = \kappa(q) - r(q)$, s.t. $| \Delta_q | \le 0.08$,  
$q$ prime, $3\le q\le \bound$; on the right: the histogram of the same values.
The red dashed lines represent the mean value.}
\label{fig6}
\end{figure}  
\clearpage
}{}

\medskip\noindent Neelam Kandhil  \par\noindent
{\footnotesize 
The University of Hong Kong, Department of Mathematics, Pokfulam, Hong Kong.
\hfil\break
e-mail: {\tt neelamkandhil091@gmail.com}}

\medskip\noindent Alessandro Languasco \par\noindent
{\footnotesize Universit\`a di Padova, Dipartimento di Matematica ``Tullio Levi-Civita'', via Trieste 63, 35121 Padova, Italy.\hfil\break
e-mail: {\tt alessandro.languasco@unipd.it}}

\medskip\noindent Pieter Moree  \par\noindent
{\footnotesize Max-Planck-Institut f\"ur Mathematik,
Vivatsgasse 7, D-53111 Bonn, Germany.\hfil\break
e-mail: {\tt moree@mpim-bonn.mpg.de}}

\medskip\noindent Sumaia Saad Eddin \par\noindent
{\footnotesize Johann Radon Institute for Computational and Applied Mathematics,\\
Austrian Academy of Sciences,
Altenbergerstrasse 69, A-4040 Linz, Austria.\hfil\break
e-mail: {\tt sumaia.saad-eddin@ricam.oeaw.ac.at}}

\medskip\noindent Alisa Sedunova \par\noindent
{\footnotesize 
Mathematics Institute, Zeeman Building, University of Warwick, Coventry CV4 7AL, England. \hfil\break
e-mail: {\tt alisa.sedunova@gmail.com}}

\end{document}